\documentclass[a4paper]{amsart}

\pdfoutput=1

\usepackage[utf8]{inputenc}
\usepackage[T1]{fontenc}

\usepackage{amsthm, amssymb, amsmath, amsfonts}
\usepackage{upgreek, mathrsfs}

\usepackage{graphicx}

\usepackage[ocgcolorlinks=true, colorlinks=true, pdfstartview=FitV, linkcolor=blue, citecolor=blue, urlcolor=blue, pagebackref=false]{hyperref}

\usepackage{microtype}

\newtheorem{thm}{Theorem}[section]
\newtheorem{prop}[thm]{Proposition}
\newtheorem{lem}[thm]{Lemma}

\theoremstyle{remark}
\newtheorem{rem}[thm]{Remark}

\renewcommand{\le}{\leqslant}
\renewcommand{\ge}{\geqslant}
\renewcommand{\subset}{\subseteq}
\newcommand{\mcl}{\mathcal}
\newcommand{\msf}{\mathsf}
\newcommand{\E}{\mathbb{E}}

\newcommand{\EEt}{\mathbf{E}^{\tau_N}}
\newcommand{\EEtu}{\mathbf{E}^{\tau_N}_{\mfk{u}}} 
\newcommand{\VVtu}{\mathbf{V}^{\tau_N}_{\mfk{u}}} 
\newcommand{\CCtu}{\mathbf{C}^{\tau_N}_{\mfk{u}}} 
\newcommand{\N}{\mathbb{N}}

\newcommand{\Ll}{\left}
\newcommand{\Rr}{\right}
\newcommand{\1}{\mathbf{1}}
\newcommand{\R}{\mathbb{R}}

\newcommand{\Z}{\mathbb{Z}}
\renewcommand{\P}{\mathbb{P}}
\newcommand{\V}{\mathbb{V}}
\newcommand{\PP}{\mathbf{P}}
\newcommand{\PPt}{\mathbf{P}^{\tau_N}}
\newcommand{\PPtu}{\mathbf{P}^{\tau_N}_{\mfk{u}}}
\newcommand{\ov}{\overline}
\newcommand{\td}{\tilde}
\newcommand{\eps}{\varepsilon}
\renewcommand{\d}{\mathrm{d}}

\newcommand{\mfk}{\mathfrak}
\newcommand{\msft}{\mathsf{t}_N}
\newcommand{\msfd}{\mathsf{d}_N}
\renewcommand{\tau}{\uptau}
\newcommand{\de}{^{(\delta)}}
\newcommand{\fty}{^\infty}
\newcommand{\ssim}{\approx}%{\overset{\text{(2)}}{\sim}}
\renewcommand{\emptyset}{\varnothing}

\newcommand{\EEotu}{\mathbf{E}^{\tau_N^{(0)}}_{\mfk{u}}} 
\newcommand{\PPotu}{\mathbf{P}^{\tau_N^{(0)}}_{\mfk{u}}} 
\newcommand{\EExtu}{\mathbf{E}^{\tau_N^{(x)}}_{\mfk{u}}} 

\newcommand{\ups}{\Upsilon} 
\newcommand{\bvd}{\mcl{V}^{(\delta)}} 
\newcommand{\bwd}{\mcl{W}^{(\delta)}} 
\newcommand{\bzd}{\mcl{Z}^{(\delta)}}

\title{Aging of asymmetric dynamics on the random energy model}
\author{Pierre Mathieu, Jean-Christophe Mourrat}

\address[Pierre Mathieu]{CMI, Universit\'e de Provence, 39 rue Joliot Curie, 13453 Marseille cedex 13, France}

\address[Jean-Christophe Mourrat]{Ecole polytechnique f\'ed\'erale de Lausanne, Institut de math\'ematiques, station 8, 1015 Lausanne, Switzerland}

\begin{document}
\begin{abstract}

We show aging of Glauber-type dynamics on the random energy model, in the sense that we obtain the scaling limits of the clock process and of the age process. The latter encodes the Gibbs weight of the configuration occupied by the dynamics. Both limits are expressed in terms of stable subordinators.

\bigskip


\medskip


\end{abstract}
\maketitle

\tableofcontents

%
%
%
%
%
%%%%%%%%%%%%%%%%%%%%%%%%%%%%%%%%%%%%%%%%%%%%%%%%%%%%%%%%%%%%%%
%%%%%%%%%%%%%%%%%%%%%%%%%%%%%%%%%%%%%%%%%%%%%%%%%%%%%%%%%%%%%%
%
%
%
\section{Introduction}
\label{s:intro}
\setcounter{equation}{0}

%The characteristic feature of glassy systems is that of \emph{aging}: the system reacts to a perturbation over a typical response time which increases with its age. This increase of the response time spans every time scale accessible to experiment, and exhibits remarkable scaling properties \cite{struik}. 

The aim of this article is to study the off-equilibrium behaviour and show aging properties of Glauber-type dynamics on the Random Energy Model (REM).

The REM was introduced in \cite{derr1,derr2} as a simple model of a spin glass. It is a mean-field model living on the hypercube $V_N=\{0,1\}^N$, and the energies $(E_x)$ associated to each configuration $x \in V_N$ are taken to be i.i.d.\ random variables. We assume that they are distributed as the positive part of a standard Gaussian random variable. The Gibbs weights of the model are thus of the form 
\begin{equation}
\label{deftauNx}
\tau_{N,x} = e^{\beta \sqrt{N} E_x},
\end{equation}
where $\beta > 0$ is the inverse temperature. 

Unlike other spin glass models where the energies display some complicated correlation, as for instance the SK model, the statics of the REM may be described in detail using results on the extremes of families of independent Gaussian random variables \cite{olipic,gamapi}. 
Roughly speaking, the energy landscape of the REM may be thought of as consisting of a bulk of configurations of 
energies of order $1$, with a few scattered deep wells where $\sqrt{N} E_x$ is of order $cN$ for some constant~$c$. These deep wells carry most of the weight of the Gibbs measure (for a constant $c$ which actually depends on $\beta$).  
%Most deep wells are  separated from each other. 

The simple structure of the energies in the REM also made it possible to study its relaxation to equilibrium under a Metroplis 
or Glauber dynamics. Asymptotics of the spectral gap are known \cite{fikp}. We refer to \cite{mp} for further results on thermalization times.  

It is believed that, for a large class of reversible dynamics, the off-equilibrium dynamics observed before thermalization  
should be described using the language of aging. A crude picture of the phenomenon of aging goes as follows: on well-chosen time and depth scales, the dynamics 
should spend most of its time in a few deep wells. Before it reaches equilibrium, one may think that the dynamics is transient: it finds a deep well, spends some amount of time around it and then moves through the bulk to a second well, and so on, without ever returning to a previously visited well. In such a scenario, at a given time $t$, the dynamics is with high probability in a deep well whose depth is of order $t$. Thus one can evaluate the ``age'' of the dynamics from the value of the energy of the current configuration. 

One way to quantify this aging phenomenon is by computing correlation functions. One may for instance try to derive asymptotics of the probability that the dynamics stays in a small region of its state space between times $t$ and $t+s$. From the discussion above, it follows that a reasonable choice is to take $s$ proportional to $t$. 
As we expect the main contribution to such a correlation function to arise from deep wells, and since the statistics of extremes of exponentials of Gaussian fields are given by Poisson point processes with polynomial tails, it is natural to believe that asymptotics of correlation functions should be related to stable subordinators. More precisely, back-of-the-envelope computations show that correlation functions should follow the arc-sine law. 

More recently, a different and more ambitious description of aging has been proposed that, instead of correlation functions, focuses on deriving the scaling limit of what we call the \emph{age process}. This process gives, at any time, the value of the Gibbs weight of the configuration occupied by the dynamics.
Once again, it is easy to guess how these should be related to stable subordinators. In the transient regime we are considering now, the scaling limit of the age process should be universal and given by the jumps of a stable subordinator computed at the inverse of a (different and correlated) subordinator, as in \cite{fm_aging}.  

Although this description of the dynamics may seem simple% (or even naive)
, proving that it is correct turned out to be a mathematically difficult 
problem. As far as the REM is concerned, so far, rigorous results had only been obtained for one specific dynamics called the \emph{random hopping times} dynamics. It is the aim of the present paper to go beyond this limitation.

The random hopping times dynamics (RHT) was introduced in \cite{ma} in the context of the discussion on thermalization times for spin glasses. In the RHT dynamics, the process stays at a given configuration an exponentially distributed time whose parameter is the inverse of its Gibbs weight, and then jumps to any neighbouring configuration with uniform probability. 
In other words, in the RHT dynamics, one spin flips at a time, and the spin that flips is chosen uniformly at random among the $N$ possibilities, independently of the past, independently of the current configuration and independently of the energy landscape. This feature implies that the trajectory (i.e.\ the sequence of successively visited configurations) follows the same law as the trajectory of a simple random walk on $V_N$: 
the only way the energies enter in the dynamics is through the holding times. In other words, the RHT dynamics is a time-changed simple random walk on $V_N$, and this simplifies the mathematical analysis a lot. The time change itself is usually called the \emph{clock process}. To be coherent with our description of aging, one conjectures that this clock process converges to a stable subordinator. 

Rigorous results on the RHT dynamics for the REM focused on proving the convergence of correlation functions to the arc-sine law, see \cite{bbg1,bbg2,arcsine,ga2}. The main step in this approach is to prove that correlation functions satisfy an approximate renewal equation. 

The aim of our paper is to prove that the aging picture holds for Glauber-type local dynamics in which the law of the jumps depends on the energies of both the current and the target configurations. We obtain the scaling limit of the corresponding clock process (a stable subordinator, as expected) and show that the scaling limit of the age process is the same as in \cite{fm_aging}. 

When moving from RHT dynamics to Glauber dynamics, new difficulties appear. Some of these are due to the fact that the potential theory of a symmetric Markov chain on the hypercube is more involved in the generic case than in the simple symmetric case. %Rather than renewal theory, our approach relies on properties of the process of the \emph{environment viewed by the particle}. 
Also, for the RHT dynamics, the distribution of the energy landscape around (but excluding), say, the first deep trap visited, is simply the Gibbs measure. This is no longer the case for a Glauber dynamics, and the crux is to understand this distribution. 

\medskip 

Let us summarize our strategy of proof in a few words. We view the dynamics as a time-change of another process, which is the original dynamics accelerated by a factor $\tau_{N,x}$ at every site $x \in V_N$. Although this new process is no longer a simple random walk as for the RHT dynamics, it retains some of its properties, as for instance the reversibility with respect to the uniform measure over $V_N$. We prove a law of large numbers for the number of sites discovered by this walk: on appropriate time scales, we show that this number grows asymptotically linearly with time, with a proportionality constant depending neither on the trajectory, nor on the environment. 

Our second major step is to identify how the environment around a deep trap found by the dynamics is distributed. Roughly speaking, we follow an idea of \cite{scaling} and show that asymptotically, this distribution has an explicit density with respect to the product measure, given in terms of a Green function. 
Care must be taken in order to make this statement precise, since the underlying graph changes with~$N$. As opposed to \cite{scaling}, we must thus resort to a \emph{finitary} version of asymptotic convergence, in the sense that we only write statements for finite $N$ with explicit bounds, instead of statements about limit distributions. 

In order to complete the description of the distribution of the environment around a deep trap, we must then study the asymptotic behaviour of the Green function under the product measure. This can be performed via explicit computations, provided that the dynamics we study is not too predictable, see assumption~\eqref{anconstraint} and the related discussion. 

From the law of large numbers and the knowledge of the Green function at a deep trap, we finally derive the joint asymptotic behaviour of all the relevant quantities associated with the visits of traps: the time until the $n^\text{th}$ deep trap is met, the depth of the trap, and the total time spent on it. The aging picture then follows by noticing that the time spent outside of deep traps is negligible.

%The main steps of the proofs often consist in describing the energy landscape around the current configuration at some appropriately chosen (random) time. In the RHT, the distribution of the energy landscape around a configuration that is visited by the dynamics for the first time (e.g.\ the distribution of the energy of the $n$-th configuration visited for a given $n$) does not depend on the past of the process. This is not true for a Glauber dynamics. For instance, the environment around the first deep trap found by the dynamics is not distributed according to the Gibbs measure, as it would be in the RHT case. Following an idea of \cite{scaling}, we analyse this distribution in part (*  ... *), and prove that it is a size-bias of the Gibbs distribution with a weight given by a Green function. 

\medskip

We close this introduction by a short discussion on related issues that are not directly addressed here. 

There exists another spin-glass toy model on which a rather complete description of the dynamics can be easily obtained, namely the REM on the complete graph. In the RHT dynamics on the complete graph, when jumping, the dynamics chooses to go to any of the $2^N$ possible configurations with equal probability. Aging results (in the sense of convergence of correlation functions to the arc-sine law) have been obtained in this case \cite{bou,boudean,bf,bc}. For this model on the complete graph, it turns out to be also possible to consider a certain class of dynamics called the $a$-symmetric trap dynamics, and prove aging for these with renewal techniques \cite{ga1}. Observe however that the $a$-symmetric dynamics on the complete graph (and only on the complete graph!)\ has a very special structure: 
at a jump time, it chooses where to go according to a law that depends on the energies (which is reasonable) but that does not depend on the current configuration (which is less reasonable). As a consequence, the clock process is a subordinator already in finite volume, a fact that helps to prove it converges to a stable subordinator. 

Limit theorems for the RHT dynamics on the $p$-spin model are discussed in \cite{bbc,bovgay}.

The aging picture we described above has also been investigated on models that are not related to spin glass theory. These are often referred to in the literature as \emph{trap models}: one starts with a discrete-time Markov chain on a graph, which may be finite or infinite, and considers its time change through an additive functional as in the RHT dynamics, looking for scaling limits of either the correlation functions or the age process. We refer to \cite{arcsine} for results in this direction.

So far we only discussed dynamics in the transient regime long before thermalization occurs. For dynamical spin glasses, as for more general trap models on a growing sequence of graphs, it makes sense to look for scaling limits on the ergodic time scale as well. The first results in that direction dealt with the RHT dynamics on the complete graph \cite{fm_K}. The scaling limits that appear are known as $K$-processes. These interpolate between the transient regime, for small times, and the usual thermalization regime. It was later proved that $K$-processes describe the RHT dynamics in the REM \cite{fonlim} or the $a$-symmetric dynamics on the complete graph \cite{bfggm}. More recently, \cite{jalate} showed how $K$-processes arise as a universal limit for trap models on growing sequences of graphs. 

On the other side, it is also possible to derive aging results for dynamics of spin glass models on very short time scales. This is often referred to as \emph{extreme aging}, see \cite{begu,bogasv}.

%We did our best to make the paper as self-contained as possible. %Classical results are recalled in the first sections of the paper. 
%We %also 
%divided the proof into several steps, with the hope that it will make its structure more transparent. 

\medskip

The paper is organized as follows. In Section~\ref{s:defs}, we introduce basic definitions and notations. We then provide various quantitative estimates on the mixing properties of the accelerated walk in Section~\ref{s:mixing}. Section~\ref{s:explore} introduces the notion of the exploration process. The law of large numbers for the number of sites discovered by the walk is then achieved in Section~\ref{s:lln}. Next, Section~\ref{s:scales} defines the relevant scales of our problem, in particular that of the deep traps. In Section~\ref{s:deep}, we obtain the joint scaling limit of the time it takes to find a deep trap, and of the depth of the deep trap. After recalling elementary properties of the process of the environment viewed by the particle in Section~\ref{s:envviewed}, we tackle the problem of controlling the distribution of the environment around the first deep trap met in Section~\ref{s:envaround}. Section~\ref{s:green} is devoted to the asymptotic analysis of the Green function under the product measure. Combining the two previous sections, we obtain the scaling limit of the Green function at the first deep trap met in Section~\ref{s:greenatdeep}. We then derive the joint convergence of all the relevant quantities associated to the visits of the traps in Section~\ref{s:manytraps}. We conclude by deducing the convergence of the clock process and of the age process, in Sections~\ref{s:convclock} and \ref{s:convage} respectively.

\medskip

\noindent \textbf{Acknowledgements.} We wish to thank G.\ Ben Arous for insightful discussions at an early stage of this paper, and for encouraging us to complete this work. P.M.\ acknowledges his kind hospitality at the Courant Institute.

%
%
%
%
%%%%%%%%%%%%%%%%%%%%%%%%%%%%%%%%%%%%%%%%%%%%%%%%%%%%%%%%%%%%%%
%%%%%%%%%%%%%%%%%%%%%%%%%%%%%%%%%%%%%%%%%%%%%%%%%%%%%%%%%%%%%%
%
%
%
\section{Definitions and notations}
\label{s:defs}
\setcounter{equation}{0}

We view the hypercube $V_N = \{0,1\}^N$ as an additive group, endowed with its usual graph structure (that is, nearest-neighbour for the Hamming distance, that we write $| \cdot |$). For $x,y \in V_N$, we write $x \sim y$ when $|x-y| = 1$. Let $(E_x)_{x \in V_N}$ be i.i.d.\ random variables distributed as the positive part of a standard Gaussian. We write $\P$ for their joint law (with associated expectation $\E$), and $\mcl{N}_+(0,1)$ for the marginal distribution. Recall that $\beta > 0$ is a fixed parameter, and that $\tau_{N,x} = e^{\beta \sqrt{N} E_x}$. We write $\tau_N = (\tau_{N,x})_{x \in V_N}$. 

For $a_N \in \R$, we consider the continuous-time Markov chain $(Z_t)_{t \ge 0}$ whose jump rate from $x$ to a neighbour $y$ is given by
\begin{equation}
\label{jumprate}
\frac{e^{a_N (E_x+E_y)}}{e^{\beta \sqrt{N} E_x}},
\end{equation}
and write $\PPt_z$ for the law of this chain started from $z \in V_N$ (and $\EEt_z$ for the associated expectation). This dynamics, often called \emph{Bouchaud's dynamics}, is reversible for the measure with weights $(\tau_{N,x})_{x \in V_N}$. The case $a_N = 0$ makes the jump rate in \eqref{jumprate} independent of $y$, and thus corresponds to the RHT dynamics. It is physically natural to consider dynamics that are more attracted to sites whose associated Gibbs weight is high, so we will only consider the case $a_N \ge 0$. For technical reasons, it is convenient to assume from now on that actually $a_N \ge 1$. 

If $a_N$ is allowed to grow too fast with $N$, then the dynamics exhibits certain features that we view as undesirable. More precisely, if $a_N \gg \sqrt{\log N}$, then with probability tending to $1$, the walk started from the origin will make its first jump to the neighbouring site of highest energy. This asymptotic complete predictability persists for subsequent jumps, in the sense that for every $k$, one can predict with probability tending to $1$ the sequence of the first $k$ distinct sites that will be visited by the dynamics. In order to avoid such a behaviour, we will assume instead that
\begin{equation}
\label{anconstraint}
a_N \le \ov{a} \sqrt{2 \log N },
\end{equation}
and moreover, that $\ov{a} < 1/20$. It is instructive to figure out the large $N$ behaviour of the extreme values of the jump rates from the origin to each of its neighbours in this case. Up to a multiplicative factor, this amounts to understanding the extreme values of the family $(e^{a_N E_y})_{y \sim 0}$. By the classical theory of extremes for independent Gaussian random variables, when the constraint \eqref{anconstraint} is replaced by an equality, the point process of the rescaled $(e^{a_N E_y})_{y \sim 0}$ converges to the point process with intensity measure $x^{-(1+1/\ov{a})} \, \d x$. In particular, the jump rates remain inhomogeneous even in the large $N$ limit.

Assumption \eqref{anconstraint} will be used in order to prove that when a walk moves at some fixed distance away from a deep trap, it will not manage to find it back rapidly. The parameter $1/20$ is chosen for convenience. The proof below can be optimized to cover higher values of $\ov{a}$, but a more detailed analysis would be required to cover every $\ov{a} < 1$ for instance (see Section~\ref{s:green}).

We rewrite $Z$ as a time-change of a ``nicer'' random walk. Denote by $(X_t)_{t \ge 0}$ the walk speeded up by $\tau_x$ at each $x$, whose jump rates are thus
\begin{equation}
\label{defomeganxy}
\omega_N(x,y)=e^{a_N (E_x+E_y)} \qquad (x \sim y \in V_N).
\end{equation}
Note that these jump rates are symmetric: $\omega_N(x,y) = \omega_N(y,x)$. In other words, the uniform measure is reversible for this process. We will denote the uniform probability measure over $V_N$ by $\mfk{u}_N$, and write $\PPtu = \PPt_{\mfk{u}_N}$ for simplicity. We write $\ov{\P}_N$ for $\P \PPtu$,
\begin{equation}
\label{defomegaN}
\omega_N(x)= e^{a_N E_x}\sum_{z\sim x} e^{a_N E_z},
\end{equation}
and $\ov{\omega}_N=\sup_{x\in V_N}\omega_N(x)$. 
%We observe that $\log \ov{\omega}_N$ is of order $a_N\sqrt N$. 
We also define
\begin{equation}
\label{defphi}
\phi(\lambda) = \E\Ll[e^{\lambda E_x}\Rr] = \frac{1}{2} + \frac{e^{\lambda^2/2}}{\sqrt{2\pi}} \int_{-\lambda}^{+\infty} e^{-u^2/2} \ \d u \le 2 e^{\lambda^2/2}.
\end{equation}

% (If $a_N$ tends to infinity with $N$, then $\phi(a_N) \sim e^{a_N^2/2}$).

%(* bouger ca avec les pptés \eqref{defpropbn}, \eqref{gaussrecall}, \eqref{caractB}, \eqref{e:negligpasprofonds} et les pptés sur les max de gaussiennes, et éventuellement les pptés de convergence vers des stables, dans une section à part ? *)

%(* adapt what is a proposition and what is a lemma. *)

\medskip

\noindent \textbf{General notations.} If $A$ is a set, $|A|$ denotes its cardinality. We write $\N = \{1,2,\ldots\}$, and $a \wedge b = \min(a,b)$. In the spirit of the Bachmann-Landau ``big O'' notations, for $\eps > 0$, we write $\pm \eps$ to denote any real number whose value lies in $[-\eps,\eps]$ (this notation is convenient when we want to make explicit the fact that an estimate holds uniformly in other parameters).

%
%
%
%
%%%%%%%%%%%%%%%%%%%%%%%%%%%%%%%%%%%%%%%%%%%%%%%%%%%%%%%%%%%%%%
%%%%%%%%%%%%%%%%%%%%%%%%%%%%%%%%%%%%%%%%%%%%%%%%%%%%%%%%%%%%%%
%
%
%
\section{Mixing properties}
\label{s:mixing}
\setcounter{equation}{0}

In this section, we provide various precise versions of the idea that the random walk $X$ has good mixing properties. In the first proposition, we show that the walk started from any fixed point in $V_N$ admits strong stationary times of order~$N$. In particular, the walk at the strong stationary time is distributed as $\mfk{u}_N$ (the uniform measure on $V_N$), and the time scale of order $N$ is very small compared to the exponentially growing time scales of our problem. The purpose of Propositions~\ref{p:fuite-sg} and \ref{p:fuite} is slightly different, since we will use them to show that the random walk has very little probability to be at a specific location after some time. Proposition~\ref{p:fuite-sg} says that for almost every environment, it suffices to wait a time $t \gg 1$ for the walk to ``get lost''. Proposition~\ref{p:fuite} requires only $t \gg 1/N$, but the result is averaged over the environment.

The definition of strong stationary times requires that the probability space contains additional independent randomness. Possibly enlarging the probability space, we assume that there exists infinitely many independent random variables distributed uniformly on $[0,1]$ and independent of everything else. Following \cite{ad1}, we say that a random variable $\msf{T} \ge 0$ is a \emph{randomized stopping time} if the event $\{\msf{T} \le t\}$ depends only on $\{X_s,\ s \le t\}$ and on the value of these additional random variables. 

\begin{prop}
\label{p:mixing}
For any $N$ large enough and any $\tau_N$, there exists a randomized stopping time $\msf{T}_N$ with values in $\{N, 2N, \ldots\}$ such that for any $x \in V_N$, under $\PPt_x$, 
\begin{enumerate}
\item
the distribution of $X_{\msf{T}_N}$ is $\mfk{u}_N$,
\item
for any $k \in \N$,
$$
\PPt_x\Ll[\msf{T}_N \ge k N\Rr] = e^{-(k-1)},
$$
\item
$\msf{T}_N$ and $X_{\msf{T}_N}$ are independent random variables.
\end{enumerate}
\end{prop}
\begin{rem}
A (randomized) stopping time satisfying (1) is called a \emph{stationary time}; one sasistfying (1) and (3) a \emph{strong stationary time}.
\end{rem}
\begin{rem}
\label{r:lawindep}
Note in particular that the distribution of $(\msf{T}_N, X_{\msf{T}_N})$ does not depend on the environment $\tau_N$.
\end{rem}
\begin{proof} 

Since the jump rates of $X$ are symmetric, the measure $\mfk{u}_N$ is reversible for this walk. Let 
$$\mcl{E}_N(f,f) 
= \frac{1}{2^{N+1}} \sum_{x\sim y} (f(x)-f(y))^2 e^{a_N(E_x+E_y)}\,,$$
where $x\sim y$ means that $x$ and $y$ are neighbours in $V_N$. 
For a given realization of the $E_x$'s, the bilinear form $\mcl{E}_N$ acting on $L_2(V_N,\mfk{u}_N)$ is the Dirichlet form of the process $X$. 
We similarly introduce 
$$\mcl{E}_N^\circ(f,f) 
= \frac{1}{2^{N+1}} \sum_{x\sim y} (f(x)-f(y))^2,$$  
which defines the Dirichlet form of the Markov chain jumping along the edges of $V_N$ with constant rate $1$. We denote this Markov chain by $(X_t^\circ)_{t \ge 0}$. 

Observe that 
\begin{equation}
\label{comparisondirichlet}
\mcl{E}_N^\circ(f,f)\le\mcl{E}_N(f,f).
\end{equation}
Let 
$$
\msf{SG}_N = \inf\Ll\{  \frac{\mcl{E}_N(f,f)}{\mfk{u}_N\Ll[(f-\mfk{u}_N(f))^2\Rr]},\  f : V_N \to \R \text{ non constant} \Rr\}
$$
be the spectral gap associated with the Dirichlet form $\mcl{E}_N$ on $L_2(V_N,\mfk{u}_N)$, and let 
$$
\msf{SG}_N^\circ = \inf\Ll\{  \frac{\mcl{E}_N^\circ(f,f)}{\mfk{u}_N\Ll[(f-\mfk{u}_N(f))^2\Rr]},\  f : V_N \to \R \text{ non constant} \Rr\}
$$
be the spectral gap of $\mcl{E}^\circ_N$. 

On the one hand, inequality (\ref{comparisondirichlet}) implies that $\msf{SG}_N^\circ\le \msf{SG}_N$. On the other hand, 
$X^\circ$ is simply the $N$-fold product of the simple random walk on $V_1$, so by \cite[Theorem~2.5]{gz}, we have $\msf{SG}^\circ_N = \msf{SG}^\circ_1$, and in fact $\msf{SG}^\circ_1 = 2$. Therefore 
\begin{equation}
\label{spectralgap}
	\msf{SG}_N \ge 2.
\end{equation}

Let us now introduce
$$
s_\msf{ST}(\tau_N,t) = \min\Ll\{ s \ge 0 :  \forall x,y \in V_N, \ \PPt_x\Ll[X_t = y\Rr] \ge (1-s) \mfk{u}_N(y)\Rr\},
$$
$$
\ov{d}(\tau_N,t) = \max_{x,y \in V_N} \Ll\| \PPt_x\Ll[X_t \in \cdot\Rr] - \PPt_y\Ll[X_t \in \cdot\Rr]\Rr\|_{\msf{TV}},
$$
where $\|\cdot \|_{\msf{TV}}$ denotes the total variation distance, and
$$
\mathscr{T}_N = \inf\{t \ge 0 : \ov{d}(\tau_N,t)\le e^{-1}\}.
$$
We learn from \cite[Lemmas 5, 7 and 23 of Chapter 4]{af} that, respectively,
$$
\ov{d}(\tau_N,t) \le e^{-\lfloor t/\mathscr{T}_N \rfloor},
$$
$$
s_\msf{ST}(\tau_N,2t) \le 1-(1-\ov{d}(\tau_N,t))^2,
$$
$$
\mathscr{T}_N \le \frac{1}{\msf{SG}_N}\Ll(1+\frac{1}{2} \log \frac{1}{\mfk{u}_N^*}\Rr),
$$
where $\mfk{u}_N^* = \min_x \mfk{u}_N(x)$, which in our case is $2^{-N}$. Using \eqref{spectralgap}, we thus derive from these three observations that, for $N$ large enough and any $\tau_N$,
$$
s_\msf{ST}(\tau_N,N) \le e^{-1}, 
$$
that is, for any $x,y \in V_N$ and any $\tau_N$,
$$
\PPt_x[X_{N} = y] \ge (1-e^{-1}) \mfk{u}_N(y).
$$
Following \cite[Proposition (3.2)]{ad2}, we can thus define the strong stationary time $\msf{T}_N$, with values in $\{N,2N,\ldots\}$. Let $(U_1, U_2, \ldots)$ be independent random variables distributed uniformly on $[0,1]$, independent of everything else. Conditionally on $X_N= y$, we let $\msf{T}_N = N$ if 
$$
U_1 \le \frac{(1-e^{-1}) \mfk{u}_N(y)}{\PPt_x[X_N = y]} \quad (\le 1). 
$$
Otherwise, we define $\msf{T}_N$ inductively: for every $k \in \N$, conditionally on $\msf{T}_N > kN$, on $X_{kN} = z$ and on $X_{(k+1)N} = y$, we let $\msf{T}_N = (k+1)N$ if
$$
U_{k+1}\le \frac{(1-e^{-1}) \mfk{u}_N(y)}{\PPt_z[X_N = y]}\quad (\le 1).
$$
By construction, we have
$$
\PPt_x[\msf{T}_N = N \ | \ X_N = y] = \frac{(1-e^{-1}) \mfk{u}_N(y)}{\PPt_x[X_N = y]},
$$
so that
$$
\PPt_x[\msf{T}_N = N,  \ X_N = y] = (1-e^{-1}) \mfk{u}_N(y).
$$
Similarly, we have
$$
\PPt_x[\msf{T}_N = (k+1)N , \ X_{(k+1)N} = y \ | \ \msf{T}_N > kN, \ X_{kN} = z] = (1-e^{-1}) \mfk{u}_N(y).
$$
By induction on $k$, we obtain that for any $k \in \N$ and any $x,y \in V_N$,
$$
\PPt_x\Ll[\msf{T}_N = k N, \ X_{kN} = y\Rr] = e^{-(k-1)} (1-e^{-1}) \mfk{u}_N(y),
$$ 
and this finishes the proof.
\end{proof}

\begin{prop}
\label{p:fuite-sg}
For any environment $\tau_N$, any $x,y \in V_N$ and any $t \ge 0$,
$$
\Ll|\PPt_x[X_t = y] - 2^{-N} \Rr| \le e^{-2t}.
$$
\end{prop}
\begin{proof}
The proposition is a direct consequence of \cite[Corollary~2.1.5]{sc} and the spectral gap estimate \eqref{spectralgap}.
\end{proof}

\begin{prop}
\label{p:fuite}
For any $x \in V_N$ and any $t \ge 0$,
$$
\P\PPt_x[X_t = x] \le \Ll(\frac{1+e^{-2t}}{2}\Rr)^N.
$$
\end{prop}
\begin{proof}
We recall that  $(X_t^\circ)_{t \ge 0}$ is the Markov chain whose jump rate from one point to any of its neighbours is equal to $1$. Since these jump rates are dominated by the jump rates of $(X_t)$, and using \cite[Lemma~2.2]{fm}, we get that
$$
\P\PPt_x[X_t = x] \le \P\PPt_x[X_t^\circ = x].
$$
Each coordinate of $(X_t^\circ)$ is a Markov chain on two states with jump rate equal to $1$ from one state to the other, and that evolves independently of the other coordinates. This leads to the announced result.
\end{proof}
\begin{rem}
\label{r:fuite}
As is well-known, such on-diagonal upper bound on the heat kernel automatically transfers to off-diagonal decay. In other words, for any $x,y \in V_N$,
\begin{equation}
\label{e:fuite}
\P\PPt_x[X_t = y] \le \Ll(\frac{1+e^{-2t}}{2}\Rr)^N.
\end{equation}
To see this, note that
$$
\PPt_x[X_t = y] = \sum_{z} \PPt_x[X_{t/2} = z] \ \PPt_z[X_{t/2} = y],
$$
which by the Cauchy-Schwarz inequality, is smaller than
$$
\Ll( \sum_{z} \PPt_x[X_{t/2} = z]^2 \Rr)^{1/2} \Ll( \sum_{z} \PPt_z[X_{t/2} = y]^2 \Rr)^{1/2}.
$$
By reversibility, the first sum above is equal to
$$
\sum_{z} \PPt_x[X_{t/2} = z] \ \PPt_z[X_{t/2} = x] = \PPt_x[X_t = x],
$$
and we arrive at
$$
\PPt_x[X_t = y] \le \PPt_x[X_t = x]^{1/2} \ \PPt_y[X_t = y]^{1/2}.
$$
After integrating with respect to $\P$, we use the Cauchy-Schwarz inequality again and Proposition~\ref{p:fuite} to obtain \eqref{e:fuite}.
\end{rem}
%
%
%
%
%
%%%%%%%%%%%%%%%%%%%%%%%%%%%%%%%%%%%%%%%%%%%%%%%%%%%%%%%%%%%%%%
%%%%%%%%%%%%%%%%%%%%%%%%%%%%%%%%%%%%%%%%%%%%%%%%%%%%%%%%%%%%%%
%
%
%
\section{The exploration process}
\label{s:explore}
In this section, we define an enumeration of $V_N$ along the trajectory of the random walk $X$, as follows. Since the walk eventually visits every site of $V_N$, we can let $Y_1,\ldots, Y_{{2^N}}$ be the sequence of successive sites visited by $X$. Consider the sequence of sites
$$
\mathcal{S} = (Y_1+z)_{|z |\le 1}, \ldots, (Y_{{2^N}}+z)_{|z| \le 1},
$$
where $(\cdot+z)_{|z| \le 1}$ is enumerated in some arbitrary order.
%, starting with the element corresponding to $z = 0$
Clearly, $\mathcal{S}$ spans $V_N$. We let
\begin{equation}
\label{defxn}
\mathcal{S'} = x_1, x_2, \ldots, x_{2^N}
\end{equation}
be the sequence $\mathcal{S}$ with repetitions removed (in other words, for every $x \in V_N$, the first occurrence of $x$ in $\mcl{S}$ is kept, and all subsequent occurrences are deleted). The sequence $\mcl{S}'$ spans $V_N$ by construction, and we refer to this sequence as the \emph{exploration process}.

Let $x \in V_N$. The distribution of the trajectory up to the time when the walk is within distance $1$ from $x$ is independent of $E_x$. From this observation, we can derive the following result (see \cite[Proposition~4.1]{scaling} for a complete proof).
\begin{prop}
\label{p:iid}
For any $z \in V_N$, under the measure $\P\PPt_z$, the random variables $(E_{x_n})_{1 \le n \le 2^N}$ are independent and follow a $\mcl{N}_+(0,1)$ distribution. This holds also under the measure $\ov{\P}_N$.
\end{prop}

%
%
%
%
%
%%%%%%%%%%%%%%%%%%%%%%%%%%%%%%%%%%%%%%%%%%%%%%%%%%%%%%%%%%%%%%
%%%%%%%%%%%%%%%%%%%%%%%%%%%%%%%%%%%%%%%%%%%%%%%%%%%%%%%%%%%%%%
%
%
%
\section{Law of large numbers for the number of discovered sites}
\label{s:lln}
\setcounter{equation}{0}
For every $t \ge 0$, let 
$$
\mcl{D}_N(t) = \{x \in V_N : \ \exists s \le t \text{ s.t. } |X_s-x| \le 1 \},
$$ 
$$
\mcl{R}_N(t) = \{x \in V_N : \ \exists s \le t \text{ s.t. } X_s=x  \},
$$ 
and 
$D_N(t) = |\mcl{D}_N(t)|$, $R_N(t)=|\mcl{R}_N(t)|$. We say that a site $x \in V_N$ is \emph{discovered before time} $t$ if it belongs to $\mcl{D}_N(t)$, 
and that it is \emph{visited before time} $t$ if it belongs to $\mcl{R}_N(t)$.

More generally, for a subset $A\subset V_N$, we define 
$$
\mcl{D}^A_N(t) = \{x \in V_N : \ \exists s \le t \text{ s.t. } X_s\in x+A \},
$$ 
and 
$D^A_N(t) = |\mcl{D}^A_N(t)|$. Observe that $\mcl{D}^A_N(t) = \mcl{D}_N(t)$ when $A=\{z \text{ s.t. } \vert z\vert\le 1\}$ 
and $\mcl{D}^A_N(t) = \mcl{R}_N(t)$ when $A$ is reduced to the identity element in $V_N$. 

The purpose of this section is to prove the following law of large numbers for the number of discovered sites.

\begin{thm}
\label{t:lln}
Let $\ov{c} \in (0,\log(2))$ and $(\msft)$ be such that $\log(\msft) \sim \ov{c} N$. 
Define  $\msfd=\ov{\E}_N[D_N(\msft)]$. For any $t>0$, we have 
\begin{equation} \label{lln}
\frac{D_N(t\msft)}{t\msfd} \xrightarrow[N \to + \infty]{\ov{\P}_N\text{-prob.}} 1.
\end{equation} 
Moreover, $\log(\msfd) \sim \ov{c} N$. 
\end{thm}

Before starting the proof of this result, we collect some useful bounds on $\ov{\omega}_N = \sup_{x \in V_N} \omega_N(x)$, where we recall that $\omega_N(x)$ was defined in \eqref{defomegaN}.
\begin{lem}
\label{l:ovomega}
For every $c > \sqrt{2 \log 2}$, we have
\begin{equation}
\label{e:ovomega1}	
\P\Ll [\log(\ov{\omega}_N) > 2 c \ a_N \sqrt{N}\Rr] \xrightarrow[N \to \infty]{} 0.
\end{equation}
Moreover, for $N$ sufficiently large,
\begin{equation}
\label{e:ovomega2}	
\P\Ll [\log(\ov{\omega}_N) > N^{3/4} \Rr] \le e^{-N^{5/4}}.
\end{equation}
\end{lem}
\begin{proof}
In order to prove these results, we will rely on the fact that for $z \ge 1$,
\begin{equation}
\label{easybound}
\P[E_x > z] \le e^{-z^2/2}.
\end{equation}
From this, it readily follows that, for every $c > 0$,
\begin{eqnarray*}
\P\Ll [\sup_{x \in V_N} E_x > c \sqrt{N}\Rr] & \le & 2^N \P[E_x > c \sqrt{N}] \\
& \le & 2^Ne^{-c^2 N/2},
\end{eqnarray*}
which tends to $0$ as soon as $c^2 > 2 \log(2)$. On the event
$$
\sup_{x \in V_N} E_x \le c \sqrt{N},
$$
we have
$$
\ov{\omega}_N \le N e^{2 a_N c \sqrt{N}},
$$
so \eqref{e:ovomega1} is proved.

For the second part, note that
$$
\P[\sup_{x \in V_N} E_x > N^{2/3}] \le 2^N e^{-N^{4/3}/2}.
$$
On the event
$$
\sup_{x \in V_N} E_x \le N^{2/3},
$$
we have
$$
\ov{\omega}_N \le N e^{2 a_N N^{2/3}}.
$$
Since we assume that $a_N \le \ov{a} \sqrt{2 \log N}$, we obtain \eqref{e:ovomega2}.
\end{proof}

\begin{proof}[Proof of Theorem~\ref{t:lln}]
In a first step of the proof, we shall derive estimates on the mean and variance of $D_N(t)$ for a fixed realization of $\tau_N$. 
These imply a law of large numbers as in (\ref{lln}) but with $\msfd$ replaced by $\EEtu\Ll[D_N(\msft)\Rr]$ (and therefore depending on $\tau_N$). 
In a second step, we establish a concentration inequality which allows to replace $\EEtu\Ll[D_N(\msft)\Rr]$ by $\ov{\E}_N[D_N(\msft)]$.

We start by proving that $\log(\msfd) \sim \ov{c} N$ and $\EEtu\Ll[D_N(\msft)\Rr]\sim \ov{c} N$ in $\P$-probability. 
It will be a consequence of the following 

\begin{lem}\label{l:meannumber} 
We have 
$$\EEtu\Ll[R_N(t)\Rr] \le 3t \ (1+\ov{\omega}_N)$$ whenever $t\ge 1$, and 
$$\EEtu\Ll[R_N(t)\Rr]\ge \frac t 2$$ whenever $t2^{-N}\le\frac 12$.
\end{lem} 

\begin{proof}
Let use fix a time $t>0$ and consider an independent exponentially distributed random variable of mean $t$, say  $T$. 
Let 
$$G_N^t(x,y)=\EEt_x\Ll[\int_0^{+\infty}e^{-s/t} \  \1_{X_s = y} \ \d s \Rr],$$ 
be the Green function of the process $X$. 
Our estimates rely on the following formula: 
\begin{equation} \label{equalitygreen} 
\EEtu\Ll[R_N(T)\Rr]=2^{-N}\sum_y \frac{t}{G_N^t(y,y)}.\end{equation} 
Indeed it follows from the Markov property that
$$\EEt_x\Ll[R_N(T)\Rr]=   \sum_{y \in V_N} \PPt_x[X \text{ visits } y \text{ before time } T]  = \sum_{y \in V_N} \frac{G_N^t(x,y)}{G_N^t(y,y)}.$$ 
%The invariance of $\P$ under left multiplication on $V_N$ implies that, when considered as random variables depending on $\tau_N$, 
%for any $z\in V_N$, 
%then $(G_N^t(x,y))_{x,y\in V_N}$ and $(G_N^t(x-z,y-z))_{x,y\in V_N}$ have the same law. 
%In particular ${G_N^t(x,y)}/{G_N^t(y,y)}$ has the same law as ${G_N^t(x-y,0)}/{G_N^t(0,0)}$. 
%
The reversibility of the process $X$ implies that $G_N^t(x,y)=G_N^t(y,x)$. Therefore, summing over $x$, we get that 
$$\EEtu\Ll[R_N(T)\Rr]=2^{-N}\sum_x\sum_y \frac{G_N^t(y,x)}{G_N^t(y,y)}.$$ 
%$$\E\Ll[\sum_y \frac{G_N^t(x,y)}{G_N^t(y,y)}\Rr]=\E\Ll[\sum_y \frac{G_N^t(0,x-y)}{G_N^t(0,0)}\Rr].$$
The equality $\sum_x G_N^t(y,x)=t$ leads to (\ref{equalitygreen}). 

We now explain how to derive an upper bound on $\EEtu[R_N(T)]$  from (\ref{equalitygreen}). We assume that $t\ge 1$. 

When starting from $y$, the process $X$ waits for an exponential time of parameter $\omega_N(y)=\sum_{z\sim y} \exp(a_N(E_y+E_z))$ before leaving the point $y$. 
Thus we have 
$$G_N^t(y,y)\ge \frac 1{1/t+\omega_N(y)}\ge \frac 1{1+\omega_N(y)}.$$ 
Plugging this estimate in (\ref{equalitygreen}), we get that 
$$\EEtu\Ll[R_N(T)\Rr]\le t 2^{-N}\sum_y (1+\omega_N(y)).$$ 
%$$\ov{\E}_N\Ll[R_N(T)\Rr]\le t(1 + \E\Ll[\sum_{z\sim 0} e^{a_N(E_0+E_z)}\Rr]).$$ 
%Since $\E\Ll[e^{a_N(E_0+E_z)}\Rr]=\E\Ll[e^{a_NE_0}\Rr]^2=\phi(a_N)^2\le 4\exp(a_N^2)$, we have 
%$$\ov{\E}_N\Ll[R_N(T)\Rr]\le t (1+4N e^{a_N^2}).$$ 
%Since $N e^{a_N^2}\ge 1$, we also have 
%$\ov{\E}_N\Ll[R_N(T)\Rr]\le 5t N e^{a_N^2}.$
%
%
One may now replace $T$ by $t$ observing that the function 
$s\to\EEtu\Ll[R_N(s)\Rr]$ is increasing and therefore 
$$\EEtu\Ll[R_N(T)\Rr]
=\int_0^{+\infty} \EEtu\Ll[R_N(st)\Rr]\ e^{-s}\ \d s 
\ge \EEtu\Ll[R_N(t)\Rr] \int_1^{+\infty} e^{-s}\ \d s.$$ 
Thus we have established that, for all $t\ge 1$, 
$$\EEtu\Ll[R_N(t)\Rr] \le 3t \ 2^{-N} \sum_y (1+\omega_N(y))\le 3t \ (1+\ov{\omega}_N) $$ 
(we used the inequality $\int_1^{+\infty} e^{-s}\ \d s=e^{-1}\ge 1/3$).

We now turn to lower bounds. We assume that $t2^{-N}\le\frac 12$. 
%
%Let 
%$$G_N^{\circ,t}(x,y)=\E_x\Ll[\int_0^{+\infty}e^{-s/t} \  \1_{X^\circ_s = y} \ \d s \Rr],$$ 
%be the Green function of the random walk $X^\circ$. It follows from (\ref{comparisondirichlet}) that 
%$$G_N^t(0,0)\le G_N^{\circ,t}(0,0)$$ 
Integrating the estimate  from Proposition \ref{p:fuite-sg}, we get that 
$$G_N^t(y,y)\le t2^{-N}+\frac t{1+2t}\le 1$$ 
(we assumed that $t2^{-N}\le 1/2$).  

Plugging this estimate in (\ref{equalitygreen})   
we see that, for all $t$ such that  $t2^{-N}\le 1/2$, one has 
$$\EEtu\Ll[R_N(T)\Rr] \ge  t . $$ 

One may now replace $T$ by $t$ observing that the function 
$s\to\EEtu\Ll[R_N(s)\Rr]$ is subadditive and increasing; thus satisfies 
$\EEtu\Ll[R_N(st)\Rr]\le (s+1)\EEtu\Ll[R_N(t)\Rr]$ and 
$$\EEtu\Ll[R_N(T)\Rr]\le \EEtu\Ll[R_N(t)\Rr]\int_0^{+\infty}(s+1)e^{-s}\ \d s.$$ 

Thus we have established that, for all $t$ such that  $t2^{-N}\le 1/2$, 
$$\EEtu\Ll[R_N(t)\Rr] \ge \frac t 2.$$ 
\end{proof}

Observe from \eqref{anconstraint} and \eqref{e:ovomega1} that $\frac 1 N \log (1+\ov{\omega}_N) $ converges to $0$ in $\P$-probability. 
Thus, if we replace $t$ by $\msft$, we deduce from the bounds in Lemma \ref{l:meannumber} that 
$ \frac 1 N \log \EEtu\Ll[R_N(\msft)\Rr]$ converges to  $\ov{c}$  in $\P$-probability. 
Since $(N+1)\cdot R_N(t)\ge D_N(t)\ge R_N(t)$, we also conclude that 
$ \frac 1 N \log \EEtu\Ll[D_N(\msft)\Rr]$ converges to  $\ov{c}$ in $\P$-probability.  

Taking the expectation in the bounds in Lemma \ref{l:meannumber} and using (\ref{defphi}), we also get that 
\begin{equation} \label{eq:mean} 
\frac 1 N \log \ov{\E}_N\Ll[D_N(\msft)\Rr] \xrightarrow[N \to + \infty]{} \ov{c}. 
\end{equation} 

Now we wish to show that $\EEtu\Ll[D_N(\msft)\Rr]^{-1} D_N(\msft)$ converges to $1$ in probability. 
This will follow from the fact that the variance of $D_N(\msft)$ under $\PPtu$ is negligible compared to the square of its mean.

\begin{lem}\label{l:crudevariance} 
For all $p>0$, there exists a numerical constant $c(p)$ such that 
$$\EEt_x\Ll[R_N(t)^p\Rr]\le c(p) (1+(\ov{\omega}_N t)^p).$$
\end{lem} 

\begin{proof}
Let $k\ge1$. On the event $R_N(t)\ge k+1$, the process has visited and left some point at least $k$ times. 
Note that, when at $x$,  the process $X$ waits for an exponential time of parameter $\omega_N(x)$ before leaving the point $x$. 
Thus we see that the probability $\PPt_x\Ll[R_N(t)\ge k+1\Rr]$ is bounded by the probability that a sum of $k$ i.i.d. 
exponential$(1)$ random variables exceeds $\ov{\omega}_N t$. The lemma follows from this observation. 
\end{proof}

We deduce from Lemma \ref{l:crudevariance} with $p=2$ and \eqref{e:ovomega1} that, for some constant $c_2$, 
\begin{equation}\label{crudevariance} 
\EEt_x\Ll[R_N(t)^2\Rr]\le c(2)(1+t^2e^{c_2a_N\sqrt{N}}),
\end{equation} 
with high $\P$-probability. 
Since $D_N(t) \le (N+1) R_N(t)$, a similar bound holds for $D_N(t)$ with an extra factor $(N+1)^2$ multiplying $c(2)$. 

We also need some control on the self-intersections of the trajectories of the process $X$. 
In the next lemma, we consider two Markov chains $X$ and $X'$ that start from the same point in $V_N$ (which will be chosen according to the 
uniform probability in $V_N$),  
and then evolve independently in the same fixed environment $\tau_N$ with jump rate from $x$ to $y$ given by $\omega_N(x,y)$ (defined in \eqref{defomeganxy}). We define 
$$
\mcl{D'}^A_N(s) = \{x \in V_N : \ \exists v \le s  \text{ s.t. } X'_v\in x+A \},
$$ 
and $\mcl{D}^A_N(s,t) = \mcl{D}^A_N(t) \cap \mcl{D'}^A_N(s)$,  
$D^A_N(s,t) = |\mcl{D}^A_N(s,t)|$. 

\begin{lem}\label{l:intersectionsausages} 
Let $s\ge t\ge 1$. We have 
$$\EEtu\Ll[D^A_N(s,t)\Rr] \le 9(2^{-N} s^2+1) \vert A\vert ^2\ (1+\ov{\omega}_N)^2. $$
\end{lem} 

\begin{proof} 

The strategy is similar to the one we already used in the proof of the upper bound in Lemma \ref{l:meannumber}. 

Let us fix two times $s\ge t\ge 1$ and consider independent exponentially distributed random variables $S$ and $T$ of mean respectively $s$ and $t$. 

We recall that  
$$G_N^t(x,y)=\EEt_x\Ll[\int_0^{+\infty}e^{-w/t} \  \1_{X_w = y} \ \d w \Rr],$$ 
is the Green function of the process $X$. Let us also define 
$$K_N^t(x,y)=\EEt_x\Ll[\int_0^{+\infty}e^{-w/t} \  \1_{X_w = y} \ w\ \d w \Rr].$$

We claim that 
\begin{equation} \label{inequalitygreen} 
\EEtu\Ll[D^A_N(S,T)\Rr]\le\ 2^{-N}\sum_x\sum_{y,y'\in A}\frac {K_N^s(x+y,x+y')}{G_N^t(x+y,x+y)G_N^s(x+y',x+y')}.
\end{equation} 

To prove equation (\ref{inequalitygreen}), we start by noting that 
\begin{equation}\label{intermediateinequalitygreen}
\EEt_x\Ll[D^A_N(S,T)\Rr]\le \sum_z \sum_{y,y'\in A} \frac {G_N^t(x,z+y)}{G_N^s(z+y,z+y)} \frac {G_N^s(x,z+y')}{G_N^s(z+y',z+y')}.
\end{equation}  
(This follows from the independence of $X$ and $X'$ and the Markov property: $\frac {G_N^t(x,z+y)}{G_N^t(z+y,z+y)}$ is the probability that $X$ 
hits $z+y$ before time $T$. We then use a union bound to control the probability of hitting the set $z+A$ by a sum over $y\in A$.) 

%Taking the expectation with respect to $\E$ and using the translation invariance, we then deduce that 
%$$\ov{\E}_N\Ll[D^A_N(S,T)\Rr]\le \sum_{y,y'\in A} \E\Ll[\sum_x  \frac {G_N^t(x,y)}{G_N^t(y,y)} \frac {G_N^s(x,y')}{G_N^s(y',y')}\Rr].$$ 

%{\color{red} (* add indication on what is done below ? *)}

From the reversibility, we get that $G_N^t(x,z+y)=G_N^t(z+y,x)$. 
Observe that 
\begin{eqnarray*}
&&\sum_x G_N^t(z+y,x) G_N^s(x,z+y')\\ 
& = &\int_0^{+\infty}\d u\ e^{-u/t}\int_0^{+\infty}\d v\ e^{-v/s}\ \sum_x\PPt_{z+y}[X_u=x]\PPt_x[X_v=z+y'] \\ 
& = &\int_0^{+\infty}\d u\ e^{-u/t}\int_0^{+\infty}\d v\ e^{-v/s}\ \PPt_{z+y}[X_{u+v}=z+y'],  
\end{eqnarray*} 
and, since $s\ge t$,  
%\begin{eqnarray*} Ê 
%\int_0^{\infty}\d u\ e^{-u/t}\int_0^{+\infty}\d v\ e^{-v/s}\ \PPt_y[X_{u+v}=y']&\le&  \int_0^{+\infty}\d u\ e^{-u/s}\int_0^{+\infty}\d v\ e^{-v/s}\ \PPt_y[X_{u+v}=y']Ê\\ 
% &=& K^s (y,y'). 
%\end{eqnarray*}
\begin{eqnarray*}
&  &\int_0^{+\infty}\d u\ e^{-u/t}\int_0^{+\infty}\d v\ e^{-v/s}\ \PPt_{z+y}[X_{u+v}=z+y']\\ 
&\le&  \int_0^{+\infty}\d u\ e^{-u/s}\int_0^{+\infty}\d v\ e^{-v/s}\ \PPt_{z+y}[X_{u+v}=z+y']=K^s(z+y,z+y'),\\ 
\end{eqnarray*} 
Summing inequality (\ref{intermediateinequalitygreen}) over $x$, we thus obtain (\ref{inequalitygreen}).

As in the proof of Lemma \ref{l:meannumber}, we plug in (\ref{inequalitygreen}) crude estimates on the Green functions. 
We first recall that, for all $z$ and $z'$, 
$$G_N^t(z,z)\ge \frac 1{1/t+\omega_N(z)}\ge \frac 1{1+\omega_N(z)},$$ 
and it follows from Proposition \ref{p:fuite-sg} that 
$$K_N^s(z,z')\le 2^{-N}s^2+1.$$ 
Thus we derive that 
\begin{eqnarray*}
\EEtu[D^A_N(S,T)]&\le&\ (2^{-N}s^2+1)\sum_{y,y'\in A} \ 2^{-N}\sum_x (1+\omega_N(x+y))(1+\omega_N(x+y'))\\
&=&  \ (2^{-N}s^2+1)\ 2^{-N} \sum_x (\sum_{y\in A} (1+\omega_N(x+y)))^2\\
&\le& \ (2^{-N}s^2+1)\ 2^{-N}\vert A\vert^2 \sum_x  (1+\omega_N(x))^2\\ 
&\le& \ (2^{-N}s^2+1)\ \vert A\vert^2 (1+\ov{\omega}_N)^2, 
\end{eqnarray*}  
where we used Cauchy-Schwarz in the last inequality.

On the other hand, the function $(s,t)\rightarrow \EEtu\Ll[D^A_N(s,t)\Rr]$ being increasing in both arguments, we have 
$\EEtu\Ll[D^A_N(s,t)\Rr]\le e^2\ \EEtu[D^A_N(S,T)]$. 

Putting everything together, we conclude the proof of Lemma \ref{l:intersectionsausages}. 
\end{proof} 

As a consequence of Lemma \ref{l:intersectionsausages}, we obtain the linearity of the expected number of discovered sites. 

\begin{lem}\label{l:linearity} 
For all $t$ and $s$ we have 
$$\EEtu\Ll[D^A_N(t)\Rr]+\EEtu\Ll[D^A_N(s)\Rr] \ge \EEtu\Ll[D^A_N(t+s)\Rr],$$ and 
$$\EEtu\Ll[D^A_N(t)\Rr]+\EEtu\Ll[D^A_N(s)\Rr] \le \EEtu\Ll[D^A_N(t+s)\Rr]+\EEtu\Ll[D^A_N(s,t)\Rr].$$
 \end{lem} 

\begin{proof} 

The first inequality is obvious. 

We define 
\begin{equation}
\label{defDarrow}
\mcl{D}^A_N(t\rightarrow s) = \{x \in V_N : \ \exists u\in (t,t+s]  \text{ s.t. } X_s\in x+A \}
\end{equation}
to be the set of discovered sites between times $t$ and $t+s$, 
and 
$D^A_N(t\rightarrow s) = |\mcl{D}^A_N(t\rightarrow s)|$.

Note that, by stationarity, we have $\EEtu\Ll[D^A_N(t\rightarrow s)\Rr]=\EEtu\Ll[D^A_N(s)\Rr]$. Therefore 
\begin{eqnarray*}&&\EEtu\Ll[D^A_N(t)\Rr]+\EEtu\Ll[D^A_N(s)\Rr] - \EEtu\Ll[D^A_N(t+s)\Rr]\\ 
&=&\EEtu\Ll[D^A_N(t)+D^A_N(t\rightarrow s)-D^A_N(t+s)\Rr].\end{eqnarray*} 
This last quantity is bounded by the number of sites that are discovered both before time $t$ and also between times $t$ and $t+s$ 
i.e. $|\mcl{D}^A_N(t)\cap\mcl{D}^A_N(t\rightarrow s)|$. 

The Markov property (applied at time $t$) and the reversibility of $\mfk{u}_N$ imply that $|\mcl{D}^A_N(t)\cap\mcl{D}^A_N(t\rightarrow s)|$ 
has the same law as $D^A_N(s,t)$, thus leading to the inequality in Lemma \ref{l:linearity}.  
\end{proof} 

We deduce from Lemma \ref{l:linearity} that, for all $t>0$,  
\begin{equation}\label{eq:lin} 
\frac{ \ov{\E}_N\Ll[D_N(t\msft)\Rr] } { t\ov{\E}_N\Ll[D_N(\msft)\Rr] }
\xrightarrow[N \to + \infty]{} 1. 
\end{equation}

The upper bound in (\ref{eq:lin}) follows from the integrated version of the first inequality in Lemma  \ref{l:linearity}. 
To get the lower bound, we also take the expectation in the second inequality in Lemma  \ref{l:linearity} and iterate it 
approximately $t$ times. The cumulated error term is estimated with Lemma \ref{l:intersectionsausages}: neglecting 
subexponential terms, it is of order $2^{-N}\msft^2$. On the other hand, we already know from (\ref{eq:mean}) that 
$\ov{\E}_N\Ll[D_N(\msft)\Rr]$ is of the same order as $\msft$ on the exponential scale. Since $2^{-N}\msft^2$ is negligible 
compared to $\msft$, we can conclude the proof.

Our next task is to get estimates on the variance of $D^A_N(t)$.  
We shall use the notation $\VVtu(X)$ to denote the variance of a random variable $X$ under $\PPtu$ and 
$\CCtu(X,Y)$ to denote the covariance of the random variables $X$ and $Y$ under $\PPtu$. 

We obviously have 
\begin{equation}\label{eq:1} 
\VVtu(X+Y)\le \VVtu(X)+\EEtu(Y^2)+2\CCtu(X,Y). 
\end{equation}
If we write that  $D^A_N(t+s)=D^A_N(t)+(D^A_N(t+s)-D^A_N(t))$, we are lead to estimating $\CCtu\Ll[D^A_N(t),D^A_N(t+s)-D^A_N(t)\Rr]$. 

Observe that $D^A_N(t+s)-D^A_N(t)\le D^A_N(t\rightarrow s)$ (where $D^A_N(t\rightarrow s)$ was defined in \eqref{defDarrow}), and therefore 
$$\EEtu\Ll[D^A_N(t)(D^A_N( t+s )-D^A_N(t))\Rr]\le \EEtu\Ll[D^A_N(t)D^A_N(t\rightarrow s)\Rr].$$  

On the other hand we read from Lemma \ref{l:linearity} that 
$\EEtu\Ll[D^A_N(t+s)-D^A_N(t)\Rr]\ge \EEtu\Ll[D^A_N(s)\Rr]-\EEtu\Ll[D^A_N(s,t)\Rr]$. 

Thus we have 
\begin{multline}\label{eq:2} 
\CCtu\Ll[D^A_N(t),D^A_N(t+s)-D^A_N(t)\Rr]\\
\le \CCtu\Ll[D^A_N(t),D^A_N(t\rightarrow s)\Rr]+\EEtu\Ll[D^A_N(t)\Rr]\EEtu\Ll[D^A_N(s,t)\Rr].
\end{multline}  

\begin{lem}\label{l:covariance}  
For all $t$ and $s\ge N^2$ we have 
$$\CCtu\Ll[D^A_N(t),D^A_N(t\rightarrow s)\Rr]
\le 2c(2)\vert A\vert^2(1+\ov{\omega}_N t)(1+\ov{\omega}_N N^2)+2^{3N}e^{-2N^2} .$$
 \end{lem}  

\begin{proof}

We start observing that 
\begin{eqnarray*} 
&&\CCtu\Ll[D^A_N(t),D^A_N(t+N^2\rightarrow s-N^2)\Rr]\\ 
&\le& 2^N e^{-2N^2} \EEtu\Ll[D^A_N(t)\Rr]\EEtu\Ll[D^A_N(t\rightarrow s)\Rr] 
\le 2^{3N}e^{-2N^2}. 
\end{eqnarray*} 
(The first inequality is a consequence of  Proposition \ref{p:fuite-sg}, and the second one is obvious.) 

Therefore 
\begin{eqnarray*}
&&\CCtu\Ll[D^A_N(t),D^A_N(t\rightarrow s)\Rr]\\ 
&\le& \CCtu\Ll[D^A_N(t),D^A_N(t+N^2\rightarrow s-N^2)\Rr]+\EEtu\Ll[D^A_N(t)D^A_N(t\rightarrow N^2)\Rr]\\ 
&&+ \EEtu\Ll[D^A_N(t)\Rr](\EEtu\Ll[D^A_N(t+N^2\rightarrow s-N^2)\Rr]-\EEtu\Ll[D^A_N(t\rightarrow s)\Rr])\\ 
&\le& 2^{3N}e^{-2N^2}+\EEtu\Ll[D^A_N(t)D^A_N(t\rightarrow N^2)\Rr]
+ \EEtu\Ll[D^A_N(t)\Rr] \EEtu\Ll[D^A_N(t\rightarrow N^2)\Rr]\\ 
&\le& 2^{3N}e^{-2N^2}+2\EEtu\Ll[D^A_N(t)^2\Rr]^{1/2}\EEtu\Ll[D^A_N(N^2)^2\Rr]^{1/2}.
\end{eqnarray*}
One now concludes the proof using the estimates from Lemma \ref{l:crudevariance}. 
\end{proof}

\begin{lem}
\label{l:lln} 
$$ \frac{D^A_N(\msft)}{\EEtu\Ll[D^A_N(\msft)\Rr]} \xrightarrow[N \to + \infty]{\ov{\P}_N\text{-prob.}} 1.$$
\end{lem} 

\begin{proof}

We split the time interval $[0,\msft]$ into $K$ pieces of equal size. 
Here $K$ has to be chosen properly: $K=e^{N^{3/4}}$ would do. 

We then estimate the variance of $D^A_N(\msft)$ using (\ref{eq:1}) and (\ref{eq:2}): 

\begin{eqnarray*}
\VVtu\Ll[D^A_N(\msft)\Rr]&\le& \VVtu\Ll[D^A_N(\frac{K-1}K \msft)\Rr]+\EEtu\Ll[D^A_N(\frac 1K \msft)^2\Rr]\\ 
&+& 2\CCtu\Ll[D^A_N(\frac{K-1}K\msft),D^A_N(\frac{K-1}K\msft\rightarrow \frac 1K\msft)\Rr] \\ 
&+& \EEtu\Ll[D^A_N(\msft)\Rr]\EEtu\Ll[D^A_N(\frac 1K\msft,\frac{K-1}K\msft)\Rr].
\end{eqnarray*}
 
Using our estimates from Lemmas \ref{l:crudevariance}, \ref{l:covariance}, \ref{l:meannumber} and \ref{l:intersectionsausages}, we get that 
$$\VVtu\Ll[D^A_N(\msft)\Rr]\le \VVtu\Ll[D^A_N(\frac{K-1}K \msft)\Rr]+\eta_N,$$ 
where 
\begin{eqnarray*}
\eta_N&=& c(2)\vert A\vert^2(1+(\ov{\omega}_N \frac 1 K\msft)^2)+ 
4c(2)\vert A\vert^2(1+\ov{\omega}_N \msft) (1+\ov{\omega}_N N^2)+2^{3N}e^{-2N^2}\\ 
&+& 
3\msft\vert A\vert (1+\ov{\omega}_N)9(1+2^{-N}\msft^2)\vert A\vert^2(1+\ov{\omega}_N). 
\end{eqnarray*} 

We iterate this computation $K$ times. This way we get that $\VVtu\Ll[D^A_N(\msft)\Rr]\le K\eta_N$. 

It then follows from our assumptions on $K$ that $\eta_N$ is negligible compared with $\msft^2/K$. 
(Observe in particular that, by Lemma~\ref{l:ovomega}, we have $(\ov{\omega}_N\frac 1 K\msft)^2 \ll \msft^2/K$.) 

Thus we deduce that $K\eta_N/{\msft^2}$ tends to $0$  in $\ov{\P}_N\text{-probability}$. 

We have checked that the variance of $D^A_N(\msft)$ is negligible compared with its mean with high $\ov{\P}_N\text{-probability}$, and this ends the proof. (Recall that $\msft$ is a lower bound for the mean of $D^A_N(\msft)$, see Lemma \ref{l:meannumber}.) 
\end{proof}

%Remarks on quenched estimates. 

%One may derive versions of Lemma \ref{l:meannumber} and \ref{l:intersectionsausages} 
%that hold in probability with respect to $\P$ rather than in mean. 

%As for the mean number of visited sites, we have 
%$$\EEtu\Ll[R_N(T)\Rr]=2^{-N}\sum_x\sum_y \frac{G_N^t(x,y)}{G_N^t(y,y)},$$  
%and therefore 
%\begin{equation} \label{quenchedmeannumber} 
%\EEtu\Ll[R_N(T)\Rr] 
%=t2^{-N}\sum_y\frac 1{G_N^t(y,y)}.
%\end{equation} 
%(See the proof of (\ref{equalitygreen}).)  

%Thus, using the same bounds as in the proof of Lemma \ref{l:meannumber}, we deduce that 
%\begin{equation}\label{meannumberquenchedestimate} 
%\frac t 2 \le \EEtu\Ll[R_N(t)\Rr] \le 3t \ 2^{-N} \sum_y (1+\omega_N(y)),  
%\end{equation} 
%whenever $1\le t\le \frac 12 2^N$. 

%Similarly, we get the corresponding quenched version of (\ref{inequalitygreen}), namely 
%$$\EEtu\Ll[D^A_N(S,T)\Rr]\le\ 2^{-N}\sum_x\sum_{y,y'\in A}\frac {K_N^s(x+y,x+y')}{G_N^t(x+y,x+y)G_N^s(x+y',x+y')},$$
%which leads to 
%$$\EEtu\Ll[D^A_N(s,t)\Rr] \le 9(2^{-N}s^2+1) 2^{-N} \sum_x \sum_{y,y'\in A} (1+\omega_N(x+y))(1+\omega_N(x+y')),$$ 
%and, using Cauchy-Schwarz inequality, 
%\begin{equation} \label{intersectionquenched}
%\EEtu\Ll[D^A_N(s,t)\Rr] \le 9(2^{-N} s^2+1) \vert A\vert ^2\ 2^{-N} \sum_x(1+\omega_N(x))^2. 
%\end{equation} 

%\begin{lem}\label{l:variance} 
%$$\VVtu\Ll[D_N(t)\Rr]\le$$ 
%\end{lem}

\noindent {\it End of the proof of Theorem \ref{t:lln}.} 
In order to deduce Theorem \ref{t:lln} from Lemma~\ref{l:lln}, it is sufficient to show that the variance of 
$\EEtu\Ll[D_N(\msft)\Rr]$ under $\P$ (i.e.\ as a function of the random variable $\tau_N$)  is negligible 
compared to $\msft^2$. 

We use the notation $\V[\msf{X}]$ to denote the variance of a random variable $\msf{X}$ with respect to $\P$. 
To estimate $\V[\EEtu\Ll[D_N(\msft)\Rr]]$, we first rely on the Efron-Stein inequality, namely 
$$\V[\EEtu\Ll[D_N(\msft)\Rr]]
\le\frac 12\sum_{x\in V_N} \E\Ll[(\EEtu\Ll[D_N(\msft)\Rr]-\EExtu\Ll[D_N(\msft)\Rr])^2\Rr],$$
where $\tau_N^{(x)}$ denotes an environment of traps in which we have replaced $E_x$ by an independent copy 
and left all other values unchanged. 

By symmetry, all the terms in this last sum have the same value, so that 
$$\V[\EEtu\Ll[D_N(\msft)\Rr]]\le 2^{N-1} \E\Ll[(\EEtu\Ll[D_N(\msft)\Rr]-\EEotu\Ll[D_N(\msft)\Rr])^2\Rr].$$ 

Thus we see that we will be done if we prove that 
$$2^N \E[(\EEtu\Ll[D_N(\msft)\Rr]-\EEotu\Ll[D_N(\msft)\Rr])^2]$$ 
is negligible with respect to $\msft^2$. We will in fact show that, on the exponential scale,  $ \E[(\EEtu\Ll[D_N(\msft)\Rr]-\EEotu\Ll[D_N(\msft)\Rr])^2]$ 
is at most of order $(2^{-N}\msft)^2$: 

\begin{equation}\label{eq:efronstein} 
\limsup_N \frac 1 N \log \E[(\EEtu\Ll[D_N(\msft)\Rr]-\EEotu\Ll[D_N(\msft)\Rr])^2] \le 2\lim_N \frac 1 N \log (2^{-N}\msft).
\end{equation} 

In order to show (\ref{eq:efronstein}), we split the trajectory of the walk in different time intervals: let $T_1$ be the first time the point $0$ 
is discovered by the process $X$; let ${\mathsf T}^{(1)}$ be the first strong stationary time after $T_1$: in other words,  ${\mathsf T}^{(1)}-T_1$ is 
a strong stationary time for the Markov chain $(X_{t+T_1})_{t\ge 0}$ obtained as in Proposition \ref{p:mixing}. 

We iterate these definitions: let $T_2$ be the first time $0$ is discovered after time ${\mathsf T}^{(1)}$; let ${\mathsf T}^{(2)}$ be the first strong stationary time after time $T_2$, \ldots
We define $I_j=[T_j,{\mathsf T}^{(j)})$ and $I=\cup_j I_j$. 

Let us also define 
$$
\mcl{D}^{(0)}_N(t) = \{x \in V_N : \ \exists s \le t\,,\, s\notin I \text{ s.t. } |X_s-x| \le 1 \},
$$ 
 
and 
$D^{(0)}_N(t) = |\mcl{D}^{(0)}_N(t)|$. 

From the construction of strong stationary times, see Proposition \ref{p:mixing} and its proof, it follows that the law of $\mcl{D}^{(0)}_N(t)$ 
(and therefore the law of $D^{(0)}_N(t)$ and its mean) is the same under $\PPtu$ or $\PPotu$. 
Therefore $\EEtu\Ll[D^{(0)}_N(\msft)\Rr]=\EEotu\Ll[D^{(0)}_N(\msft)\Rr]$ and 
\begin{eqnarray*} 
&&\EEtu\Ll[D_N(\msft)\Rr]-\EEotu\Ll[D_N(\msft)\Rr]\\ 
&=&\EEtu\Ll[D_N(\msft)-D^{(0)}_N(\msft)\Rr]-\EEotu\Ll[D_N(\msft)-D^{(0)}_N(\msft)\Rr].\end{eqnarray*}  
Note that these two terms have the same law under $\P$. 
Thus we will have established (\ref{eq:efronstein}) once we show that 
$\E\Ll[\EEtu[D_N(\msft)-D^{(0)}_N(\msft)]^2\Rr]$ is at most of order $(2^{-N}\msft)^2$ (up to subexponential factors).

Applying the Markov property and using the stationarity of the uniform probability, we have  
$$\Ll|\EEtu\Ll[D_N(\msft)-D^{(0)}_N(\msft)\Rr] \Rr|
\le \max_{x\sim 0} \EEt_x\Ll[D_N({\mathsf T}_N)\Rr] \sum_{k\ge 1} {\PPtu[T_k\le \msft]},$$ 
where ${\msf T}_N$ is the strong stationary time considered in Proposition \ref{p:mixing}. 
It is easy to deduce from Proposition \ref{p:mixing} and Lemma \ref{l:crudevariance} that 
$\max_{x\sim 0} \EEt_x\Ll[D_N({\mathsf T}_N)\Rr]$ is of subexponential order. 

To bound the second term, observe that, when the process $X$ hits a point $x$, it then spends at $x$ an exponentially 
distributed time of parameter $\omega_N(x)\le\ov{\omega}_N$. Therefore, for all $\lambda>0$, 
$$\EEtu\Ll[\int_0^{+\infty} e^{-\lambda s} \,\1_{X_s\sim 0}\, \d s\Rr]
\ge \frac 1{\lambda+\ov{\omega}_N} \EEtu\Ll[e^{-\lambda T_1}\Rr].$$
But, due to the stationarity of the uniform probability, we also have 
$$\EEtu\Ll[\int_0^{+\infty} e^{-\lambda s} \,\1_{X_s\sim 0}\, \d s\Rr]=\frac{N+1}{2^N} \frac 1\lambda.$$ 
Therefore 
$$\EEtu\Ll[e^{-\lambda T_1}\Rr] 
\le \frac{N+1}{2^N} \frac{\lambda+\ov{\omega}_N}\lambda.$$ 
Choosing $\lambda=1/t$ yields 
$$\EEtu\Ll[e^{- T_1/t}\Rr]
\le  \frac{N+1}{2^N} (1+\ov{\omega}_Nt).$$ 
Observe that $T_k$ stochastically dominates a sum of $k$ independent copies of $T_1$. 
Therefore 
$$\EEtu\Ll[e^{- T_k/t}\Rr]
\le \Ll( \frac{N+1}{2^N} (1+\ov{\omega}_Nt)\Rr)^k,$$ 
which leads to 
\begin{equation}
\label{Tkt}
\sum_{k\ge 1} {\PPtu[T_k\le t]}
\le 2e \frac{N+1}{2^N} (1+\ov{\omega}_Nt),
\end{equation}
whenever $\frac{N+1}{2^N} (1+\ov{\omega}_Nt)\le \frac 12$. 

On the other hand, $\EEtu[D_N(\msft)-D^{(0)}_N(\msft)]\le 2^N$. 
Thus we have 
$$
\EEtu\Ll[D_N(\msft)-D^{(0)}_N(\msft)\Rr]  $$ 
$$\le \Ll(2^N \1(\frac{N+1}{2^N} (1+\ov{\omega}_N \msft)\ge \frac 12) 
+\max_{x\sim 0} {\EEt_x\Ll[D_N({\mathsf T}_N)\Rr]} 2e \frac{N+1}{2^N} (1+\ov{\omega}_N \msft)\Rr).
$$  
We now have to square this last inequality and take the expectation with respect to $\P$. 
The $\P$ probability of the event $\frac{N+1}{2^N} (1+\ov{\omega}_N \msft)\ge \frac 12$ decays to $0$ at a super-exponential rate (see \eqref{e:ovomega2}), 
which compensates for the $2^N$ factor in front. 
We already noticed that $\max_{x\sim 0} {\EEt_x\Ll[D_N({\mathsf T}_N)\Rr]}$ is of subexponential order and 
the term $(1+\ov{\omega}_N \msft)$ is of the same order as $\msft$ (still on the  exponential scale). Thus we conclude 
that $\E\Ll[\EEtu[D_N(\msft)-D^{(0)}_N(\msft)]^2\Rr]$ is at most of order $(2^{-N}\msft)^2$ and therefore (\ref{eq:efronstein}) holds. 
\end{proof}

%
%
%
%
%
%%%%%%%%%%%%%%%%%%%%%%%%%%%%%%%%%%%%%%%%%%%%%%%%%%%%%%%%%%%%%%
%%%%%%%%%%%%%%%%%%%%%%%%%%%%%%%%%%%%%%%%%%%%%%%%%%%%%%%%%%%%%%
%
%
%
\section{Defining the scales}
\label{s:scales}
\setcounter{equation}{0}
We fix $\ov{c} \in (0,\log(2))$, and $\msft$ such that $\log(\msft) \sim \ov{c} N$. The sequence $\msft$ is our time scale for the accelerated dynamics $X$. Theorem~\ref{t:lln} associates to it a second scale $(\msfd)$ measuring the number of sites discovered, and we recall that $\log(\msfd) \sim \ov{c} N$.

It is a classical fact that 
\begin{equation}
\label{gaussrecall}
\P\Ll[E_x \ge z \Rr] \sim \frac{1}{z \sqrt{2\pi}} e^{-z^2/2} \qquad (z \to +\infty).
\end{equation}
Letting 
$$
b_N = (2 \log \msfd)^{1/2} - \frac{1}{2(2 \log \msfd)^{1/2}}\Ll(\log \log \msfd + \log(4\pi) \Rr),
$$
we obtain
\begin{equation}
\label{defpropbn} 
\P\Ll[E_x \ge b_N\Rr] \sim \msfd^{-1}. 
\end{equation}
Since by Proposition~\ref{p:iid}, the random variables $(E_{x_i})_{1 \le i \le \msfd}$ are i.i.d.\ under $\ov{\P}_N$, their maxima will be of the order of $b_N$. In order to measure the maximal values of $(\tau_{N,x_i})_{1 \le i \le \msfd}$, we thus define
\begin{equation}
\label{defBn}
B_N = e^{\beta \sqrt{N} b_N}.
\end{equation}
Using the fact that $\log(\msfd) \sim \ov{c} N$, one can show that for any $z > 0$,
\begin{equation}
\label{caractB}
\P\Ll[ \tau_{N,x} \ge z B_N \Rr] \sim \msfd^{-1} z^{-\alpha},
\end{equation}
where 
\begin{equation}
\label{defalpha}
\alpha = \frac{\sqrt{2 \ov{c}}}{\beta}.
\end{equation}
As a way to make it clear that the scale chosen is the correct one, we note for instance that
$$
\ov{\P}_N\Ll[\max_{1 \le i \le \msfd} \tau_{N,x_i} \le z B_N\Rr] \xrightarrow[N \to \infty]{} \exp(-z^{-\alpha}).
$$
From now on, we assume that 
$
\alpha < 1,
$
which corresponds to a low temperature assumption, and is the regime of aging. Actually, this assumption is only required at the very end of our proof, in order to show that the time spent out of deep traps is negligible. Specifically, we will need the following result.
\begin{prop}
\label{p:negligpasprofonds}
We assume $\alpha < 1$. We have
\begin{equation}
\label{e:negligpasprofonds}
\limsup_{N \to + \infty}\ \frac{\msfd}{B_N} \ \E\Ll[ \tau_{N,x} \ \1_{\tau_{N,x} \le \delta B_N} \Rr] = O(\delta^{1-\alpha}) \quad (\delta \to 0). 
\end{equation}
\end{prop}
\begin{proof}
This can be justified using the general results of \cite{bbm}, but we prefer to provide the reader with a more direct proof. We can write
$
\E\Ll[ \tau_{N,x} \ \1_{\tau_{N,x} \le \delta B_N} \Rr]
$
as
\begin{equation}
\label{negligpasint}
1/2 + \frac{1}{\sqrt{2\pi}} \int_0^{+\infty} e^{\beta \sqrt{N} u}  \ e^{-u^2/2} \ \1_{e^{\beta \sqrt{N} u} \le \delta B_N} \ \d u.
\end{equation}
The condition on the indicator function can be rewritten as
$$
u \le b_N + \frac{\log(\delta)}{\beta \sqrt{N}}.
$$
We let
$$
u_N = \frac{b_N}{\sqrt{N}} + \frac{\log(\delta)}{\beta N}.
$$
Note that $u_N$ tends to $\sqrt{2 \ov{c}}$ as $N$ tends to infinity. In \eqref{negligpasint}, the term $1/2$ can be neglected since the integral is bounded away from $0$ (or because one can easily check that $B_N/\msfd$ grows exponentially fast with $N$). By a change of variables, we can rewrite this integral as
\begin{equation}
\label{negligpasint2}
{\sqrt{N}} \int_0^{u_N} e^{N(\beta u - u^2/2)} \ \d u.
\end{equation}
We write $f(u) = \beta u - u^2/2$. 
%This function is increasing up to $u = \beta$. Since we assume $\alpha < 1$, it follows that $\sqrt{2\ov{c}} < \beta$. We can thus find a small $\eta > 0$ such that $f$ is increasing on $[0,\sqrt{2\ov{c}} + \eta]$, and for $N$ large enough, we have $u_N \in [0,\sqrt{2\ov{c}} + \eta]$. 
Since $f$ is concave, we have
$$
f(u) \le f(u_N) + (u-u_N) f'(u_N),
$$
and thus the quantity appearing in \eqref{negligpasint2} is bounded by
\begin{equation}
\label{negligpasint3}
{\sqrt{N}} \ e^{N f(u_N)} \int_0^{u_N} e^{N(u - u_N) f'(u_N)} \ \d u \le \frac{1}{f'(u_N) \sqrt{N}} \ e^{N f(u_N)}.
\end{equation}
The term $f'(u_N)$ converges to $f'(\sqrt{2\ov{c}}) > 0$. We need to compute $f(u_N)$: 
$$
N f(u_N) = \beta \sqrt{N} b_N + \log(\delta) -\frac{b_N^2}{2} - {\frac{b_N}{\beta\sqrt{N}}} \log(\delta) - \underbrace{\frac{\log^2(\delta)}{\beta N}}_{\to 0}.
$$
Recalling that $B_N = e^{\beta \sqrt{N} b_N}$, we obtain the the right-hand side of \eqref{negligpasint3} is bounded, up to a constant, by
$$
\frac{ B_N \ e^{-b_N^2/2}}{\sqrt{N}} \ \exp\Bigg( \underbrace{\Ll[1 - {\frac{b_N}{\beta\sqrt{N}}} \Rr]}_{\to 1-\alpha} \log(\delta) \Bigg).
$$
The result is thus proved provided
$$
\frac{ e^{-b_N^2/2}}{\sqrt{N}} = O(\msfd^{-1}),
$$
but this is a consequence of \eqref{gaussrecall} and \eqref{defpropbn}.
\end{proof}
\begin{rem}
A slightly more careful analysis reveals that actually,
$$
\lim_{N \to +\infty} \frac{\msfd}{B_N} \ \E\Ll[ \tau_{N,x} \ \1_{\tau_{N,x} \le \delta B_N} \Rr] = \frac{\alpha}{1-\alpha} \ \delta^{1-\alpha} .
$$	
\end{rem}

%
%
%
%
%
%%%%%%%%%%%%%%%%%%%%%%%%%%%%%%%%%%%%%%%%%%%%%%%%%%%%%%%%%%%%%%
%%%%%%%%%%%%%%%%%%%%%%%%%%%%%%%%%%%%%%%%%%%%%%%%%%%%%%%%%%%%%%
%
%
%
\section{The deep traps}
\label{s:deep}
\setcounter{equation}{0}
%\begin{thm}
%\label{t:main}
%Assume that $\ov{a} < 1/20$. As $N$ tends to infinity,
%$$
%\msfd/\msft \sim N^2 \phi(a_N)^2.
%$$
%Moreover, if $\alpha < 1$, then
%$$
%N^{2}\phi(a_N)^{2}B_N^{-1} \int_0^{t \msft} \tau_{N,X_s} \ \d s
%$$
%converges to an $\alpha$-stable subordinator.
%\end{thm}

Let $\delta > 0$. We say that $x \in V_N$ is a \emph{deep trap} if $\tau_{N,x} \ge \delta B_N$. Let $(r\de_N(n))_{n \in \N}$ be the increasing sequence spanning the set 
$$
\{i \in \N : \tau_{N,x_i} \ge \delta B_N\},
$$
with the convention that $r\de_N(n) = +\infty$ if $n$ exceeds the cardinality of this set. In other words, in the numbering defined by the exploration process, $r\de_N(1)$ is the index of the first deep trap discovered by the random walk, $r\de_N(2)$ is the index of the second one, and so on. When $r\de_N(n)$ is finite, we let $x\de_N(n)$ be the position of the $n^\text{th}$ trap discovered, that is, $x\de_N(n) = x_{r\de_N(n)}$. We set $x\de_N(n) = \infty$ if $r\de_N(n) = +\infty$. The quantity $\tau_{N,x\de_N(n)}$ denotes the depth of the $n^\text{th}$ trap discovered, with the understanding that $\tau_{N,x\de_N(n)} = \infty$ when $x\de_N(n) = \infty$. Let $T\de_N(n)$ be the instant when the $n^\text{th}$ trap is discovered:
$$
T\de_N(n) = \inf\{ t : D_N(t) \ge r\de_N(n)\} = \inf\{ t : x\de_N(n) \in \mcl{D}_N(t)\} \quad (+ \infty \text{ if empty}).
$$
\begin{prop}
\label{p:poisson}
\label{p:distribtrap}
Under $\ov{\P}_N$, the random variables 
$$
\Ll((\msft^{-1} \ T\de_N(n), \ B_N^{-1} \ \tau_{N,x\de_N(n)})\Rr)_{n \in \N}
$$ 
jointly converge in distribution as $N$ tends to infinity (in the sense of finite-dimensional distributions). Let us write 
$$
\Ll((T\de(n), \ \tau\de(n))\Rr)_{n \in \N}
$$ 
for limiting random variables, which we assume to be defined on the $\P$-probability space for convenience. Their joint distribution is described as follows. The family $\Ll(T\de(n)\Rr)_{n \in \N}$ is distributed as the jump instants of a Poisson process of intensity $\delta^{-\alpha}$, while $\Ll(\tau\de(n)\Rr)_{n \in \N}$ are i.i.d.\ with common distribution 
\begin{equation}
\label{defdistrib}
\alpha \delta^{\alpha} \frac{\d z}{z^{\alpha+1}} \ \1_{[\delta, +\infty)}(z).
\end{equation}
Moreover, the two families $\Ll(T\de(n)\Rr)_{n \in \N}$ and $\Ll(\tau\de(n)\Rr)_{n \in \N}$ are independent.
\end{prop}
\begin{proof}
Recall that as given by Proposition~\ref{p:iid}, the random variables $(\tau_{N,x_j})_{1 \le j  \le 2^N}$ are i.i.d. Let us write 
$$
p_N = \ov{\P}_N\Ll[ x_j \text{ deep trap} \Rr] = \ov{\P}_N\Ll[ \tau_{N,x_j} \ge \delta B_N \Rr].
$$
Recall from \eqref{caractB} that
\begin{equation}
\label{asymptpn}
\msfd p_N \xrightarrow[N \to  \infty]{} \delta^{-\alpha}.
\end{equation}
Let $F_n$ be the $\sigma$-algebra generated by $(\tau_{N,x_j})_{1 \le j  \le r\de_N(n)}$. For $k \in \N$ and $z \ge \delta$, we compute
$$
\ov{\P}_N\Ll[r\de_N(n) = k + r\de_N(n-1), \ B_N^{-1} \tau_{N,x\de_N(n)} \ge z \ \big| \ F_{n-1}\Rr].
$$
The event $r\de_N(n) = k + r\de_N(n-1)$ corresponds to the fact that for every $j$ such that $r\de_N(n-1) < j < k+r\de_N(n-1)$, the site $x_j$ is not a deep trap, while $x_{k+r\de_N(n-1)}$ is a deep trap. This event is independent of $F_{n-1}$ and has probability
$$
p_N(1-p_N)^{k-1}.
$$
Conditionally on $r\de_N(n) = k + r\de_N(n-1)$ and on $F_{n-1}$, the probability of the event 
$$
B_N^{-1} \tau_{N,x\de_N(n)} \ge z
$$
is nothing but
$$
\P\Ll[\tau_{N,0} \ge zB_N \ | \ \tau_{N,0} \ge \delta B_N\Rr].
$$
Hence, we obtain
\begin{multline*}
\ov{\P}_N\Ll[r\de_N(n) = k + r\de_N(n-1), \ B_N^{-1} \tau_{N,x\de_N(n)} \ge z \ \big| \ F_{n-1}\Rr] \\
= p_N(1-p_N)^{k-1}\ \P\Ll[\tau_{N,0} \ge zB_N \ | \ \tau_{N,0} \ge \delta B_N\Rr].
\end{multline*}
Summing over $k$ leads to
\begin{multline*}
\ov{\P}_N\Ll[r\de_N(n) > k + r\de_N(n-1), \ B_N^{-1} \tau_{N,x\de_N(n)} \ge z \ \big| \ F_{n-1}\Rr] \\
= (1-p_N)^{k}\ \P\Ll[\tau_{N,0} \ge zB_N \ | \ \tau_{N,0} \ge \delta B_N\Rr].
\end{multline*}
In view of \eqref{asymptpn} and of the fact that (see \eqref{caractB})
$$
\P\Ll[\tau_{N,0} \ge zB_N \ | \ \tau_{N,0} \ge \delta B_N\Rr]\xrightarrow[N \to \infty]{} \frac{\delta^\alpha}{z^\alpha},
$$
we obtain that under $\ov{\P}_N$, the random variables
\begin{equation}
\label{somereplacements}
\Ll((\msfd^{-1} \ r\de_N(n), \ B_N^{-1} \ \tau_{N,x\de_N(n)})\Rr)_{n \in \N}
\end{equation}
jointly converge in distribution to the random variables
\begin{equation}
\label{limit}
\Ll((T\de(n), \ \tau\de(n))\Rr)_{n \in \N}
\end{equation}
described in the proposition. To conclude, we need to replace the random variables $\msfd^{-1} \ r\de_N(n)$ by $\msft^{-1} T\de_N(n)$ in the convergence of \eqref{somereplacements} to \eqref{limit}.

The process $t \mapsto D_N(t)$ is increasing, right-continuous, and the size of its jumps cannot be larger than $(N+1)$ since the walk discovers at most $(N+1)$ sites at once. Hence,
$$
r\de_N(n) \le D_N(T\de_N(n)) \le r\de_N(n) + N + 1.
$$
Since $\msfd$ grows exponentially in $N$, it follows that in the convergence of \eqref{somereplacements} to \eqref{limit}, we can replace $\msfd^{-1} \ r\de_N(n)$ by $\msfd^{-1} \ D_N(T\de_N(n))$.

Let us see that for any $y > 0$, 
\begin{equation}
\label{convPois}
\ov{\P}_N\Ll[ \msft^{-1} \ T\de_N(1) \ge y \Rr] \xrightarrow[N \to \infty]{} e^{-y\delta^{-\alpha}}.
\end{equation}
For any $\eps > 0$, we have
\begin{multline*}
\ov{\P}_N\Ll[ T\de_N(1) \ge \msft(y + \eps) \Rr] \\
\le \ov{\P}_N\Ll[ D_N(T\de_N(1)) \ge \msfd y \Rr] + \ov{\P}_N\Ll[ D_N(\msft (y+\eps)) < \msfd y \Rr].
\end{multline*}
The first term tends to $e^{-y\delta^{-\alpha}}$ by the preceding observation, while the second tends to $0$ by Theorem~\ref{t:lln}. Hence,
\begin{equation}
\label{eqconvPois1}
\limsup_{N \to +\infty} \ov{\P}_N\Ll[ T\de_N(1) \ge \msft (y+\eps) \Rr] \le e^{-y\delta^{-\alpha}}.
\end{equation}
Similarly, one can show that
\begin{equation}
\label{eqconvPois2}
\liminf_{N \to +\infty} \ov{\P}_N\Ll[ T\de_N(1) \ge \msft (y-\eps) \Rr] \ge e^{-y\delta^{-\alpha}}.
\end{equation}
Inequalities \eqref{eqconvPois1} and \eqref{eqconvPois2} being valid for any $\eps > 0$, they justify \eqref{convPois}. The generalization of this identity to the convergence of finite-dimensional distributions is then only a matter of notations (the key point besides Theorem~\ref{t:lln}, as seen in the proof of \eqref{convPois}, being the monotonicity of $t \mapsto D_N(t)$).
\end{proof}

%
%
%
%
%%%%%%%%%%%%%%%%%%%%%%%%%%%%%%%%%%%%%%%%%%%%%%%%%%%%%%%%%%%%%%
%%%%%%%%%%%%%%%%%%%%%%%%%%%%%%%%%%%%%%%%%%%%%%%%%%%%%%%%%%%%%%
%
%
%
\section{The environment viewed by the particle}
\label{s:envviewed}
\setcounter{equation}{0}

Let $\Omega_N = [1,+\infty)^{V_N}$, so that $\tau_N$ is a random variable taking values in $\Omega_N$. The group $V_N$ naturally acts on $\Omega_N$ by translations: for $\tau \in \Omega_N$ and $x \in V_N$, we let $\theta_x \ \tau \in \Omega_N$ be such that
$$
\Ll(\theta_x \ \tau\Rr)_z = \tau_{x+z}.
$$
We define the \emph{environment viewed by the particle} as the process $(\tau_N(t))_{t \ge 0}$, taking values in $\Omega_N$ and such that
$$
\tau_N(t) = \theta_{X_t} \ \tau_N.
$$
\begin{prop}
\label{p:reversibility}
For any $\tau_N \in \Omega_N$, the process $(\tau_N(t))_{t \ge 0}$ is a reversible Markov chain under the measure $\PPtu$.
\end{prop}
\begin{proof}
The Markov property of this process is inherited from $(X_t)$. 	A minor adaptation of the proof of \cite[Lemma~4.3 (iv)]{dfgw} shows reversibility.
\end{proof}

%
%
%
%
%%%%%%%%%%%%%%%%%%%%%%%%%%%%%%%%%%%%%%%%%%%%%%%%%%%%%%%%%%%%%%
%%%%%%%%%%%%%%%%%%%%%%%%%%%%%%%%%%%%%%%%%%%%%%%%%%%%%%%%%%%%%%
%
%
%
\section{The environment around the first trap discovered}
\label{s:envaround}
\setcounter{equation}{0}
Let us write $\tau\de_N(n)$ for the environment seen at $x\de_N(n)$, that is,
$$
\tau\de_N(n) = \theta_{x\de_N(n)} \ \tau_N.
$$
The main purpose of this section is to understand the distribution of $\tau\de_N(1)$. Let
\begin{equation}
\label{defG}
G_N(\tau_N) = \EEt_0\Ll[\int_0^{+\infty}e^{-t/N^2} \  \1_{X_t = 0} \ \d t \Rr].
\end{equation}
The value of $\tau\de_N(1)$ at $0$ is described by Proposition~\ref{p:distribtrap}. Roughly speaking, we will show that as $N$ tends to infinity, the distribution of $\Ll((\tau\de_N(1))_z\Rr)_{z \neq 0}$ approaches the distribution proportional to
$$
\frac{1}{G_N(\tau_N)} \ \d \P(\tau_N).
$$
More precisely, we will prove the following result.
\begin{thm}
\label{t:Radon-Nikodym}
Let $\mcl{H}\de_N(n)$ be the event that $x\de_N(n)$ is visited by $X$ during the time interval $[T\de_N(n),T\de_N(n) + N]$. For $M > 0$, let $\mcl{E}^M_N$ be the event defined by
\begin{equation}
\label{defmclE}
T\de_N(1) \le M\msft \text{ and } \mcl{H}\de_N(1) \text{ holds}.
\end{equation}
For any $\eps > 0$, any $M$ large enough and any $N$ large enough, if $h_N : \Omega_N \to \R$ is such that $\|h_N\|_\infty \le 1$, then
\begin{equation}
\label{e:Radon-Nikodym}
\ov{\E}_N \Ll[(h_N G_N)(\tau\de_N(1)) , \ \mcl{E}^M_N \Rr] = \frac{\msft}{\msfd} \Ll(\E\Ll[h_N(\tau_N) \ | \ {\tau_{N,0} \ge \delta B_N} \Rr]  \pm \eps\Rr),
\end{equation}
where we recall that the notation $\pm$ was introduced at the end of Section~\ref{s:intro}.
\end{thm}
\begin{rem}
A notable feature of Theorem~\ref{t:Radon-Nikodym} is that as soon as $N$ is large enough, identity \eqref{e:Radon-Nikodym} holds for any function $h_N$ satisfying $\|h_N\|_\infty \le 1$. In particular, it is not required that, for instance, $h_N$ depend only on a finite number of $(\tau_{N,x})_{x \in V_N}$, or have any other form of spatial mixing property. This will make the proof of Theorem~\ref{t:Radon-Nikodym} slightly more involved than otherwise, but will also greatly simplify our subsequent reasoning.
\end{rem}
We now describe our starting point for the proof of Theorem~\ref{t:Radon-Nikodym}. Let us introduce $(\td{e}_N(n))_{n \in \N}$ i.i.d.\ exponential random variables of parameter $1/N^2$ (under $\PPt_x$ for any $x$), independent of everything else, and write $(\Theta_t)_{t \ge 0}$ to denote the time translations on the space of trajectories, that is, $(\Theta_t X)_s = X_{t+s}$. For any $M \ge 1$, we define $T^M_N(1) = T\de_N(1) \wedge (M\msft)$ (although this depends also on $\delta$, we keep it implicit in the notation, since $\delta$ will be kept fixed as long as we consider this quantity). Letting $\td{T}^M_N(0) = 0$, we then define recursively, for any $n \in \N$,
\begin{equation}
\label{deftdT0}
T^M_N(n) = \td{T}^M_N(n-1) + T^M_N(1) \circ \Theta_{\td{T}^M_N(n-1)},
\end{equation}
\begin{equation}
\label{deftdT}
\td{T}^M_N(n) = T^M_N(n) + \td{e}_N(n) + \msf{T}_N \circ \Theta_{T^M_N(n)+\td{e}_N(n)}.
\end{equation}
In words, to reach time $\td{T}^M_N(n)$, we wait until the end of a strong stationary time (as given by Proposition~\ref{p:mixing}) started at time $T^M_N(n) + \td{e}_N(n)$.
\begin{prop}
\label{p:stationarity}
Under the measure $\PPtu$, the successive pieces of trajectory
$$
\Ll( (X_s)_{\td{T}^M_N(n-1) \le s  < \td{T}^M_N(n)} \Rr)_{n \in \N}
$$
form a stationary sequence, while 
$$
\Ll( (X_s)_{\td{T}^M_N(n-1) \le s  < T^M_N(n)+\td{e}_N(n)} \Rr)_{n \in \N}
$$
are independent and identically distributed. %Moreover, under $\PPtu$, the random variables
%$$
%\Ll( \td{T}^M_N(n) - \td{T}^M_N(n-1) \Rr)_{n \in \N}
%$$
%are also independent and identically distributed.
\end{prop}
\begin{proof}
The first part is a consequence of the Markov property and the fact that $X_{\td{T}^M_N(n)}$ is distributed uniformly over $V_N$, itself a consequence of the fact that $\msf{T}_N$ is a stationary time.

For the second part, we only need to check the independence. Let $F$ be a bounded function on the space of trajectories which is measurable with respect to the $\sigma$-algebra generated by $(X_s)_{s < T^M_N(n)+\td{e}_N(n)}$, and $G$ be a bounded measurable function on the space of trajectories. We want to show that
\begin{equation}
\label{e:stationarity}
\EEtu\Ll[ F  \Ll(G \circ \Theta_{\td{T}^M_N(n+1)}\Rr) \Rr] = \EEtu[F] \ \EEtu[G].
\end{equation}
The left-hand side above is equal to
$$
\EEtu\Ll[F \ \EEt_{X_{T^M_N(n) + \td{e}_N(n)}}\Ll[G \circ \Theta_{\msf{T}_N}\Rr] \Rr],
$$
and this leads to \eqref{e:stationarity} by the Markov property and the fact that $\msf{T}_N$ is a stationary time. %The last part follows along the same line of reasoning, using the fact that $\msf{T}_N$ is a strong stationary time.
\end{proof}
Let $h_N : \Omega_N \to \R$ be a function. By the ergodic theorem, for any environment $\tau_N \in \Omega_N$,
\begin{multline}
\label{ergodicthm}
t^{-1} \  \EEtu\Ll[ \int_0^{t\msft} h_N(\tau_N(s)) \ \1_{\tau_{N,X_s}\ge \delta B_N} \ \d s \Rr] \\
 \xrightarrow[t \to \infty]{} \frac{\msft}{\EEtu\Ll[\td{T}^M_N(1)\Rr]} \EEtu\Ll[ \int_0^{\td{T}^M_N(1)} h_N(\tau_N(s)) \ \1_{\tau_{N,X_s}\ge \delta B_N} \ \d s\Rr].
\end{multline}
On the other hand, the stationarity of the environment viewed by the particle (that is, Proposition~\ref{p:reversibility}) guarantees that
$$
\EEtu\Ll[ \int_0^{t\msft} h_N(\tau_N(s)) \ \1_{\tau_{N,X_s}\ge \delta B_N} \ \d s \Rr] = t\msft \EEtu\Ll[  h_N(\tau_N(0)) \ \1_{\tau_{N,X_0}\ge \delta B_N}  \Rr].
$$
Combining this with \eqref{ergodicthm}, we thus obtain that, for any environment,
\begin{multline*}
\frac{1}{\EEtu\Ll[\td{T}^M_N(1)\Rr]} \EEtu\Ll[ \int_0^{\td{T}^M_N(1)} h_N(\tau_N(s)) \ \1_{\tau_{N,X_s}\ge \delta B_N} \ \d s\Rr] \\
= \EEtu\Ll[  h_N(\tau_N(0)) \ \1_{\tau_{N,X_0}\ge \delta B_N}  \Rr],
\end{multline*}
which we can rewrite as
\begin{multline}
\label{e:basicidentity}
\EEtu\Ll[ \int_0^{\td{T}^M_N(1)} h_N(\tau_N(s)) \ \1_{\tau_{N,X_s}\ge \delta B_N} \ \d s\Rr] \\ 
= {\EEtu\Ll[\td{T}^M_N(1)\Rr]} \ {\PPtu\Ll[\tau_{N,X_0}\ge \delta B_N\Rr]} \  \EEtu\Ll[  h_N(\tau_N(0)) \ | \ {\tau_{N,X_0}\ge \delta B_N}  \Rr] ,
\end{multline}
where the conditional expectation is understood to be $1$ in case the event $\tau_{N,X_0}\ge \delta B_N$ has $\PPtu$-probability equal to $0$ (that is, when $\tau_N$ contains no deep trap).
The rest of this section is devoted to passing from \eqref{e:basicidentity} to Theorem~\ref{t:Radon-Nikodym}. We want to take $\E$-expectations on both sides of this identity. Considering first the right-hand side, we prove that $\PPtu\Ll[\tau_{N,X_0}\ge \delta B_N\Rr]$ and $\EEtu\Ll[\td{T}^M_N(1)\Rr]$ are concentrated respectively around $\delta^{-\alpha} \msfd^{-1}$ and~$\delta^\alpha \msft$.
\begin{lem}
\label{l:concentration_of_P}
There exists $c > 0$ such that for any $\eps > 0$ and any $N$ large enough,
$$
\P\Ll[ \PPtu\Ll[\tau_{N,X_0}\ge \delta B_N\Rr] = (1\pm\eps)\delta^{-\alpha} \msfd^{-1} \Rr] \ge 1- \exp\Ll(-e^{cN}\Rr).
$$
\end{lem}
\begin{proof}
We can rewrite the probability $\PPtu\Ll[\tau_{N,X_0}\ge \delta B_N\Rr]$ as 
$$
|V_N|^{-1} \sum_{x \in V_N} \mfk{B}_{N,x},
$$
where $\mfk{B}_{N,x} = \1_{\tau_{N,x}\ge \delta B_N}$. Under $\P$, the $(\mfk{B}_{N,x})_{x \in V_N}$ are independent Bernoulli random variables with parameter $p_N = \P[\mfk{B}_{N,x} = 1] \sim \delta^{-\alpha} \msfd^{-1}$ by \eqref{caractB}. Lemma~\ref{l:concentration_of_P} thus reflects a classical large deviation estimate. More precisely, for any $\lambda > 0$,
$$
\P\Ll[ |V_N|^{-1} \sum_{x \in V_N} \mfk{B}_{N,x} > (1+\eps)\delta^{-\alpha} \msfd^{-1} \Rr] \le \frac{\E[e^{\lambda \mfk{B}_{N,0}}]^{|V_N|} }{ \exp\Ll(\lambda |V_N| (1+\eps) \delta^{-\alpha} \msfd^{-1}\Rr)},
$$
with
$$
\E[e^{\lambda \mfk{B}_{N,0}}]^{|V_N|} = (1-p_N + p_Ne^{\lambda})^{|V_N|} = \exp\Ll(|V_N|\log(1+p_N(e^\lambda -1))\Rr).
$$
Recalling that $p_N \sim \delta^{-\alpha} \msfd^{-1}$, we see that if $\lambda > 0$ is chosen small enough and $N$ is large, then
$$
\log(1+p_N(e^\lambda -1)) \le (1+\eps/2)\delta^{-\alpha}\msfd^{-1},
$$
and since $|V_N| \msfd^{-1} \ge e^{cN}$ for some $c > 0$, we obtain, for any $N$ large enough,
$$
\P\Ll[ |V_N|^{-1} \sum_{x \in V_N} \mfk{B}_{N,x} > (1+\eps)\delta^{-\alpha} \msfd^{-1} \Rr] \le \exp\Ll(-e^{cN}\Rr).
$$
The probability
$$
\P\Ll[ |V_N|^{-1} \sum_{x \in V_N} \mfk{B}_{N,x} < (1-\eps)\delta^{-\alpha} \msfd^{-1} \Rr]
$$ 
is handled in the same way.
\end{proof}
\begin{lem}
\label{l:concentration_of_ET}	
For any $\eps > 0$, any $M$ large enough and any $N$ large enough,
$$
\P\Ll[ \EEtu\Ll[\td{T}^M_N(1)\Rr] = (1 \pm \eps) \delta^\alpha \msft \Rr] \ge 1-\eps.
$$
\end{lem}
\begin{proof}
The difference $\td{T}^M_N(1) - T^M_N(1)$ consists in the sum of the random variables $\td{e}_N(1)$, which is exponential with mean $N^2$, and $\msf{T}_N$, which is of order $N$ by Proposition~\ref{p:mixing}. Since $\msft$ grows exponentially with $N$, it suffices to prove that the lemma holds with $T^M_N(1)$ instead of $\td{T}^M_N(1)$. 

We know from Proposition~\ref{p:poisson} that $\msft^{-1} T\de_N(1)$ converges in law under $\ov{\P}_N$ to $T\de(1)$, an exponential random variable of parameter $\delta^{-\alpha}$. Recall that $T^M_N(1)$ is defined as $T\de_N(1) \wedge (M\msft)$, and hence $\msft^{-1} T^M_N(1)$ converges in law to
\begin{equation}
\label{e:density}
%\delta^{-\alpha} e^{-\delta^{-\alpha} t} \ \1_{0 \le t \le M} \ \d t + c_\delta \mathbf{\delta}_M,
T\de(1) \wedge M,
\end{equation}
and is always bounded by $M$. As a consequence, for $M$ and $N$ large enough, we have
$$
\E \EEtu\Ll[ T^M_N(1) \Rr] = (1 \pm \eps/2) \delta^\alpha \msft.
$$
In order to conclude, it thus suffices to show that the variance of $\EEtu\Ll[ {T}^M_N(1) \Rr]$ is much smaller than $\msft^2$. The proof of this fact is similar to the end of the proof of Theorem~\ref{t:lln}, where we showed that the variance of $\EEtu[D_N(\msft)]$ is much smaller than $\msft^2$. 

Recall that we write $\V[\msf{X}]$ to denote the variance of the random variable $\msf{X}$ with respect to $\P$. By the Efron-Stein inequality and translation invariance, we have
$$
\V\Ll[ \EEtu[T^M_N(1)] \Rr] \le 2^{N-1} \ \E\Ll[  \Ll( \EEtu[T^M_N(1)] - \EEotu[T^M_N(1)] \Rr)^2 \Rr],
$$
where $\tau_N^{(0)}$ is the environment which coincides with $\tau_N$, except at $0$ where $E_0$ has been replaced by an independent copy of it. As in the proof of Theorem~\ref{t:lln}, we introduce $T_1$ the first time the point $0$ is discovered by the process $X$; we let $\msf{T}^{(1)}$ be the first strong stationary time after $T_1$. We iterate these definitions: $T_2$ is the first time $0$ is discovered after $\msf{T}^{(1)}$, then $\msf{T}^{(2)}$ is the first strong stationary time after $T_2$, and so on. We define $I = \cup [T_j,\msf{T}^{(j)})$, and 
$$
T_N^{(0)} = \inf \{t \ge 0,  t \notin I : \ \exists x \text{ deep trap s.t. } |X_t - x| \le 1\} \wedge (M\msft).
$$
This definition would coincide with that of $T^M_N(1)$ if the constraint $t\notin I$ was suppressed. The quantity $T_N^{(0)}$ is convenient for our purpose, because its law under $\PPtu$ does not depend on $\tau_{N,0}$, and is thus the same under $\PPtu$ and under $\PPotu$. We can thus write
$$
\EEtu[T^M_N(1)] - \EEotu[T^M_N(1)] = \EEtu[T^M_N(1) - T_N^{(0)}] - \EEotu[T^M_N(1) - T_N^{(0)}].
$$
Since these two terms have the same law under $\P$, it suffices for our purpose to show that 
\begin{equation}
\label{e:concent}
2^N \ \E\Ll[  \Ll( \EEtu[T^M_N(1) - T_N^{(0)}] \Rr)^2 \Rr] = o(\msft^2).
\end{equation}
Let $\mcl{T}$ be the event defined by
$$
\exists s \le \msf{T}_N, \ x \text{ deep trap s.t. } |X_s - x| \le 1,
$$
where $\msf{T}_N$ is the strong stationary time of Proposition~\ref{p:mixing}. 
By the Markov property, we have
$$
\EEtu[T^M_N(1) - T_N^{(0)}] \le M \msft \ \max_{x \sim 0} \PPt_x[\mcl{T}] \ \sum_{k \ge 1} \PPtu[T_k \le M\msft].
$$
We have seen in \eqref{Tkt} that the sum on the right-hand side gives a contribution which, up to sub-exponential factors, is bounded by $\msft/2^N$. There remains to evaluate 
$$
\E\Ll[ \Ll( \max_{x \sim 0} \PPt_x[\mcl{T}] \Rr)^2 \Rr] \le \E\Ll[ \Ll( \sum_{x \sim 0} \PPt_x[\mcl{T}] \Rr)^2 \Rr] \le N \sum_{x \sim 0} \E\Ll[\PPt_x[\mcl{T}] ^2 \Rr] \le N^2 \ \ov{\P}_N\Ll[\mcl{T}\Rr],
$$
where in the last step, we simply bounded the square of the probability by the probability itself, and then used translation invariance. We will show that, up to sub-exponential factors, the quantity above is bounded by $\msft^{-1}$, that is,
\begin{equation}
\label{e:concentstep}
\limsup_{N \to \infty} \frac{1}{N} \log\Ll( \ov{\P}_N\Ll[\mcl{T}\Rr] \Rr) \le \lim_{N \to \infty} \frac{1}{N} \log(\msft^{-1})  \quad (=-\ov{c}).	
\end{equation}
Assuming this to be true, we obtain that up to sub-exponential factors, the left-hand side of \eqref{e:concent} is bounded by
$$
2^N \frac{\msft^2}{\msft} \Ll(\frac{\msft}{2^N}\Rr)^2 = \msft^2 \ \frac{\msft}{2^N}.
$$
Since $\log(\msft) \sim \ov{c} N$ with $\ov{c} < \log(2)$, this proves \eqref{e:concent}, and thus completes the proof of the lemma.

There remains to justify \eqref{e:concentstep}. In view of Proposition~\ref{p:mixing}, the probability that $\msf{T}_N\ge N^3$ is super-exponentially small. Hence, this event can be discarded, and we can focus on bounding the probability of the event
\begin{equation}
\label{e:concevent}
\exists s \le N^3, \ x \text{ deep trap s.t. } |X_s - x| \le 1.
\end{equation}
When at a site $x$, the walk $X$ spends an exponentially distributed time of parameter $\omega_N(x) \le \ov{\omega}_N$. As was seen in Lemma~\ref{l:ovomega}, outside of an event whose probability is super-exponentially small, we have $\ov{\omega}_N \le e^{N^{3/4}}$. It thus follows that outside of an event whose probability is super-exponentially small, the number of sites visited by the walk up to time $N^3$ is bounded by $e^{N^{4/5}}$. Indeed, if $(\msf{e}_k)_{k \in \N}$ are independent random variables of parameter $e^{N^{3/4}}$, then a standard large deviation estimate shows that the probability of the event
$$
\sum_{k=1}^{e^{N^{4/5}}} \msf{e}_k \le N^3
$$
decays super-exponentially fast with $N$.
As a consequence, outside of an event whose $\ov{\P}_N$-probability is super-exponentially small, the walk does not discover more than $(N+1)e^{N^{4/5}} \le e^{N^{5/6}}$ sites up to time $N^3$. But we know from Proposition~\ref{p:iid} that under $\ov{\P}_N$, the sequence of depths of newly discovered traps is i.i.d. As a consequence, the probability that a deep trap is found among the $e^{N^{5/6}}$ first discovered sites is bounded by
$$
e^{N^{5/6}} \ \P[0 \text{ is a deep trap}] \sim \frac{e^{N^{5/6}}}{\delta^\alpha \msfd},
$$
see \eqref{caractB}. Recalling that $\log(\msfd) \sim \log(\msft)$, we obtain \eqref{e:concentstep}, and thus complete the proof.
\end{proof}
With the two lemmas above, we can now control the expectation of the right-hand side of \eqref{e:basicidentity}.
\begin{prop}
\label{p:expect_rhs}
For any $\eps > 0$ and any $N$ large enough, if $\|h_N\|_\infty \le 1$, then
\begin{multline}
\label{e:expect_rhs}
\E\Ll[ {\PPtu\Ll[\tau_{N,X_0}\ge \delta B_N\Rr]} \ {\EEtu\Ll[\td{T}^M_N(1)\Rr]} \ \EEtu\Ll[  h_N(\tau_N(0)) \ | \ {\tau_{N,X_0}\ge \delta B_N}  \Rr] \Rr] \\
= \frac{\msft}{\msfd}\Ll( \E\Ll[h_N(\tau) \ | \ {\tau_{N,0} \ge \delta B_N} \Rr]  \pm \eps \Rr).
\end{multline}
\end{prop}
\begin{proof}
Under the expectation on the left-hand side of \eqref{e:expect_rhs} appears a product of three terms, which for convenience we rewrite as $\msf{X}_N \msf{Y}_N \msf{Z}_N$, with obvious identifications. We first show that
\begin{equation}
\label{e:expectrhs1}
\E[\msf{X}_N \msf{Y}_N \msf{Z}_N] = (1\pm \eps) \delta^{-\alpha} \msfd^{-1} \E[\msf{Y}_N \msf{Z}_N] \pm 2 e^{-N}.
\end{equation}
Let $\mfk{E}_N$ be the event 
$$
\msf{X}_N = (1 \pm \eps) \delta^{-\alpha} \msfd^{-1},
$$
and $\mfk{E}_N^c$ be the complementary event. 
We decompose $\E[\msf{X}_N \msf{Y}_N \msf{Z}_N]$ as
$$
\E[\msf{X}_N \msf{Y}_N \msf{Z}_N, \mfk{E}_N] + \E[\msf{X}_N \msf{Y}_N \msf{Z}_N, \mfk{E}_N^c].
$$
Since $\msf{X}_N, \msf{Z}_N \le 1$, while $\msf{Y}_N \le M\msft$ and $\msft$ grows exponentially with $N$, the second term in the above sum is smaller than $e^{-N}$ by Lemma~\ref{l:concentration_of_P}. For the first term, the definition of $\mfk{E}_N$ implies that
$$
\E[\msf{X}_N \msf{Y}_N \msf{Z}_N, \mfk{E}_N] = (1\pm\eps)\delta^{-\alpha} \msfd^{-1} \E[\msf{Y}_N \msf{Z}_N, \mfk{E}_N].
$$
Reasoning as above, we can show that
$$
(1+\eps)\delta^{-\alpha} \msfd^{-1} \E[\msf{Y}_N \msf{Z}_N, \mfk{E}_N^c] \le e^{-N},
$$
so that
$$
\E[\msf{X}_N \msf{Y}_N \msf{Z}_N, \mfk{E}_N] = (1\pm\eps)\delta^{-\alpha} \msfd^{-1} \E[\msf{Y}_N \msf{Z}_N] \pm e^{-N},
$$
and this proves \eqref{e:expectrhs1}. 

Let us now see how a similar reasoning enables to show that
\begin{equation}
\label{e:exectrhs2}
\E[\msf{Y}_N \msf{Z}_N] = \delta^\alpha \msft \Ll( \E[\msf{Z}_N] \pm \eps \Rr).
\end{equation}
Define the event $\mfk{F}_N$ as
$$
\msf{Y}_N = (1\pm\eps) \delta^\alpha \msft.
$$
We decompose $\E[\msf{Y}_N \msf{Z}_N]$ as
$$
\E[\msf{Y}_N \msf{Z}_N, \mfk{F}_N] + \E[\msf{Y}_N \msf{Z}_N, \mfk{F}_N^c] = (1\pm\eps)\delta^\alpha \msft \E[\msf{Z}_N, \mfk{F}_N] + \E[\msf{Y}_N \msf{Z}_N, \mfk{F}_N^c].
$$
Since $\msf{Y}_N \le M\msft$ and $\msf{Z}_N \le 1$, Lemma~\ref{l:concentration_of_ET} ensures that
$$
0 \le \E[\msf{Y}_N \msf{Z}_N, \mfk{F}_N^c] \le M\msft \P[\mfk{F}_N^c] \le \eps M \msft.
$$
On the other hand,
$$
\E[\msf{Z}_N, \mfk{F}_N] = \E[\msf{Z}_N] - \E[\msf{Z}_N, \mfk{F}_N^c]
$$
with 
$$
0 \le \E[\msf{Z}_N, \mfk{F}_N^c] \le \eps.
$$
To sum up, we have shown that for any $\eps > 0$ and any $N$ large enough,
$$
\E[\msf{Y}_N \msf{Z}_N] = (1 \pm \eps)\delta^\alpha \msft\Ll(\E[\msf{Z}_N] \pm \eps\Rr) \pm \eps M \msft.
$$
Recalling that $\msf{Z}_N \le 1$, we obtain \eqref{e:exectrhs2} by suitably tuning $\eps$. Combining \eqref{e:expectrhs1} and \eqref{e:exectrhs2} yields
$$
\E[\msf{X}_N \msf{Y}_N \msf{Z}_N] = (1\pm\eps)\frac{\msft}{\msfd} \Ll(\E[\msf{Z}_N] \pm \eps\Rr) \pm e^{-N}.
$$
Since $\msf{Z}_N\le 1$ and $\log(\msfd/\msft) = o(N)$, this completes the proof as long as we can show that
\begin{equation}
\label{expecZ}
\E[\msf{Z}_N] = \E\Ll[h_N(\tau) \ | \ {\tau_{N,0} \ge \delta B_N} \Rr] \pm \eps.
\end{equation}
We have
$$
\E[\msf{Z}_N] = \E\Ll[ \frac{\EEtu\Ll[h_N(\tau_N(0))\ \1_{\tau_{N,X_0} \ge \delta B_N}\Rr]}{\PPtu[\tau_{N,X_0} \ge \delta B_N]} \Rr],
$$
where the fraction in the right-hand side above is understood to be $1$ when the probability in the denominator is $0$. We have seen in Lemma~\ref{l:concentration_of_P} (used with \eqref{caractB}) that with probability tending to $1$ as $N$ tends to infinity,
$$
\PPtu[\tau_{N,X_0} \ge \delta B_N] = (1\pm\eps) \P[\tau_{N,0} \ge \delta B_N],
$$
while
$$
\E\EEtu\Ll[h_N(\tau_N(0))\ \1_{\tau_{N,X_0} \ge \delta B_N}\Rr] = \E\Ll[h_N(\tau) \ \1_{\tau_{N,0} \ge \delta B_N} \Rr].
$$
Equation \eqref{expecZ} follows from these observations, and this finishes the proof.
\end{proof}

We now consider the left-hand side of \eqref{e:basicidentity}, and start with the following lemma.
\begin{lem}
\label{l:notwotraps}
Let $\mfk{D}^M_N$ be the event
\begin{equation}
\label{defmfkD}
\begin{array}{l}
 X \text{ discovers a (previously undiscovered) deep} \\
\text{trap during the time interval } (T^M_N(1),\td{T}^M_N(1)]
\end{array}
\end{equation}
There exists $c > 0$ such that for any $M$,
$$
\ov{\P}_N\Ll[
\mfk{D}^M_N
\Rr] \le e^{-cN}.
$$	
\end{lem}
\begin{proof}
Note first that the difference $\td{T}^M_N(1)-T^M_N(1)$ consists in the sum of the random variable $\td{e}_N(1)$ (exponentially distributed with mean $N^2$), and of a random variable distributed as $\msf{T}_N$ appearing in Proposition~\ref{p:mixing}. By part (2) of this Proposition, we thus know that outside of an event whose probability decays exponentially with $N$, we have
$$
\td{T}^M_N(1)-T^M_N(1) \le N^3.
$$
It is thus sufficient to prove the claim with the time interval $(T^M_N(1),\td{T}^M_N(1)]$ replaced by $(T^M_N(1),{T}^M_N(1) + N^3]$. What we need to do then is almost the same as  what we proved in the end of the proof of Lemma~\ref{l:concentration_of_ET}. Namely, the argument there ensures that outside of an event whose probability is super-exponentially small, the walk does not discover more than $e^{N^{5/6}}$ sites during this time interval. But we know from Proposition~\ref{p:iid} that under $\ov{\P}_N$, the sequence of depths of newly discovered traps is i.i.d. As a consequence, the probability that a deep trap is found in no more than $e^{N^{5/6}}$ newly discovered sites is bounded by
$$
e^{N^{5/6}} \ \P[0 \text{ is a deep trap}] \sim \frac{e^{N^{5/6}}}{\delta^\alpha  \msfd},
$$
see \eqref{caractB}. Since $\msfd$ grows exponentially with $N$, this finishes the proof.
\end{proof}

We consider the total time spent at site $x\de_N(1)$ from the moment of discovery $T\de_N(1)$ up to time $T\de_N(1) + \td{e}_N(1)$, that is,
\begin{equation}
\label{deflde}
l\de_N = \int_{T\de_N(1)}^{T\de_N(1)+\td{e}_N(1)} \1_{X_s = x\de_N(1)} \ \d s.
\end{equation}

The next proposition starts the computation of the expectation of the left-hand side of \eqref{e:basicidentity}.
\begin{prop}
\label{p:expect-lhs}
There exists $c > 0$ such that for any $M$ and any $N$ large enough, if $\|h_N\|_\infty \le 1$, then
\begin{multline*}
\ov{\E}_N\Ll[ \int_0^{\td{T}^M_N(1)} h_N(\tau_N(s)) \ \1_{\tau_{N,X_s}\ge \delta B_N} \ \d s\Rr] \\
= \ov{\E}_N\Ll[ h_N(\tau\de_N(1))\ l\de_N, \ T\de_N(1) \le M \msft \Rr] \pm e^{-cN}.
\end{multline*}
\end{prop}
\begin{proof}
Note that in the integral
$$
\int_0^{\td{T}^M_N(1)} h_N(\tau_N(s)) \ \1_{\tau_{N,X_s}\ge \delta B_N} \ \d s,
$$
the interval of integration can be replaced by $[T^M_N(1), \td{T}^M_N(1)]$, since the walk does not discover any deep trap before $T^M_N(1)$. Since we assume $\|h_N\|_\infty \le 1$, the integral is bounded by $\td{T}^M_N(1) - {T}^M_N(1)$, which for short we simply write $\Delta^M_N$.

Our first step is to estimate the difference 
\begin{multline}
\label{e:expect-lhs1}
\Bigg| \EEtu\Ll[ \int_0^{\td{T}^M_N(1)} h_N(\tau_N(s)) \ \1_{\tau_{N,X_s}\ge \delta B_N} \ \d s\Rr] \\
-  \EEtu\Ll[ \int_{T^M_N(1)}^{\td{T}^M_N(1)} h_N(\tau\de_N(1)) \ \1_{X_s = x\de_N(1)} \ \d s, \ T\de_N(1) \le M \msft \Rr] \Bigg|.
\end{multline}
In words, the second integral counts the time spent by the walk only on the first deep trap discovered, and only if this deep trap is discovered before time $M\msft$. The difference in \eqref{e:expect-lhs1} is thus bounded by
\begin{equation}
\label{e:expect-lhs1.5}
\EEtu\Ll[\Delta^M_N, \ \mfk{D}^M_N\Rr],
\end{equation}
where $\mfk{D}^M_N$ is the event defined in \eqref{defmfkD}. The $\E$-expectation of this term is bounded by
$$
N^3 \ov{\P}_N\Ll[\mfk{D}^M_N\Rr] + \ov{\E}_N\Ll[\Delta^M_N, \ \Delta^M_N \ge N^3\Rr].
$$
The first term of this sum is exponentially small in $N$ by Lemma~\ref{l:notwotraps}. So is also the second term, since $\Delta^M_N = \td{e}_N(1) + \msf{T}_N \circ \Theta_{T^M_N(1)+\td{e}_N(1)}$ and $\td{e}_N(1)$ is an exponential random variable with mean $N^2$, while $\msf{T}_N-N$ is stochastically dominated by an exponential random variable with mean $N$ by Proposition~\ref{p:mixing}. We have thus shown that there exists $c > 0$ such that for any $N$ large enough,
\begin{multline}
\label{e:expect-lhs2}
\ov{\E}_N\Ll[ \int_0^{\td{T}^M_N(1)} h_N(\tau_N(s)) \ \1_{\tau_{N,X_s}\ge \delta B_N} \ \d s\Rr] \\
=  \ov{\E}_N\Ll[ \int_{T^M_N(1)}^{\td{T}^M_N(1)} h_N(\tau\de_N(1)) \ \1_{X_s = x\de_N(1)} \ \d s, \ T\de_N(1) \le M \msft \Rr]  \pm e^{-cN}.
\end{multline}
On the event $T\de_N(1)\le M \msft$, we have $T^M_N(1) = T\de_N(1)$. By the definition of $l\de_N$ (see \eqref{deflde}), on this event, we have
\begin{equation}
\label{e:expect-lhs2.5}
\int_{T\de_N(1)}^{T\de_N(1) +\td{e}_N(1)} h_N(\tau\de_N(1)) \ \1_{X_s = x\de_N(1)} \ \d s = h_N(\tau\de_N(1)) \ l\de_N.
% = (h_N G_N)(\tau\de_N(1)) \ e\de_N(1).
\end{equation}
We now wish to show that for some $c > 0$ and $N$ large enough,
\begin{equation}
\label{e:expect-lhs3}
\ov{\E}_N\Ll[ \int_{T\de_N(1)+\td{e}_N(1)}^{\td{T}^M_N(1)} h_N(\tau\de_N(1)) \ \1_{X_s = x\de_N(1)} \ \d s, \ T\de_N(1) \le M \msft \Rr] \le e^{-cN}.
\end{equation}
Since $\|h_N\|_\infty \le 1$, we may as well show that
$$
\ov{\E}_N\Ll[ \int_{T\de_N(1)+\td{e}_N(1)}^{\td{T}^M_N(1)} \1_{X_s = x\de_N(1)} \ \d s, \ T\de_N(1) \le M \msft \Rr] \le e^{-cN}.
$$
In order to do so, let us define $\mfk{D}'_N$ the event 
$$
\td{e}_N(1) < N \text{ or } \Delta^M_N \ge N^3.
$$
Arguing as above, we see that the probability of this event is exponentially small in $N$. Proceeding as in the analysis of the term \eqref{e:expect-lhs1.5}, it then follows that 
$$
\ov{\E}_N\Ll[ \int_{T\de_N(1)+\td{e}_N(1)}^{\td{T}^M_N(1)}  \1_{X_s = x\de_N(1)} \ \d s, \ T\de_N(1) \le M\msft , \ \mfk{D}'_N \Rr] \le e^{-cN}.
$$
In order to prove \eqref{e:expect-lhs3}, it thus suffices to show that 
$$
\ov{\E}_N\Ll[  \int_{T\de_N(1)+N}^{T\de_N(1)+N^3} \1_{X_s = x\de_N(1)} \ \d s \Rr] \le e^{-cN}.
$$
Recall that as $x\de_N(1)$ is discovered at time $T\de_N(1)$, it must be that $|X_{T\de_N(1)} - x| \le 1$. Hence,
$$
\EEtu\Ll[\int_{T\de_N(1)+N}^{T\de_N(1)+N^3} \1_{X_s = x\de_N(1)} \ \d s \Rr] \le \EEtu\Ll[\int_{T\de_N(1)+N}^{T\de_N(1)+N^3} \1_{|X_s - X_{T\de_N(1)}| \le 1} \ \d s \Rr].
$$
By the Markov property, the latter integral is equal to
$$
\EEtu\Ll[ \int_N^{N^3} \underbrace{\PPt_{X_{T\de_N(1)}}\Ll[|X_s - X_0| \le 1\Rr]}  \Rr].
$$
By Proposition~\ref{p:fuite-sg}, the probability underbraced above is exponentially small in $N$, uniformly over $\tau_N$, $X_{T\de_N(1)}$ and $s \ge N$. This completes the proof of \eqref{e:expect-lhs3}. Equations \eqref{e:expect-lhs2}, \eqref{e:expect-lhs2.5} and \eqref{e:expect-lhs3} imply the proposition.
\end{proof}

We can now conclude our analysis of the expectation of the left-hand side of identity \eqref{e:basicidentity}.
\begin{prop}
\label{p:expect-lhs-bis}
There exists $c > 0$ such that for any $M$ and any $N$ large enough, if $\|h_N\|_\infty \le 1$, then
$$
\ov{\E}_N\Ll[ h_N(\tau\de_N(1))\ l\de_N, \ T\de_N(1) \le M \msft \Rr] = \ov{\E}_N\Ll[ (h_N G_N)(\tau\de_N(1)), \ \mcl{E}^M_N \Rr] \pm e^{-cN},
$$
where we recall that the event $\mcl{E}^M_N$ was defined in \eqref{defmclE}.
\end{prop}
\begin{proof}
First, note that
\begin{multline}
\label{e:bis1}
\ov{\E}_N\Ll[ h_N(\tau\de_N(1))\ l\de_N, \ T\de_N(1) \le M \msft \Rr] \\ = \ov{\E}_N\Ll[  h_N(\tau\de_N(1))\ l\de_N, \ \mcl{E}^M_N \Rr] \pm e^{-cN}.
\end{multline}
Indeed, the quantities under the two expectations above are equal unless the deep trap $x\de_N(1)$ is visited in the time interval $(T\de_N(1)+N, T\de_N(1)+\td{e}_N]$. Hence \eqref{e:bis1} is obtained as in the proof of Proposition~\ref{p:expect-lhs}, noting that the probability of this event is exponentially small in $N$ by Proposition~\ref{p:fuite-sg}.

Let us write $H\de_N(n)$ for the first time the walk visits the site $x\de_N(n)$, and define
\begin{equation}
\label{deftdl}
\td{l}\de_N(n) = \int_{H\de_N(n)}^{H\de_N(n)+\td{e}_N(n)} \1_{X_s = x\de_N(n)} \ \d s.
\end{equation}
By a similar reasoning, we have
\begin{equation}
\label{e:bis2}
\ov{\E}_N\Ll[  h_N(\tau\de_N(1))\ l\de_N, \ \mcl{E}^M_N \Rr] = \ov{\E}_N\Ll[  h_N(\tau\de_N(1))\ \td{l}\de_N(1), \ \mcl{E}^M_N \Rr] \pm e^{-cN}.
\end{equation}
Indeed, the quantities under the two expectations above are equal unless the deep trap $x\de_N(1)$ is visited during the time interval $(T\de_N(1) + \td{e}_N(1) ; H\de_N(1) + \td{e}_N(1)]$. On the event $\mcl{E}^M_N$, we have that the length of this interval is smaller than $N$. Moreover, outside of an event whose probability is exponentially small, we have $\td{e}_N(1) \ge N$. So we can argue as above and obtain \eqref{e:bis2}.

Let $e\de_N(n)$ be such that 
\begin{equation}
\label{defede}
\td{l}\de_N(n) = G_N(\tau\de_N(n)) \ {e}\de_N(n).
\end{equation}
Under the measure $\PPtu$ and conditionally on $x\de_N(1) = x$, the quantity $\td{l}\de_N(1)$ is independent of $(X_t)_{t \le H\de_N(1)}$. Moreover, it is the total time spent on $x$ by the walk $X$ starting at $x$ and with a killing parameter equal to $1/N^2$. In particular, $\td{l}\de_N(1)$ is an exponential random variable with mean $G_N(\tau\de_N(1))$. In other words, we thus see that under the measure $\PPtu$ and conditionally on $x\de_N(1) = x$, the quantity ${e}\de_N(1)$ is independent of $(X_t)_{t \le H\de_N(1)}$ and follows an exponential distribution with parameter~$1$. Since this does not depend on $x$, we have thus shown that under the measure $\PPtu$, the quantity ${e}\de_N(1)$ is independent of $(X_t)_{t \le H\de_N(1)}$ and follows an exponential distribution with parameter $1$. In particular, (and since $\mcl{E}^M_N$ is measurable with respect to $(X_t)_{t \le H\de_N(1)}$), we have
\begin{eqnarray*}
\EEtu\Ll[  h_N(\tau\de_N(1))\ \td{l}\de_N(1), \ \mcl{E}^M_N  \Rr] & = & \EEtu\Ll[  (h_N G_N)(\tau\de_N(1))\ \td{e}\de_N(1), \ \mcl{E}^M_N  \Rr] \\
& = & \EEtu\Ll[  (h_N G_N)(\tau\de_N(1)), \ \mcl{E}^M_N  \Rr].
\end{eqnarray*}
Taking expectations and using \eqref{e:bis2}, we obtain the announced result. 
\end{proof}

\begin{proof}[Proof of Theorem~\ref{t:Radon-Nikodym}]
The theorem follows by combining the identity \eqref{e:basicidentity} with Propositions~\ref{p:expect_rhs}, \ref{p:expect-lhs} and \ref{p:expect-lhs-bis}, and recalling that $\log(\msfd/\msft) = o(N)$.	
\end{proof}

Of course, Theorem~\ref{t:Radon-Nikodym} is truly interesting only if the probability of $\mcl{E}^M_N$ can be taken arbitrarily close to $1$. It follows from Proposition~\ref{p:poisson} that $\ov{\P}_N[T\de_N(1) \le M \msft]$ can be taken arbitrarily close to $1$ by choosing $M$ large enough. Our goal is now to prove that $\mcl{H}\de_N(n)$ occurs with high probability.
\begin{prop}
\label{p:gotouch}
For any $n \in \N$, 
$$
\ov{\P}_N\Ll[\mcl{H}\de_N(n)\Rr] \xrightarrow[N \to \infty]{} 1.
$$
\end{prop}
Let us write $x \ssim y$ when $x,y \in V_N$ are neighbours or second neighbours. The proof of the proposition relies on the following lemma.
\begin{lem}
\label{l:tight}
For any $n \in \N$, 
$$
\ov{\P}_N\Ll[ \max_{z \ssim x\de_N(n)} E_{z} \le 2 \sqrt{\log N}  \Rr] \xrightarrow[N \to \infty]{} 1.
$$
\end{lem}
\begin{proof}
We introduce some terminology (slightly adapted from \cite{scaling}). For any $z \ssim 0$, we say that $x \in V_N$ is $z$-\emph{atypical} if it is a deep trap (that is, $\tau_{N,x} \ge \delta B_N$) and moreover, $E_{x+z} > 2 \sqrt{\log N}$. The lemma can be rephrased as stating that
\begin{equation*}
%\label{defprob}
\ov{\P}_N\Ll[ x\de_N(n) \text{ is } z\text{-atypical for some } z \ssim 0 \Rr]
\end{equation*}
tends to $0$ as $N$ tends to infinity.

We say that $x \in V_N$ is $z$-\emph{uncommon} if it is $z$-atypical, or if $(x-z)$ is atypical. Finally, for a subset $\Gamma \subset V_N$, we say that $x$ is $z$-\emph{uncommon regardless of} $\Gamma$ if one can infer that $x$ is $z$-uncommon without considering sites inside $\Gamma$, i.e. if one of the two following conditions occur~:
$$
x \text{ is } z\text{-atypical and } \{x,x+z\} \cap \Gamma = \emptyset,  \text{ or }  (x-z) \text{ is } z\text{-atypical and } \{x-z,x\} \cap \Gamma = \emptyset.
$$
We learn from \cite[Lemma~7.2]{scaling} that if $x\de_N(n)$ is $z$-atypical, then there exists $k \le r\de_N(n)$ such that $x_k$ is $z$-uncommon regardless of $\{x_1,\ldots,x_{k-1}\}$ (recall that $x_1,x_2,\ldots$ denotes the exploration process introduced in Section~\ref{s:explore}). As was seen in the proof of Proposition~\ref{p:poisson}, the random variable $\msfd^{-1} r\de_N(n)$ converges in distribution under $\ov{\P}_N$, hence the probability
\begin{equation}
\label{e:Cr}
\ov{\P}_N\Ll[ r\de_N(n) \le C_r \msfd \Rr]
\end{equation}
can be made as close to $1$ as desired by choosing $C_r$ sufficiently large. Moreover,
\begin{multline}
\label{e:tight}
\ov{\P}_N\Ll[ x\de_N(n) \text{ is } z\text{-atypical for some } z \ssim 0, r\de_N(n) \le C_r \Rr] \\
\le \sum_{z \ssim 0} \sum_{k = 1}^{C_r \msfd} \underbrace{\ov{\P}_N\Ll[x_k \text{ is } z\text{-uncommon regardless of } \{x_1,\ldots,x_{k-1}\} \Rr]}.
\end{multline}
As was observed in the proof of \cite[Proposition~7.1]{scaling}, the invariance of $\P$ under translations implies that the probability underbraced above is bounded by
$$
{\P}[0 \text{ is } z\text{-uncommon}],
$$
which in turn is bounded by
$$
2\  \ov{\P}_N\Ll[0 \text{ is } z\text{-atypical}\Rr] \le 2 \ \ov{\P}_N\Ll[\tau_{N,0} \ge \delta B_N\Rr] \ \ov{\P}_N\Ll[ E_{z} > 2 \sqrt{\log N} \Rr].
$$
Using \eqref{caractB} and \eqref{gaussrecall}, we obtain that for $N$ large enough, the probability on the left-hand side of \eqref{e:tight} is bounded by
$$
8 C_r N^2 \delta^{-\alpha} \frac{1}{2\sqrt{2\pi \log N}}e^{-(2\sqrt{\log N})^2 /2},
$$
which tends to $0$ as $N$ tends to infinity.
\end{proof}
\begin{proof}[Proof of Proposition~\ref{p:gotouch}]
Let us condition on the event that the position of the walk at time $T\de_N(n)$	is $x \in V_n$. Note that in this case, $x\de_N(n)$ must be a neighbour of $x$. On this event, the $\PPtu$-probability that the next jump of the walk is from $x$ to $x\de_N(n)$ is given by
$$
\frac{\exp\Ll(a_N E_{x\de_N(n)}\Rr)}{
\displaystyle{
\sum_{y \sim x}\exp\Ll(a_N E_{y}\Rr)
}
} = \frac{\exp\Ll(a_N E_{x\de_N(n)}\Rr)}{
\displaystyle{
\exp\Ll(a_N E_{x\de_N(n)}\Rr) + \sum_{y \sim x, y \neq x\de_N(n)}\exp\Ll(a_N E_{y}\Rr)
}
}.
$$
Recall that $a_N \ge 1$ and $a_N  \le \ov{a}\sqrt{2 \log N}$. Since $x\de_N(n)$ is a deep trap, the numerator above is larger than $\exp\Ll(\sqrt{2\ov{c} N}(1+o(1))\Rr)$, while the sum appearing in the denominator has no more than $N$ terms, each of them being smaller than $\exp\Ll(2\sqrt{2} \ov{a}\log N\Rr)$ on the event that
$$
\max_{z \ssim x\de_N(n)} E_{z} \le 2 \sqrt{\log N}.
$$
Using Lemma~\ref{l:tight}, we can thus conclude that with probability tending to $1$, the walk jumps to $x\de_N(n)$ right after having discovered it. The total time spent on any site is an exponential random variable whose parameter is at least $N$, so this jump occurs fast enough to ensure that Proposition~\ref{p:gotouch} holds.
\end{proof}

%
%
%
%
%%%%%%%%%%%%%%%%%%%%%%%%%%%%%%%%%%%%%%%%%%%%%%%%%%%%%%%%%%%%%%
%%%%%%%%%%%%%%%%%%%%%%%%%%%%%%%%%%%%%%%%%%%%%%%%%%%%%%%%%%%%%%
%
%
%
\section{The Green function at the origin}
\label{s:green}
\setcounter{equation}{0}
Let us write $\P^\circ = \P[ \ \cdot \ | \ \tau_{N,0} \ge \delta B_N]$, and $\ov{\P}_N^\circ = \P^\circ\PPt_0$. The purpose of this section is to prove the following result.
\begin{thm}
\label{t:green}
Under the measure $\P^\circ$, the random variable $N^2\phi(a_N)^2 \ G_N(\tau_N)$ converges to $1$ in probability.
\end{thm}
The idea of the proof consists in (1) computing the expected time the walk spends on the origin before exiting its $1$-neighbourhood, and (2) showing that once the walk has exited the $1$-neighbourhood of the origin, it will not come back to it sufficiently fast to be accounted in $G_N(\tau_N)$. The most technical part is (2), which is achieved in two main steps. In the first step (Proposition~\ref{p:descent}), we show that there exists $\zeta > 0$ such that with probability tending to $1$, the walk does not come back to the origin up to time $N^{\zeta - 1}$. In the second step, we rely on Proposition~\ref{p:fuite} to justify that the walk will not come back to the origin soon enough after $N^{\zeta - 1}$. This is slightly delicate, since the statement in Proposition~\ref{p:fuite} is averaged over $\P$, while we consider here an average over $\P^\circ$. 

For any $n \in \N$, we let $H_n = \inf\{ t \ge 0 : |X_t - X_0| = n\}$.

\begin{lem}
\label{l:timeH2}
Under $\PPt_x$, the total time spent at the initial position up to $H_2$:
$$
\int_0^{H_2} \1_{X_s = X_0} \ \d s
$$
follows an exponential distribution with parameter
\begin{equation}
\label{e:timeH2}
e^{a_N E_x} \sum_{y \sim x} e^{a_N E_y} \ \frac{
\displaystyle{
\sum_{y \sim x}e^{a_N E_y} \sum_{z \sim y ,z \neq x} e^{a_N E_z} 
}}
{\displaystyle{
\sum_{y \sim x} e^{a_N E_y} \Ll( e^{a_N E_x} + \sum_{z \sim y, z \neq x} e^{a_N E_z} \Rr)
}}.
\end{equation}
\end{lem}
\begin{proof}
Under $\PPt_x$, the walk stays at the origin an exponential time of parameter 
$$
e^{a_N E_x} \sum_{y \sim x} e^{a_N E_y}.
$$
Then it jumps to some $y \sim x$ with probability proportional to $e^{a_N E_y}$. The fraction in \eqref{e:timeH2} is thus the probability that the walk hits a point at distance $2$ from $x$ before returning to $x$. The result then follows from these observations.
\end{proof}
\begin{prop}
\label{p:estimatyp}
Recall that we assume $\ov{a} < 1/20$. For any $\eps > 0$, we have
\begin{equation}
\label{e:estimatyp2}
{\P}\Ll[\sum_{x \sim 0} e^{a_N E_x} \sum_{{y \sim x , y \neq 0}} e^{a_N E_y} = (1\pm \eps) N^2 \phi(a_N)^2 \Rr] = 1 - o(N^{-2}).
\end{equation}
\end{prop}
\begin{rem}
For the purpose of this section, it would be sufficient to show that the probability in the left-hand side of \eqref{e:estimatyp2} tends to $1$ as $N$ tends to infinity. The quantitative estimate given here will prove important in the next section.
\end{rem}
\begin{proof}[Proof of Proposition~\ref{p:estimatyp}]
We first prove that for any $\eps > 0$,
\begin{equation}
\label{e:estimatyp}
\P\Ll[\sum_{x \sim 0} e^{a_N E_x} = (1\pm \eps) N \phi(a_N) \Rr] = 1 - o(N^{-3}),
\end{equation}
and we do so by a computation of moments. We write $\chi_N(x)$ for $e^{a_N E_x} - \E[e^{a_N E_x}] = e^{a_N E_x} - \phi(a_N)$. For any integer $m$, we have
\begin{equation}
\begin{split}
\label{e:estimmoments}
& \E\Ll[ \Ll( \sum_{x\sim 0} \chi_N(x) \Rr)^{2m} \Rr] \\
& \qquad  = \sum_{x_1,\ldots, x_{2m} \sim 0} \E\Ll[\chi_N(x_1) \cdots \chi_N(x_{2m})\Rr] \\
& \qquad = \sum_{k = 1}^{2m} \sum_{\substack{e_1+\cdots+e_{k} = 2m \\ e_i \ge 1}} C_{e_1, \ldots, e_k} \sum_{ \substack{ y_1, \ldots, y_k \sim 0 \\ y_i \neq y_j } } \E\Ll[\chi_N(y_1)^{e_1} \cdots \chi_N(y_k)^{e_k}\Rr].
\end{split}
\end{equation}
The last equality is obtained in the following way. For every sequence $(x_1,\ldots, x_{2m})$, we let $k$ be the cardinality of $\{x_1,\ldots,x_{2m}\}$, and write this set as $\{y_1,\ldots, y_k\}$, where the $y_i$'s are pairwise distinct. Then $e_i$ $(\ge 1)$ is the number of occurrences of $y_i$ in the sequence $(x_1,\ldots, x_{2m})$. It is an interesting combinatorial exercise to see that the coefficient $C_{e_1,\ldots,e_k}$ appearing in the formula is the multinomial coefficient associated with $(e_1,\ldots, e_k)$ divided by $k!$, but the only important thing for us is that this coefficient does not depend on $N$. Since $(\chi_N(x))_{x \sim 0}$ are independent and centred random variables, the expectation $\E\Ll[\chi_N(y_1)^{e_1} \cdots \chi_N(y_k)^{e_k}\Rr]$ vanishes as soon as one of the $e_i$'s is equal to $1$. Hence, up to a constant (depending only on $m$), the left-hand side of \eqref{e:estimmoments} is bounded by
$$
\sum_{k = 1}^m \ \sum_{\substack{ e_1+\cdots+e_{k} = 2m \\ e_i \ge 2 } }  \ \sum_{ \substack{ y_1, \ldots, y_k \sim 0 \\ y_i \neq y_j } } \ \prod_{i=1}^k \Ll|\E\Ll[\chi_N(y_i)^{e_i}\Rr]\Rr|.
$$
Up to a constant, the quantity above is bounded by
$$
\sum_{k = 1}^m \ \sum_{\substack{ e_1+\cdots+e_{k} = 2m \\ e_i \ge 2 } }  \ N^k \prod_{i = 1}^k \phi(a_N e_i) \le 2^m \sum_{k = 1}^m \ \sum_{\substack{ e_1+\cdots+e_{k} = 2m \\ e_i \ge 2 } }  \ N^k \exp\Ll(\frac{a_N^2}{2} \sum_{i = 1}^k e_i^2\Rr),
$$
where we used the fact that $\phi(\lambda) \le 2 e^{\lambda^2/2}$ (see \eqref{defphi}). Since the square function is convex and takes value $0$ at $0$, we have
$$
\sum_{i = 1}^k e_i^2 \le \Ll( \sum_{i = 1}^k e_i \Rr)^2 = 4 m^2.
$$
We have thus shown that for any $m$, there exists a constant $C$ such that 
$$
\E\Ll[ \Ll( \sum_{x\sim 0} \chi_N(x) \Rr)^{2m} \Rr] \le C N^m \exp\Ll(2 a_N^2 m^2\Rr).
$$
Choosing $m = 4$ and recalling that $\ov{a} < 1/20$, we obtain
$$
\E\Ll[ \Ll( \sum_{x\sim 0} \chi_N(x) \Rr)^{8} \Rr] = o(N^{5}).
$$
An application of Chebyshev's inequality then proves \eqref{e:estimatyp}.

For any $\eps > 0$, we consider the event formed by the conjunction of 
$$
\sum_{x \sim 0} e^{a_N E_x} = (1\pm \eps) N \phi(a_N)
$$
and for any $x \sim 0$,
$$
\sum_{{y \sim x, y \neq 0}} e^{a_N E_y} = (1\pm \eps) N \phi(a_N).
$$
It follows from \eqref{e:estimatyp} that this event has probability at least $1 - o(N^{-2})$. On this event, we have
$$
\sum_{x \sim 0} e^{a_N E_x} \sum_{{y \sim x, y \neq 0}} e^{a_N E_y} = (1\pm \eps)^2 N^2 \phi(a_N)^2,
$$
and thus the result is proved.
\end{proof}
\begin{rem}
Although the proof makes it transparent that the condition $\ov{a} < 1/20$ is not really necessary, one can convince oneself that Proposition~\ref{p:estimatyp} does not hold for every $\ov{a} < 1$.
\end{rem}

\begin{prop}
\label{p:descent}
Let $\zeta > 0$, and let $\mfk{A}_N$ be the event
\begin{equation}
\label{defmclA}
\inf\{ t\ge H_2 : |X_t-X_0| \le 1 \} \ge H_2 + N^{\zeta-1}.
\end{equation}
Recall that we assume $\ov{a} < 1/20$. If $\zeta$ is sufficiently small, then 
$$
\ov{\P}_N^\circ\Ll[ \mfk{A}_N\Rr] \xrightarrow[N \to \infty]{} 1.
$$
\end{prop}
\begin{proof}
We will first show that with probability tending to $1$, the walk starting from time $H_2$ keeps going away from the origin at least up to distance $N^{1/4}$. We then show that a good proportion of these steps take time at least $N^{-1-5\ov{a}}$, thus justifying that on this event, the walk has not returned to the $1$-neighbourhood of the origin up to time $N^{-3/4-5\ov{a}}$. This will prove the proposition since we assume $\ov{a} < 1/20$.

We begin by introducing $(\ov{Y}_n)_{n \ge 2}$ the discrete-time walk that records the successive positions of $(X_t)_{t \ge H_2}$. The indexation is chosen so that $\ov{Y}_2$ is at distance $2$ from the origin (that is, $|\ov{Y}_2| = 2$). We say that the walk \emph{goes down to level} $n$ if $|\ov{Y}_n| = n$. Note that since the walk is nearest-neighbour, if it goes down to level $n$, then it goes down to level $k \ge 2$ for any $k \le n$.

We also introduce a slightly modified version of the exploration process. This new exploration process, that we denote by $\ov{x}_1, \ov{x}_2, \ldots$, is obtained from the original one by suppressing all sites belonging to $\{x : |x| \le 2\}$. Note that each step of $\ov{Y}$ gives rise to a (possibly empty) set of newly discovered sites (for this modified exploration process). For $n \ge 2$, we write $\ov{\mathcal{X}}_{n+1}$ for the set of sites thus discovered by $\ov{Y}$ at (discrete) time $n$ (see Figure~\ref{f:tree}). For $n \in \{1,2\}$, we also define
$$
\ov{\mcl{X}}_n = \{x \in V_N : |x| = n\}.
$$
Under $\ov{\P}_N^\circ$, the distribution of $\ov{Y}$ up to time $n$ does not depend on $(E_{\ov{x}_i})_{i \in \ov{\mcl{X}}_{n+1}}$, and therefore the latter is a family of independent random variables distributed as $\mcl{N}_+(0,1)$, independent of $(E_{\ov{x}_i})_{i \in \bigcup_{k \le n} \ov{\mcl{X}}_{k}}$ (this is the same argument as in Proposition~\ref{p:iid}). 

Finally, for $n \ge 2$, we let 
$$
\mcl{X}_{n+1} = \Ll\{ x \in V_N : |x| = n+1 \text{ and } \ x \sim X_{H_{n}}\Rr\}.
$$
In words, to form $\mcl{X}_{n+1}$, one waits until the first time the random walk visits a point at distance $n$ from the origin, and then collects all neighbouring points that are at distance $n+1$ from the origin. The successive sets $\mcl{X}_3, \mcl{X}_4,\ldots$ can be seen as yet another exploration process, this one being very far from ultimately covering the whole graph. Again, under $\ov{\P}_N^\circ$, the distribution of $(X_t)_{t \le H_{n}}$ does not depend on $(E_x)_{x \in \mcl{X}_{n+1}}$, and the latter is thus a family of independent random variables distributed as $\mcl{N}_+(0,1)$, and is independent of $(E_x)_{x \in \bigcup_{k \le n} \mcl{X}_k}$. Note that for $2 \le n \le N$, we have
$|\mcl{X}_{n+1}| = N-n$ (see Figure~\ref{f:tree}).
\begin{figure}
\centering
\includegraphics[scale=1]{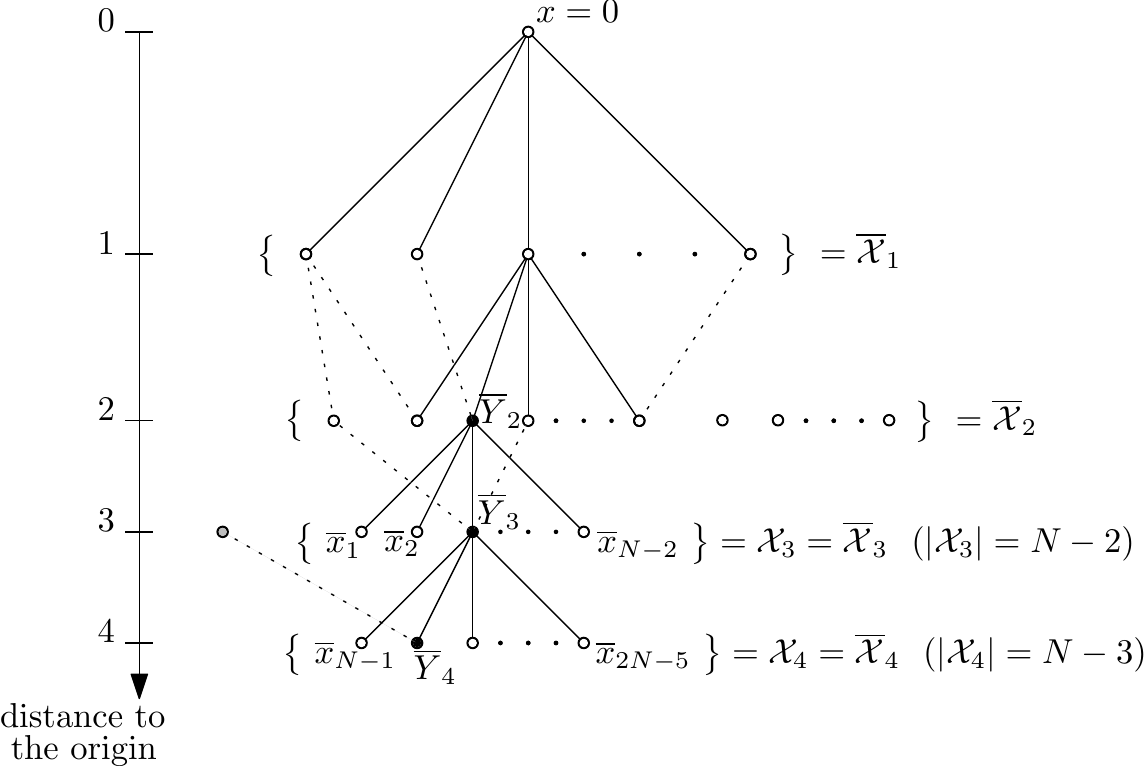}
\caption{
\small{
A schematic representation of a neighbourhood of the origin (on top). All points on one level are at the same distance from the origin. The black dots denote the sites visited by the random walk starting from time $H_2$, indexed by $\ov{Y}_2,\ov{Y}_3, \ov{Y}_4$. Although the drawing (with only the solid edges) looks like a tree, note that the further one goes away from the origin, the more there are edges destroying the tree structure, as exemplified by the dotted edges. %For clarity, many of these edges are absent in the picture. %For instance, there should be two edges linking $\ov{Y}_1$ to a point at distance $1$ from the origin, but only one of them is drawn.
For instance, any point on level~$3$ should be connected to $3$ vertices on level~$2$.
By construction, $\ov{\mathcal{X}}_3$ contains exactly $N-2$ sites, all at distance $3$ from the origin, and is equal to $\mcl{X}_3$. Note that $\mcl{X}_4 = \ov{\mathcal{X}}_4$ is a consequence of the fact that in the example, $\ov{Y}$ goes down to level~$3$. Although  on the example, since $\ov{Y}$ goes down to level $4$, we have $\mcl{X}_5 \subset \ov{\mathcal{X}}_5$ (not drawn on the picture), these two sets are not equal, since the grey dot on level~$3$ belongs to $\ov{\mathcal{X}}_5$ but not to $\mcl{X}_5$ (every point of $\mcl{X}_5$ being on level 5).
}
}
\label{f:tree}
\end{figure}

Let $\mfk{A}_N(n)$ be the event defined to be the conjunction of the event that $\ov{Y}$ goes down to level $n$ and of the event
\begin{equation}
\label{e:controlmax}
\max\Ll\{E_x, \ x \in  \bigcup_{1 \le k \le n} \ov{\mcl{X}}_k\Rr\} \le 2 \sqrt{2 \log N}.
\end{equation}
We now prove by induction on $n$ that if ($N$ is large enough and) $2\le n\le N^{1/4}$, then
\begin{equation}
\label{inductionhyp}	
\ov{\P}_N^\circ\Ll[ \mfk{A}_N(n)\Rr] \ge 1-nN^{5 \ov{a}-3/4}.
\end{equation}
We start by checking that such is the case for $n = 2$. The walk $\ov{Y}$ always goes down to level $2$, so we only need to bound the probability of the event
$$
\max\Ll\{E_x, \ x \in  \ov{\mcl{X}}_1 \cup \ov{\mcl{X}}_2\Rr\} \le 2 \sqrt{2 \log N}.
$$
Noting that $|\ov{\mcl{X}}_1 \cup \ov{\mcl{X}}_2| = N(N+1)/2$ and using \eqref{easybound}, we can bound this probability from below by
$$
\Ll(1-e^{-(2\sqrt{2 \log N})^2/2}\Rr)^{N(N+1)/2} = \Ll(1-N^{-4}\Rr)^{N(N+1)/2} = 1-\frac{1}{2N^2}(1 + o(1)),
$$
and this justifies the base case of the induction. 

Let $2 \le n < N^{1/4}$, and assume that \eqref{inductionhyp} holds. It suffices to check that
\begin{equation}
\label{e:induction0}
\ov{\P}^\circ_N\Ll[ \mfk{A}_N^c(n+1) \ | \ \mfk{A}_N(n) \Rr] \le  N^{5 \ov{a}-3/4},
\end{equation}
where we write $\mfk{A}_N^c(n+1)$ for the complement of $\mfk{A}_N(n+1)$. The probability on the left-hand side above is smaller than the sum of 
\begin{equation}
\label{e:induction1}
\ov{\P}^\circ_N\Ll[ \max_{x \in \ov{\mcl{X}}_{n+1}} E_x > 2 \sqrt{2 \log N} \ | \ \mfk{A}_N(n) \Rr]
\end{equation}
and
\begin{equation}
\label{e:induction2}
\ov{\P}^\circ_N\Ll[ \ov{Y} \text{ does not go down to level } (n+1) \ | \ \mfk{A}_N(n) \Rr].
\end{equation}
Since $(E_x)_{x \in \ov{\mcl{X}}_{n+1}}$ is independent of $\mfk{A}_N(n)$, we can proceed as on the base case of this induction and show that the probability in \eqref{e:induction1} is $O(N^{-3})$. We now consider the probability in \eqref{e:induction2}. Recall that $\ov{Y}$ chooses to jump to a neighbouring site $x$ with probability proportional to $e^{a_N E_x}$. On the event that $\ov{Y}$ goes down to level $n$, $\ov{Y}_n$ has $N-n$ neighbours at level $(n+1)$ (the elements of $\mcl{X}_{n+1}$), and $n$ other neighbours belonging to $\bigcup_{k \le n} \ov{\mcl{X}}_k$. Since the energy on any site is always positive, on the event $\mfk{A}_N(n)$, the probability that $Y$ goes down to level $(n+1)$ is at least
\begin{equation}
\label{rough1}
\frac{N-n}{N-n+ne^{2 a_N\sqrt{2 \log N}}}.
\end{equation}
Since we assume $n < N^{1/4}$ and $a_N \le \ov{a} \sqrt{2 \log N}$, the quantity above is larger than
\begin{equation}
\label{rough2}
\frac{N/2}{N/2 + N^{1/4+4\ov{a}}} = 1-2 N^{4 \ov{a}-3/4}(1+o(1)).
\end{equation}
This justifies \eqref{e:induction0}, and thus completes the proof of \eqref{inductionhyp}.

It follows from \eqref{inductionhyp} that with probability tending to $1$, the event $\mfk{A}_N(N^{1/4})$ is realized. We now explain why the walk cannot go down to level $N^{1/4}$ too fast. 

First, we recall that under $\ov{\P}_N^\circ$, the values $(E_x)_{x \in \mcl{X}_3}, (E_x)_{x \in \mcl{X}_4}, \ldots, (E_x)_{x \in \mcl{X}_{N}}$ form a family of independent random variables distributed as $\mcl{N}_+(0,1)$. Since $|\mcl{X}_{n+1}| \le N$ and $\E[e^{a_N E_x}] = \phi(a_N) \le 2 e^{a_N^2/2}$, Chebyshev's inequality guarantees that if $3 \le k \le N$, then
$$
\ov{\P}_N^\circ\Ll[ \sum_{x \in \mcl{X}_k} e^{a_N E_x} > 4 N e^{a_N^2/2} \Rr] \le \frac{1}{2}.
$$
By independence, a classical large deviation estimate guarantees that the $\ov{\P}_N^\circ$-probability that 
\begin{equation}
\label{e:eventA''}
\Ll| \Ll\{ k \le N^{1/4} : \sum_{x \in \mcl{X}_k} e^{a_N E_x} \le 4 N e^{a_N^2/2} \Rr\}  \Rr| \ge \frac{N^{1/4}}{4}
\end{equation}
tends to $1$ as $N$ tends to infinity. Let us write $\mfk{A}'_N$ for the event that both $\mfk{A}_N(N^{1/4})$ and the event displayed in \eqref{e:eventA''} are satisfied. By \eqref{inductionhyp}, we thus have
\begin{equation}
\label{proba''}
\ov{\P}_N^\circ\Ll[ \mfk{A}'_N \Rr] \xrightarrow[N \to \infty]{} 1.
\end{equation}

Let $\gamma = (\gamma_n)_{2 \le n \le N^{1/4}}$ be a path. We write $\ov{Y} = \gamma$ to denote the event that for any $n$, $2\le n \le N^{1/4}$, we have $\ov{Y}_n = \gamma_n$. In a given environment $\tau_N$, the event $\mfk{A}'_N$ depends only on $(\ov{Y}_n)_{2 \le n \le N^{1/4}}$. Let $\Gamma(\tau_N)$ be the set of paths $\gamma$ such that in the environment $\tau_N$ and on $Y = \gamma$, the event $\mfk{A}'_N$ is realized. 

We now fix some $\gamma \in \Gamma(\tau_N)$. Conditionally on $Y = \gamma$, the waiting times of the continuous-time walk $X$ on the successive sites $(\gamma_n)_{2 \le n \le N^{1/4}}$ are independent exponential random variables with respective parameter
\begin{equation}
\label{e:paramexp}
e^{a_N E_{\gamma_n}} \sum_{x \sim \gamma_n} e^{a_N E_x}.
\end{equation}
On $Y = \gamma$, the event $\mfk{A}'_N$ is realized, and in particular:
\begin{itemize}
	\item the set $\{x : x \sim \gamma_n\}$ is partitioned into the $N-n$ sites that belong to $\mcl{X}_{n+1}$ on one hand, and the $n \le N^{1/4}$ sites that belong to some $\ov{\mcl{X}}_k$ for some $k \le n$ on the other;
	\item the site $\gamma_n$ itself belongs to $\ov{\mcl{X}}_k$ for some $k \le n$;
	\item The events \eqref{e:controlmax}, \eqref{e:eventA''} hold, and $\gamma$ does not visit the origin.
\end{itemize}
These observations enable to bound the parameter in \eqref{e:paramexp} by
$$
e^{2 a_N \sqrt{2 \log N}} \Ll( \sum_{x \in \mcl{X}_{n+1}} e^{a_N E_x} + N^{1/4} e^{2 a_N \sqrt{2 \log N}} \Rr),
$$
and for at least $N^{1/4}/4$ of these $n$'s, the latter is smaller than
$$
e^{2 a_N \sqrt{2 \log N}} \Ll( 4 Ne^{a_N^2/2} + N^{1/4} e^{2 a_N \sqrt{2 \log N}} \Rr) \le N^{4 \ov{a}}\Ll( 4 N^{1+\ov{a}^2} + N^{1/4+4\ov{a}}\Rr).
$$
To sum up, under the measure $\PPt_0$ and conditionally on the event $Y=\gamma$, the total time it takes for the continuous-time walk $X$ to travel through $\gamma$ stochastically dominates a sum of $N^{1/4}/4$ independent exponential random variables of parameter bounded by $5N^{1+4\ov{a}+\ov{a}^2}$. With probability tending to $1$ (uniformly over $\tau_N$ and $\gamma \in \Gamma(\tau_N)$), this sum is larger than $N^{-3/4+5\ov{a}}$. Since we assume $\ov{a} < 1/20$, there exists $\zeta > 0$ such that $-3/4+5\ov{a} = -1+\zeta$. In particular, we have shown that, uniformly over $\tau_N$,
\begin{equation}
\label{unifgamma}
\inf_{\gamma \in \Gamma(\tau_N)} \PPt_0\Ll[ \mfk{A}_N \ | \ Y = \gamma \Rr] \xrightarrow[N \to \infty]{} 1.
\end{equation}
Moreover,
\begin{eqnarray*}
	\PPt_0\Ll[ \mfk{A}_N \cap \mfk{A}'_N \Rr] & = & \sum_{\gamma \in \Gamma(\tau_N)} \PPt_0\Ll[ \mfk{A}_N, \  Y = \gamma \Rr] \\
& = & \sum_{\gamma \in \Gamma(\tau_N)} \PPt_0\Ll[ \mfk{A}_N \ | \ Y = \gamma \Rr] \ \PPt_0\Ll[ Y = \gamma \Rr].
\end{eqnarray*}
From \eqref{unifgamma}, it thus follows that for any $N$ large enough, uniformly over $\tau_N$,
\begin{eqnarray*}
\PPt_0\Ll[ \mfk{A}_N \cap \mfk{A}'_N \Rr] & = & (1+o(1)) \sum_{\gamma \in \Gamma(\tau_N)} \PPt_0\Ll[ Y = \gamma \Rr] \\
& = & (1+o(1)) \PPt_0[\mfk{A}'_N].
\end{eqnarray*}
Integrating with respect to $\P^\circ$ and using \eqref{proba''}, we obtain the desired result.
\end{proof}
\begin{rem}
\label{r:optim-a}
There are many ways to improve the proof of Proposition~\ref{p:descent} so that it covers higher values of $\ov{a}$. Yet, we now explain why a simple optimization of exponents cannot cover every $\ov{a} < 1$. Let $\xi$ be an exponent such that with probability tending to $1$, $\ov{Y}$ goes down to level $N^{\xi}$ with probability tending to $1$. We have seen in the proof of Proposition~\ref{p:descent} that for $\ov{a}$ small enough, one can take $\xi = 1/4$. Yet, as $\ov{a}$ is taken closer and closer to $1$ (and fixing $a_N = \ov{a} \sqrt{2 \log N}$), the exponent $\xi$ must approach $0$, while the time spent on a site during the descent is an exponential random variable with parameter of order at least $N \phi(a_N) \sim N^{1+\ov{a}^2}$. That is, the total time spent on the full descent is of order no more than $N^{-\ov{\zeta}}$ with $\ov{\zeta}$ approaching $2$ as $\ov{a}$ is taken closer and closer to $1$.
\end{rem}

\begin{lem}
\label{l:hittingproba}
Let $\zeta \in (0,1)$. There exists a measurable $\Omega_N' \subset \Omega_N$ such that $\P\Ll[\Omega_N'\Rr]$ tends to $1$ as $N$ tends to infinity, and moreover, for any $\tau_N \in \Omega_N'$,
\begin{equation}
\label{e:hittingproba}
\forall x \ssim 0, \ \PPt_x\Ll[ \exists t : N^{\zeta - 1} \le t \le N^3 : |X_t| \le 1 \Rr] \le \exp(-N^{\zeta/2}).
\end{equation}
\end{lem}
\begin{proof}
Proposition~\ref{p:fuite} (or more precisely, inequality~\eqref{e:fuite}) ensures that for any $x \in V_N$,
$$
\P \PPt_x\Ll[ \int_{N^{\zeta-1}}^{N^3+1} \1_{|X_t| \le 1} \ \d t \Rr] \le (N+1)(N^3 + 1) \Ll( \frac{1+\exp\Ll( -2 N^{\zeta - 1} \Rr)}{2} \Rr)^N,
$$
which, for $N$ sufficiently large, is smaller than $\exp(-N^{\zeta}/2)$. By a union bound and Chebyshev's inequality, the $\P$-probability that
\begin{equation}
\label{e:hitting1}
\exists x \ssim 0 : \PPt_x\Ll[ \int_{N^{\zeta-1}}^{N^3+1} \1_{|X_t| \le 1} \ \d t \Rr] > \exp(-N^{\zeta}/4)
\end{equation}
tends to $0$ as $N$ tends to infinity. We now argue that the probability appearing on the left-hand side of \eqref{e:hitting1} is, up to some power of $N$, an upper bound for the probability in \eqref{e:hittingproba}. More precisely, let us write 
$$
\msf{H}_N = \inf \{ t \ge N^{\zeta - 1} : |X_t| \le 1\},
$$
and observe that
\begin{eqnarray}
\label{e:hitting1.5}
\EEt_x\Ll[ \int_{N^{\zeta-1}}^{N^3+1} \1_{|X_t| \le 1} \ \d t \Rr] & \ge & \EEt_x\Ll[ \1_{\msf{H}_N \le N^3} \int_{N^{\zeta-1}}^{N^3+1} \1_{|X_t| \le 1} \ \d t \Rr] \notag \\
& \ge & \EEt_x\Ll[ \1_{\msf{H}_N \le N^3} \int_{\msf{H}_N}^{\msf{H}_N+1} \1_{X_t=X_{\msf{H}_N}} \ \d t \Rr] \notag \\
& \ge & \EEt_x\Ll[ \1_{\msf{H}_N \le N^3} \ \EEt_{X_{\msf{H}_N}}\Ll[\int_{0}^{1} \1_{X_t=X_{0}} \ \d t \Rr] \Rr].
\end{eqnarray}
where we applied the Markov property at time $\msf{H}_N$ in the last step. 
The total time spent on site $y \in V_N$ is an exponential random variable with parameter 
$$
\msf{T}_{N,y} = e^{a_N E_y} \sum_{z \sim y} e^{a_N E_z}. 
$$
Note that
$$
\E\Ll[\msf{T}_{N,y}\Rr] = N \phi(a_N)^2 \le 2 N \Ll(e^{\ov{a} \log N}\Rr)^2 = 2 N^{1+2\ov{a}}.
$$
In particular, since we assume $\ov{a} \le 1$, the $\P$-probability that 
\begin{equation}
\label{e:hitting2}
\forall y \text{ s.t. } |y| \le 1, \ \msf{T}_{N,y} \le N^5
\end{equation}
tends to $1$ as $N$ tends to infinity. We define $\Omega_N'$ as the set where the event displayed in \eqref{e:hitting2} and the complement of the event displayed in \eqref{e:hitting1} hold. We have seen that $\P\Ll[ \Omega_N'\Rr]$ tends to $1$ as $N$ tends to infinity. For $\tau_N \in \Omega'_N$, we have
$$
\forall y \text{ s.t. } |y|\le 1, \ \EEt_{y}\Ll[\int_{0}^{1} \1_{X_t=X_{0}} \ \d t \Rr] \ge \frac{1}{2N^5}.
$$
In view of \eqref{e:hitting1.5}, it thus follows that
$$
\EEt_x\Ll[ \int_{N^{\zeta-1}}^{N^3+1} \1_{|X_t| \le 1} \ \d t \Rr] \ge \frac{\PPt_x[\msf{H}_N\le N^3]}{2 N^5}.
$$
Using also the fact that when $\tau_N \in \Omega_N'$, the complement of the event displayed in \eqref{e:hitting1} holds, we obtain that for any $x \ssim 0$,
$$
\PPt_x[\msf{H}_N\le N^3] \le 2 N^5 \exp(-N^{\zeta}/4),
$$
and this finishes the proof.
\end{proof}

\begin{proof}[Proof of Theorem~\ref{t:green}]
Let us write 
$$
\ell = \int_0^{\td{e}_N(1)}  \1_{X_t = 0} \ \d t,
$$
where we recall that $\td{e}_N(1)$ is an exponential random variable of parameter $1/N^2$, independent of everything else. It follows from \eqref{defG} that 
$$
G_N(\tau_N) = \EEt_0\Ll[\ell\Rr].
$$
Since $\ell$ is the total time spent at a point by a certain transient Markov chain (the chain $X$ with killing parameter $1/N^2$), we have in fact that under $\PPt_0$, the random variable $\ell$ follows an exponential distribution with mean $G_N(\tau_N)$.

We decompose $\EEt_0\Ll[\ell\Rr]$ as 
\begin{equation}
\label{e:convG}	
\EEt_0\Ll[\int_0^{H_2}  \1_{X_t = 0} \ \d t\Rr] + \EEt_0\Ll[\int_{H_2}^{\td{e}_N(1)}  \1_{X_t = 0} \ \d t\Rr].
\end{equation}
Our first step is to show that the second term in this sum is small. Note that it may be negative, since it may happen that $H_2 > \td{e}_N(1)$. We bound its absolute value by
\begin{equation}
\label{e:convG2}
\EEt_0\Ll[H_2, \ H_2 > \td{e}_N(1)\Rr] + \EEt_0\Ll[\int_{H_2}^{\td{e}_N(1)}  \1_{X_t = 0} \ \d t, \ H_2 \le \td{e}_N(1) \Rr].
\end{equation}
The random variable $H_2$ can be decomposed as the time spent on the origin before reaching level $2$, say $H_2^{(0)}$, plus the time spent at level $1$ before reaching level $2$, say $H_2^{(1)}$. The random variable $H_2^{(0)}$ follows an exponential distribution whose parameter is displayed in \eqref{e:timeH2} (with $x = 0$), and takes the form
$$
\frac{\msf{A} \msf{B}}{\msf{A}+\msf{B}},
$$
where 
$$
\msf{A} = \sum_{y \sim 0} e^{a_N(E_0+E_y)} \ge N,
$$
and
$$
\msf{B} = \sum_{y \sim 0} \sum_{z \sim y, z \neq 0} e^{a_N(E_y+E_z)} \ge N.
$$
As a consequence, we have
$$
\msf{A} \msf{B}\ge N(\msf{A}+\msf{B})/2,
$$
which implies that $H_2^{(0)}$ is dominated by an exponential random variable of parameter $N/2$. As for $H_2^{(1)}$, at any site at distance $1$ from the origin, the walk may jump at distance $2$ with a transition rate at least $(N-1)$. Hence, $H_2^{(1)}$ is dominated by an exponential random variable of parameter $(N-1)$. The first term in \eqref{e:convG2} is smaller than 
$$
\EEt_0\Ll[H_2^2\Rr]^{1/2} \PPt_0\Ll[H_2 > \td{e}_N(1)\Rr]^{1/2},
$$
and this is exponentially small in $N$ (uniformly over $\tau_N$) by the above observations, since
$$
\PPt_0\Ll[H_2 > \td{e}_N(1)\Rr] \le \PPt_0\Ll[H_2 > 1\Rr] + \PPt_0\Ll[1 \le  \td{e}_N(1)\Rr].
$$
We now turn to the second term appearing in \eqref{e:convG2}. Recalling the definition of $\mfk{A}_N$ given in \eqref{defmclA}, we define $\td{\mfk{A}}_N$ as the conjunction of $\mfk{A}_N$, $H_2 \le \td{e}_N(1)$ and $\td{e}_N(1) \le N^3$. By Proposition~\ref{p:descent} and the above observations, we know that 
$$
\ov{\P}_N^\circ\Ll[ \td{\mfk{A}}_N \Rr] \xrightarrow[N \to \infty]{} 1.
$$
It thus follows from Chebyshev's inequality that there exists a sequence $\eps_N$ tending to $0$ and such that 
$$
\P^\circ\Ll[\PPt_0\Ll[\td{\mfk{A}}_N^c\Rr] > \eps_N \Rr] \le \eps_N
$$
(indeed, it suffices to take $\eps_N^2 = {\ov{\P}_N^\circ\Ll[ \td{\mfk{A}}_N \Rr]}$). From now on, we only consider environments $\tau_N$ such that 
\begin{equation}
\label{e:convG4}
\PPt_0\Ll[\td{\mfk{A}}_N^c\Rr] \le \eps_N.
\end{equation}
We rewrite the second term appearing in \eqref{e:convG2} as
\begin{multline}
\label{e:convG3}
\EEt_0\Ll[\int_{H_2}^{\td{e}_N(1)}  \1_{X_t = 0} \ \d t, \ H_2 \le \td{e}_N(1) \Rr] \\ 
= \EEt_0\Ll[\int_{H_2}^{\td{e}_N(1)}  \1_{X_t = 0} \ \d t, \ \td{\mfk{A}}_N \Rr] + \EEt_0\Ll[\int_{H_2}^{\td{e}_N(1)}  \1_{X_t = 0} \ \d t, \ H_2 \le \td{e}_N(1), \ \td{\mfk{A}}_N^c \Rr].
\end{multline}
The second term in the right-hand side above is bounded by
$$
\EEt_0\Ll[\ell,  \ \td{\mfk{A}}_N^c \Rr] \le \underbrace{\EEt_0\Ll[\ell^2\Rr]^{1/2}} \ \PPt_0\Ll[ \td{\mfk{A}}_N^c \Rr]^{1/2},
$$
by the Cauchy-Schwarz inequality. Since $\ell$ is an exponential random variable of parameter $G_N(\tau_N)$, the quantity underbraced above is equal to $\sqrt{2} \ G_N(\tau_N)$. Since we also focus on environments such that \eqref{e:convG4} holds, we thus have
\begin{equation}
\label{e:convG5}
0 \le \EEt_0\Ll[\int_{H_2}^{\td{e}_N(1)}  \1_{X_t = 0} \ \d t, \ H_2 \le \td{e}_N(1), \ \td{\mfk{A}}_N^c \Rr] \le \sqrt{2 \eps_N} \ G_N(\tau_N) .	
\end{equation}
We now consider the first term in the right-hand side of \eqref{e:convG3}. By the definition of the event $\td{\mfk{A}}_N$, this term is (positive and) smaller than
$$
\EEt_0\Ll[\int_{H_2 + N^{\zeta-1}}^{N^3}  \1_{X_t = 0} \ \d t, \ \td{\mfk{A}}_N \Rr].
$$
By the Markov property at time $H_2$, this is smaller than
\begin{equation}
\label{e:convG6}
\EEt_0\Ll[\EEt_{X_{H_2}}\Ll[\int_{N^{\zeta-1}}^{N^3}  \1_{X_t = 0} \ \d t, \ \ov{\mfk{A}}_N \Rr] \Rr],
\end{equation}
where we wrote $\ov{\mfk{A}}_N$ for the event
$$
\inf\{ t \ge 0 : |X_t| \le 1 \} \ge N^{\zeta - 1}.
$$
Since $X_{H_2}$ is at distance $2$ from the origin, the quantity in \eqref{e:convG6} is bounded by
\begin{multline}
\label{e:convG7}
\sum_{x : |x| = 2} \EEt_{x}\Ll[\int_{N^{\zeta-1}}^{N^3}  \1_{X_t = 0} \ \d t, \ \ov{\mfk{A}}_N \Rr] \\
\le  N^3  \sum_{x : |x| = 2}\PPt_x\Ll[\underbrace{\ov{\mfk{A}}_N \text{ and }\exists t \text{ s.t. }  N^{\zeta - 1} \le t < N^3, |X_t| \le 1}\Rr].
\end{multline}
The crucial point of our argument is that the event underbraced above does not depend on~$\tau_{N,0}$. Indeed, it is a function of the trajectory up to hitting the $1$-neighbourhood of $0$, and this does not depend on $\tau_{N,0}$. Let us write 
$$
\Omega_N^\circ = \{ \tau_N : \exists \tau_N' \text{ s.t. } \tau_N \text{ and } \tau_N' \text{ are equal outside of } 0 \text{ and moreover, } \tau_N' \in \Omega_N'\},
$$
where $\Omega_N'$ is given by Lemma~\ref{l:hittingproba}. The probability $\P^\circ\Ll[\Omega_N^\circ\Rr]$ tends to $1$ as $N$ tends to infinity. For any $\tau_N \in \Omega_N^\circ$, there exists $\tau_N' \in \Omega_N'$ that agrees with $\tau_N$ outside of~$0$. For such $\tau_N$ and $\tau_N'$, we have
\begin{multline*}
\PPt_x\Ll[{\ov{\mfk{A}}_N \text{ and }\exists t \text{ s.t. }  N^{\zeta - 1}
\le t < N^3, |X_t| \le 1}\Rr]\\
 = \PP^{\tau_N'}_x\Ll[{\ov{\mfk{A}}_N \text{ and }\exists t \text{ s.t. }  N^{\zeta - 1} \le t < N^3, |X_t| \le 1}\Rr],
\end{multline*}
and the right-hand side is smaller than $\exp(-N^{\zeta/2})$ for any $x \ssim 0$ by Lemma~\ref{l:hittingproba}. As a consequence, the left-hand side of \eqref{e:convG7}, and thus also the first term in the right-hand side of \eqref{e:convG3}, is smaller than $\exp(-N^{\zeta/4})$ for all $N$ sufficiently large, provided $\tau_N \in \Omega_N^\circ$.

To sum up, we have shown that there exists $\ov{\Omega}_N$ such that $\P^\circ\Ll[\ov{\Omega}_N\Rr]$ tends to $1$ as $N$ tends to infinity, and moreover, for any $\tau_N \in \ov{\Omega}_N$, 
\begin{equation}
\label{e:convG8}
\Ll|G_N(\tau_N) - \EEt_0\Ll[\int_0^{H_2}  \1_{X_t = 0} \ \d t\Rr] \Rr| \le  e^{-cN} + \sqrt{2\eps_N}  G_N(\tau_N) + \exp(-N^{\zeta/4}),
\end{equation}
where $c > 0$ and we recall that $\lim \eps_N = 0$. To conclude, it thus suffices to see that for any $\eps > 0$, there exists $\Omega_N''$ such that $\P^\circ\Ll[{\Omega}_N''\Rr]$ tends to $1$ as $N$ tends to infinity, and moreover, for any $\tau_N \in \Omega_N''$, 
\begin{equation}
\label{e:convG9}
\EEt_0\Ll[\int_0^{H_2}  \1_{X_t = 0} \ \d t\Rr] = (1 \pm \eps) N^{-2} \phi(a_N)^{-2}
\end{equation}
as soon as $N$ is sufficiently large. Indeed, showing (\ref{e:convG9}) is sufficient in view of \eqref{e:convG8} since 
$$
N^{-2} \phi(a_N)^{-2} \ge N^{-2} \Ll(e^{\ov{a} \log N}\Rr)^{-2}  \ge N^{-4},
$$
which is much larger than $\exp(-N^{\zeta/4})$. By Lemma~\ref{l:timeH2}, the left-hand side of \eqref{e:convG9} is equal to the inverse of the quantity in \eqref{e:timeH2} taken for $x = 0$. By Proposition~\ref{p:estimatyp}, there exists $\Omega_N''$ such that $\P^\circ\Ll[{\Omega}_N''\Rr]$ tends to $1$ as $N$ tends to infinity, and moreover, for any $\tau_N \in \Omega_N''$,
\begin{equation}
\label{e:convG10}
\tau_{N,0} \ge \delta B_N \text{ and } \sum_{x \sim 0} e^{a_N E_x} \sum_{{y \sim x , y \neq 0}} e^{a_N E_y} = (1\pm \eps) N^2 \phi(a_N)^2,
\end{equation}
since the second condition does not involve $\tau_{N,0}$. Under the conditions displayed in \eqref{e:convG10}, the parameter in \eqref{e:timeH2}, taken for $x = 0$, is equal to
\begin{multline}
\label{e:convG11}
e^{a_N E_0} \sum_{y \sim 0} e^{a_N E_y} \frac{(1 \pm \eps) N^2 \phi(a_N)^2}{e^{a_N E_0} \sum_{y \sim 0} e^{a_N E_y} + (1 \pm \eps) N^2 \phi(a_N)^2} \\
= (1 \pm \eps) N^2 \phi(a_N)^2 \frac{e^{a_N E_0} \sum_{y \sim 0} e^{a_N E_y}}{e^{a_N E_0} \sum_{y \sim 0} e^{a_N E_y} + (1 \pm \eps) N^2 \phi(a_N)^2}.
\end{multline}
Recalling that $B_N = \exp(\beta \sqrt{N} b_N)$ and $b_N \sim \sqrt{2\ov{c} N}$, it comes that if $\tau_{N,0}= \exp(\beta \sqrt{N} E_0) \ge \delta B_N$, then it must be that $E_0 \ge N^{1/4}$. Since $a_N \ge 1$, we also have $\exp(a_N E_0) \ge \exp(N^{1/4})$. Consequently, the fraction appearing in the right-hand side of \eqref{e:convG11} tends to $1$ as $N$ tends to infinity, and this completes the proof.
\end{proof}

%
%
%
%
%%%%%%%%%%%%%%%%%%%%%%%%%%%%%%%%%%%%%%%%%%%%%%%%%%%%%%%%%%%%%%
%%%%%%%%%%%%%%%%%%%%%%%%%%%%%%%%%%%%%%%%%%%%%%%%%%%%%%%%%%%%%%
%
%
%
\section{The Green function at the deep trap}
\label{s:greenatdeep}
\setcounter{equation}{0}
The purpose of this section is to show the following theorem.
\begin{thm}
\label{t:green-trap}
Under the measure $\ov{\P}_N$, the random variable $N^2\phi(a_N)^2 \ G_N(\tau\de_N(1))$ converges to $1$ in probability.	
\end{thm}
What Theorem~\ref{t:Radon-Nikodym} says about $G_N(\tau\de_N(1))$ is that its size-biased distribution is the distribution of $G_N(\tau_N)$ under $\P^\circ$. In the previous section, we characterized the latter distribution. Knowing the size-bias of a distribution is however not sufficient to characterize it in general. Indeed, consider that any non-identically zero Bernoulli random variable has the constant $1$ as its size-biased distribution. As we will now see, in order to conclude, it suffices to show that $G_N(\tau\de_N(1))$ cannot take values smaller than $c N^{-2}\phi(a_N)^{-2}$ for some $c > 0$. This is the content of the next proposition.
\begin{prop}
\label{p:green-trap-lower-bound}
We have
$$
\ov{\P}_N\Ll[N^2 \phi(a_N)^2 G_N(\tau\de_N(1)) \ge 1/2 \Rr] \xrightarrow[N \to \infty]{} 1. 
$$
\end{prop}
\begin{proof}
The quantity $G_N(\tau\de_N(1))$ is bounded from below by the inverse of the parameter displayed in \eqref{e:timeH2}. Moreover, since $x\de_N(1)$ is a deep trap, we have $E_{x\de_N(1)} \ge N^{1/4}$. Reasoning as in \eqref{e:convG11} and the few lines below, we see that in order to prove Proposition~\ref{p:green-trap-lower-bound}, it suffices to show that
\begin{equation}
\label{condsums1}
\ov{\P}_N\Ll[\sum_{x \sim x\de_N(1)} e^{a_N E_x} \sum_{{y \sim x, y \neq x\de_N(1)}} e^{a_N E_y} \le 2 N^2 \phi(a_N)^2 \Rr] \xrightarrow[N \to \infty]{} 1.
\end{equation}
We will prove \eqref{condsums1} with the same technique as the one used to prove Lemma~\ref{l:tight}. Let $\eps \in (0,1)$, and
$$
\mcl{W}_N = \{x \in V_N : |x| \le 2\}.
$$
We say that a site $z \in V_N$ is \emph{atypical} if both
$$
\tau_{N,0} \ge \delta B_N
$$
and 
\begin{equation}
\label{condsums}
\sum_{x \sim z} e^{a_N E_x} \sum_{{y \sim x, y \neq z}} e^{a_N E_y} \neq(1\pm\eps) N^2 \phi(a_N)^2
\end{equation}
hold (where we write $\msf{a} \neq (1\pm\eps) \msf{b}$ to mean that there exists no $\msf{c}$ such that $|\msf{c}|\le \eps$ and $\msf{a} = (1+\msf{c})\msf{b}$). 
We say that $z \in V_N$ is \emph{uncommon} if 
$$
\exists z'  : \ z' \in  z+\mcl{W}_N \text{ and } z' \text{ is atypical}.
$$
For $\Gamma \subset V_N$, we say that $z$ is \emph{uncommon regardless of} $\Gamma$ if 
$$
\exists z'  : \ z' \in z+\mcl{W}_N, \ (z'+\mcl{W}_N) \cap \Gamma = \emptyset \text{ and } z' \text{ is atypical}.
$$
In order to show that \eqref{condsums1} holds, it suffices to show that
\begin{equation*}
%\label{condsums2}
\ov{\P}_N\Ll[x\de_N(1) \text{ is atypical} \Rr] \xrightarrow[N \to \infty]{} 0.
\end{equation*}
As in the proof of Lemma~\ref{l:tight}, if $x\de_N(n)$ is atypical, then there must exist $k \le r\de_N(n)$ such that $x_k$ is uncommon regardless of $\{x_1,\ldots, x_{k-1}\}$. We recall that the probability in \eqref{e:Cr} can be made arbitrarily small if $C_r$ is taken large enough. It thus suffices to show that
$$
\ov{\P}_N\Ll[ x\de_N(n) \text{ is atypical}, \ r\de_N(n) \le C_r \msfd \Rr]
$$
tends to $0$ with $N$. This probability is bounded by
$$
\sum_{k = 1}^{C_r \msfd} \ov{\P}_N \Ll[ x_k \text{ is uncommon regardless of } \{x_1,\ldots, x_{k-1}\} \Rr].
$$
Every summand is bounded by ${\P}\Ll[0 \text{ is uncommon}\Rr]$, which itself is smaller than
\begin{equation}
\label{e:condsums}
\Ll|\mcl{W}_N\Rr| \ \P\Ll[ 0 \text{ is atypical} \Rr] = \Ll|\mcl{W}_N\Rr| \ \P\Ll[\tau_{N,0} \ge \delta B_N\Rr] \ \P\Ll[ \text{\eqref{condsums} holds with } z = 0\Rr].
\end{equation}
The conclusion follows since $\msfd \P\Ll[\tau_{N,0} \ge \delta B_N\Rr]$ converges to a constant, $\Ll|\mcl{W}_N\Rr| = O(N^2)$, while the last probability in \eqref{e:condsums} is $o(N^{-2})$ by Proposition~\ref{p:estimatyp}.
\end{proof}

\begin{proof}[Proof of Theorem~\ref{t:green-trap}]
Let $A$ be a measurable subset of $[1/2,+\infty)$, and define
$$
h_N(\tau_N) = \frac{1}{2 N^2 \phi(a_N)^2 G_N(\tau_N)} \ \1_{N^2 \phi(a_N)^2 G_N(\tau_N) \in A}.
$$
We have $\|h_N\|_\infty \le 1$. Let $\eps > 0$. We learn from Theorem~\ref{t:Radon-Nikodym} that for any $M$ large enough and any $N$ large enough,
\begin{equation}
\label{e:proofgt}
\ov{\E}_N \Ll[(h_N G_N)(\tau\de_N(1)) , \ \mcl{E}^M_N \Rr] = \frac{\msft}{\msfd} \Ll(\E\Ll[h_N(\tau) \ | \ {\tau_{N,0} \ge \delta B_N} \Rr]  \pm \eps\Rr).
\end{equation}
Moreover, if $M$ is chosen sufficiently large, then for any $N$ large enough, we have
\begin{equation}
\label{e:proofgt2}
\ov{\P}_N\Ll[ \mcl{E}^M_N \Rr] \ge 1- \eps,
\end{equation}
as follows from Propositions~\ref{p:poisson} and \ref{p:gotouch}. From now on, we fix $M$ such that this and identity \eqref{e:proofgt} hold. Identity \eqref{e:proofgt} can be rewritten as
\begin{multline}
\label{e:proofgt3}
\ov{\P}_N \Ll[N^2 \phi(a_N)^2 G_N(\tau\de_N(1)) \in A , \ \mcl{E}^M_N \Rr] \\ 
= N^2 \phi(a_N)^2 \frac{\msft}{\msfd} \Ll(\E^\circ\Ll[\Ll(N^2 \phi(a_N)^2 G_N(\tau_N)\Rr)^{-1} \1_{N^2 \phi(a_N)^2 G_N(\tau_N) \in A} \Rr]  \pm 2\eps\Rr).
\end{multline}
We first choose $A = [1/2,+\infty)$. In this case, we know from \eqref{e:proofgt2} and Proposition~\ref{p:green-trap-lower-bound} that the left-hand side is equal to $(1\pm 2\eps)$ for all large enough $N$. On the other hand, the expectation on the right-hand side converges to $1$ by Theorem~\ref{t:green}, since the integrand is bounded by $2$. We thus get that for any $N$ large enough,
$$
N^2 \phi(a_N)^2 \frac{\msft}{\msfd} = \frac{1\pm 2\eps}{1 \pm 3\eps}.
$$
Since this holds for arbitrary $\eps > 0$, we have shown that
\begin{equation}
\label{e:proofgt4}
N^2 \phi(a_N)^2\frac{\msft}{\msfd} \xrightarrow[N \to \infty]{} 1.
\end{equation}
Using this information, we can rewrite \eqref{e:proofgt3} as
\begin{multline*}
%\label{e:proofgt3}
\ov{\P}_N \Ll[N^2 \phi(a_N)^2 G_N(\tau\de_N(1)) \in A , \ \mcl{E}^M_N \Rr] \\ 
= (1 \pm \eps) \Ll(\E^\circ\Ll[\Ll(N^2 \phi(a_N)^2 G_N(\tau_N)\Rr)^{-1} \1_{N^2 \phi(a_N)^2 G_N(\tau_N) \in A} \Rr]  \pm 2\eps\Rr).
\end{multline*}
We now choose $A = [1-\eta,1+\eta]$. In this case, the expectation in the right-hand side above still converges to $1$ by Theorem~\ref{t:green}. Hence,
$$
\ov{\P}_N \Ll[N^2 \phi(a_N)^2 G_N(\tau\de_N(1)) \in [1-\eta,1+\eta] , \ \mcl{E}^M_N \Rr] \\ 
= (1 \pm \eps) (1 \pm 3\eps).
$$
Using also \eqref{e:proofgt2}, we thus obtain that for any $N$ large enough,
$$
\ov{\P}_N \Ll[N^2 \phi(a_N)^2 G_N(\tau\de_N(1)) \in [1-\eta,1+\eta]\Rr] = (1 \pm \eps) (1 \pm 3\eps) \pm \eps.
$$
Since this holds for arbitrary $\eps > 0$, we have shown that
$$
\ov{\P}_N \Ll[N^2 \phi(a_N)^2 G_N(\tau\de_N(1)) \in [1-\eta,1+\eta]\Rr] \xrightarrow[N \to \infty]{} 1.
$$
The parameter $\eta > 0$ being arbitrary, this completes the proof of Theorem~\ref{t:green-trap}.
\end{proof}
As a by-product, we also proved \eqref{e:proofgt4}, which we restate as a proposition for later usage.
\begin{prop}
\label{p:ratios}
We have
\begin{equation}
\label{e:ratios}
\frac{\msfd}{\msft} \sim N^2 \phi(a_N)^2 \qquad (N \to \infty).
\end{equation}
\end{prop}

%
%
%
%
%%%%%%%%%%%%%%%%%%%%%%%%%%%%%%%%%%%%%%%%%%%%%%%%%%%%%%%%%%%%%%
%%%%%%%%%%%%%%%%%%%%%%%%%%%%%%%%%%%%%%%%%%%%%%%%%%%%%%%%%%%%%%
%
%
%
\section{Handling several traps}
\label{s:manytraps}
\setcounter{equation}{0}

What we have achieved up to now is a thorough understanding of the asymptotic behaviour of certain quantities associated to the first deep trap visited, and most importantly, of the behaviour of the Green function at this location.

The purpose of this section is to show that from this information, we can infer the asymptotic joint behaviour of the all the relevant quantities at all the deep traps visited, namely: the time it takes to find them, their depths, the value of the Green function there, and the total time actually spent on the sites. In order to do so, we use again the times introduced in \eqref{deftdT0} and \eqref{deftdT}, which we can slightly simplify by taking $M = \infty$. In particular, we have $T^\infty_N(1) = T\de_N(1)$. For values of $n$ other than $1$, it is not true in general that $T^\infty_N(n) = T\de_N(n)$: the difference is that $T\de_N(2)$ is the time when we discover a second deep trap, while $T^\infty_N(2)$ is the time for the process to discover a first deep trap, then wait $\td{e}_N(1)$, then wait a strong stationary time, and then come at distance $1$ to a deep trap again (possibly \emph{the same} as the first). The interest of these times is that they carry a lot of independence, as given by Proposition~\ref{p:stationarity}.

We define $x^\infty_N(n)$ as the deep trap met at time $T^\infty_N(n)$ (with some tie-breaking rule if multiple deep traps are simultaneously met). We write $\tau^\infty_N(n)$ for the environment around site $x\fty_N(n)$, that is, $\tau^\infty_N(n) = \theta_{x\fty_N(n)} \ \tau_{N}$. 

%We also define $H\fty_N(n)$ as the first time the site $x^\infty_N(n)$ is visited, and $e\fty_N(n)$ such that
%$$
%\int_{H\fty_N(n)}^{H\fty_N(n)+\td{e}_N(n)} \1_{X_s = x\fty_N(n)} \ \d s = G_N(\tau\fty_N(n)) \ e\fty_N(n).
%$$
%Although the notation is not explicit in this respect, note that these newly defined quantities actually depend on $\delta$.
\begin{prop}
\label{p:greenfty}
Under the measure $\ov{\P}_N$ and for every $n$, $N^2\phi(a_N)^2 \ G_N(\tau\fty_N(n))$ converges to $1$ in probability.
\end{prop}
\begin{proof}
%Let us write $\td{G}_N(n) = N^2\phi(a_N)^2 \ G_N(\tau\fty_N(n))$, and for every $\eps > 0$, write $A_\eps = [1-\eps,1+\eps]$. Under $\PPtu$, we know from Proposition~\ref{p:stationarity} that the random variables $(\td{G}_N(n))_{n \in \N}$ are independent and identically distributed. In particular, for every $k$,
%$$
%\PPtu\Ll[\td{G}_N(1) \in A_\eps, \ldots, \td{G}_N(k) \in A_\eps\Rr] = \prod_{n = 1}^k \PPtu\Ll[\td{G}_N(n) \in A_\eps \Rr].
%$$
%Now, each probability on the right-hand side converges to $1$ in $\P$-probability, so the $\E$-expectation of the right-hand side converges to $1$, and we obtain
%$$
%\ov{\P}_N\Ll[\td{G}_N(1) \in A_\eps, \ldots, \td{G}_N(k) \in A_\eps\Rr] \xrightarrow[N \to \infty]{} 1.
%$$ 

We know from Proposition~\ref{p:stationarity} that the random variables $(\tau\fty_N(n))_{n \in \N}$ are independent and identically distributed under $\PPtu$. In particular, for every $\eps > 0$ and every $n$,
\begin{multline*}
\PPtu\Ll[N^2\phi(a_N)^2 \ G_N(\tau\fty_N(n)) \in [1-\eps,1+\eps]\Rr] \\
= \PPtu\Ll[N^2\phi(a_N)^2 \ G_N(\tau\fty_N(1)) \in [1-\eps,1+\eps]\Rr].
\end{multline*}
Recall that $\tau\fty_N(1) = \tau\de_N(1)$. Taking $\E$-expectation above and using Theorem~\ref{t:green-trap}, we obtain the result.
\end{proof}
Recall that $\mcl{H}\de_N(n)$ is the event that $x\de_N(n)$ is visited during the time interval $[T\de_N(n), T\de_N(n)+N]$. For a time interval $I$, we say that $x$ is \emph{seen} during $I$ if there exists $s \in I$ such that $|X_s - x| \le 1$. This differs slightly from the notion of being discovered during $I$ in that it may happen that a point is seen during $I$, but had already been discovered at a time prior to $I$. 

For $k\in \N$, let $\mcl{F}\de_N(k)$ be the event defined as the conjunction of 
\begin{equation}
\label{event1}
\begin{array}{l}
\text{for every } n \le k, \mcl{H}\de_N(n) \text{ holds, and moreover, during }   \\
\text{} (T\de_N(n), {T}\de_N(n)+N^3]\text{, no deep trap is discovered } ;
\end{array}
\end{equation}
\begin{equation}
\label{event2}
\begin{array}{l}
\text{for every } n \le k, \text{ the site } x\de_N(n) \text{ is not seen} \\
\text{during } [0,T\de_N(k+1)] \setminus [T\de_N(n), {T}\de_N(n)+N],
\end{array}
\end{equation}

\begin{prop}
\label{p:goood}
(1) For every $k$, we have 
$$
\ov{\P}_N\Ll[ \mcl{F}\de_N(k) \Rr] \xrightarrow[N \to \infty]{} 1.	
$$

(2) For every $k\in \N$, with $\ov{\P}_N$-probability tending to $1$, we have $x\fty_N(n) = x\de_N(n)$ for every $n \le k$.
\end{prop}
%\begin{lem}
%Under the measure $\ov{\P}_N$, the distribution of $(\msft^{-1} T\fty_N(n))_{n \in \N}$ converges to that of the jump instants of a Poisson process of intensity $\delta^{-\alpha}$.
%\end{lem}
%\begin{proof}
%We know from Proposition~\ref{p:stationarity} that the random variables $(T\fty_N(n+1)-T\fty_N(n))_{n \in \N}$ are independent and identically distributed under $\PPtu$.
%\end{proof}
\begin{proof}
It was seen in Proposition~\ref{p:gotouch} and Lemma~\ref{l:notwotraps} that the event described in \eqref{event1} has $\ov{\P}_N$-probability tending to $1$. We will show that for every $n \in \N$ and every $M > 0$,
\begin{equation}
\label{e:goood}
\ov{\P}_N\Ll[x\fty_N(n) \text{ is not seen during } [T\fty_N(n)+N, T\fty_N(n)+M\msft]\Rr] \xrightarrow[N \to \infty]{} 1.
\end{equation}
Let us see why this would imply part (2) of the Proposition, and then part (1) as well. First, outside of an event whose probability can be made as close to $0$ as desired, we can always assume that $T\de_N(k+1) \le M \msft$, provided we choose $M$ large enough. Hence, \eqref{e:goood} implies
\begin{equation}
\label{e:goood2}
\ov{\P}_N\Ll[x\fty_N(n) \text{ is not seen during } [T\fty_N(n)+N, T\de_N(k+1)]\Rr] \xrightarrow[N \to \infty]{} 1.
\end{equation}
Note that we always have $x\de_N(1) = x\fty_N(1)$ and $T\fty_N(1) = T\de_N(1)$. Next, two things could make $x\de_N(2)$ and $x\fty_N(2)$ different: (i) that a deep trap is discovered during the time interval $(T\fty_N(1), \td{T}\fty_N(1)]$ (this would define $x\de_N(2)$, but not $x\fty_N(2)$) ; or (ii) that the site $x\de_N(1)$ is seen during the time $(\td{T}\fty_N(1), T\de_N(2)]$ (this would define $x\fty_N(2)$, but not $x\de_N(2)$). With probability tending to $1$, we have
$$
N \le \td{T}\fty_N(1) - T\fty_N(1) \le N^3,
$$
so the $\ov{\P}_N$-probability of the events described in (i)-(ii) tend to $0$ as $N$ tends to infinity by \eqref{e:goood2} and the fact that the event in \eqref{event1} has $\ov{\P}_N$-probability tending to $1$. The reasoning can be continued to cover any $n$, and we obtain part (2) of the proposition. Finally, part (2) and \eqref{e:goood2} ensure that indeed, the event in \eqref{event2} has $\ov{\P}_N$-probability tending to $1$ as well, and the proof is complete.

There remains to prove \eqref{e:goood}. In order to do so, we will need the following lemma.
\begin{lem}
\label{l:return}
We write $t_+ = \max(t,0)$. For every $M > 0$ and $x \in V_N$, we have
\begin{multline}
\label{e:return}
\EEt_{x}\Ll[  \int_N^{+\infty} e^{-(t-M\msft)_+/N^2} \ \1_{X_t=x\fty_N(n)} \ \d t \Rr]  \\
 \ge %P^{M}_N(n) 
\PPt_{x}\Ll[ \exists s \in [N,M \msft] : X_s = x\fty_N(n) \Rr]
\ G_N(\tau\fty_N(n)).
\end{multline}
\end{lem}
\begin{proof}[Proof of Lemma~\ref{l:return}]
	Let 
$$
\mcl{T} = \inf \{ s \ge N : X_s = x\fty_N(n) \}.
$$
The left-hand side of \eqref{e:return} is larger than
\begin{multline*}
\EEt_{x}\Ll[ \1_{\mcl{T} \le M \msft} \int_N^{+\infty} e^{-(t-M\msft)_+/N^2} \ \1_{X_t=x\fty_N(n)} \ \d t \Rr] \\
\ge \EEt_{x}\Ll[ \1_{\mcl{T} \le M \msft} \int_\mcl{T}^{+\infty} e^{-(t-\mcl{T})/N^2} \ \1_{X_t=x\fty_N(n)} \ \d t \Rr].
\end{multline*}
Applying the Markov property at time $\mcl{T}$, we obtain the result.
\end{proof}
For $M > 0$, we study the $\PPtu$-probability that the site $x\fty_N(n)$ is visited during the time interval $[T\fty_N(n)+N,T\fty_N(n)+M\msft]$. Using the Markov property at time $T\fty_N(n)$ and the Lemma at $x = X_{T\fty_N(n)}$, we see that this probability is bounded by the $\EEtu$-expectation of
\begin{equation}
\label{boundprob1}
\EEt_{X_{T\fty_N(n)}}\Ll[  \int_N^{+\infty} e^{-(t-M\msft)_+/N^2} \ \1_{X_t=x\fty_N(n)} \ \d t \Rr] \ G_N(\tau\fty_N(n))^{-1}.
\end{equation}
By Proposition~\ref{p:fuite-sg}, we know that the expectation in \eqref{boundprob1} is bounded by
$$
\int_N^{+\infty} e^{-(t-M\msft)_+/N^2} \ (2^{-N} + e^{-2t}) \ \d t \le \frac{M\msft + N^2}{2^{N}} + e^{-2N},
$$
and the latter decays exponentially fast since $\log(\msft) \sim \ov{c} N$ with $\ov{c} < \log 2$. Note that the bound thus obtained on the expectation in \eqref{boundprob1} holds uniformly over the environment. 

On the other hand, by Proposition~\ref{p:greenfty}, outside of an event whose $\ov{\P}_N$-probability tends to $0$ with $N$, we have
$$
N^2 \phi(a_N)^2 G_N(\tau\fty_N(n)) \ge 1/2.
$$
Since $N^2 \phi(a_N)^2 \le 2N^2a_N^2$ grows sub-expoentially with $N$ (recall \eqref{defphi} and \eqref{anconstraint}), we have shown that 
\begin{equation}
\label{xftyvisited}
\ov{\P}_N\Ll[ x\fty_N(n) \text{ is visited during } [T\fty_N(n)+N,T\fty_N(n)+M\msft] \Rr] \xrightarrow[N \to \infty]{} 0.
\end{equation}
We now explain how we can infer \eqref{e:goood} from this. We begin by proceeding as in the proof of Lemma~\ref{l:tight} to show that 
$$
\ov{\P}_N\Ll[ \max_{z \ssim x\fty_N(n)} E_{z} \le 2 \sqrt{\log N}  \Rr] \xrightarrow[N \to \infty]{} 1.
$$
Since on the other hand, $E_{x\fty_N(n)} \ge b_n \sim \sqrt{2 \ov{c} N}$, any visit to a neighbour of $x\fty_N(n)$ is followed (with probability close to $1$) by a visit to $x\fty_N(n)$ itself shortly afterwards (by a reasoning identical to the proof of Proposition~\ref{p:gotouch}), so that the probability in \eqref{e:goood} cannot be large if the one in \eqref{xftyvisited} is close to $0$. 
\end{proof}
\begin{prop}
\label{p:greende}
Under the measure $\ov{\P}_N$ and for every $n$, $N^2\phi(a_N)^2 \ G_N(\tau\de_N(n))$ converges to $1$ in probability.
\end{prop}
\begin{proof}
This is a direct consequence of Proposition~\ref{p:greenfty} and part (2) of Proposition~\ref{p:goood}. 
\end{proof}
Recall that we write $H\de_N(n)$ for the first time the site $x\de_N(n)$ is visited, and that we defined $e\de_N(n)$ by the relation
$$
\int_{H\de_N(n)}^{H\de_N(n)+\td{e}_N(n)} \1_{X_t = x\de_N(n)} \ \d t = G_N(\tau\de_N(n)) \ e\de_N(n).
$$
\begin{prop}
\label{p:joint-conv}
Under $\ov{\P}_N$, the family of triples
$$
((\msft^{-1}  (T\de_N(n)-T\de_N(n-1)),\ B_N^{-1} \ \tau_{N,x\de_N(n)}, \ e\de_N(n)))_{n \in \N}
$$
jointly converges in distribution as $N$ tends to infinity (where we let $T\de_N(0) = 0$). Let us write
\begin{equation}
\label{triplelim}
((T\de(n)-T\de(n-1),\ \tau\de(n),\ e\de(n))_{n \in \N}
\end{equation}
for limiting random variables, which we assume to be defined on the $\P$-probability space for convenience (these can be taken as an extension of those appearing in Proposition~\ref{p:poisson}, and we fix $T_N\de(0) = 0$). Their joint distribution is described as follows. The family of triples in \eqref{triplelim} is a family of i.i.d.\ triples. Moreover, the random variables $T\de(1)$, $\tau\de(1)$ and $e\de(1)$ are themselves independent. Their respective distributions are exponential of parameter $\delta^{-\alpha}$, the law displayed in \eqref{defdistrib}, and exponential of parameter~$1$.
\end{prop}
\begin{proof}
We introduce a slight modification of the sequence of times $(T\fty_N(n))_{n \in \N}$, that we write $(T'_N(n))_{n \in \N}$. They are constructed in the same way as the sequence $(T\fty_N(n))_{n \in \N}$, except for the fact that $T'_N(1)$ is the first time when a deep trap is \emph{visited} by the walk (instead of being simply discovered). More precisely, we let $\td{T}'_N(0) = 0$,
$$
T'_N(1) = \inf\{t : \tau_{N,X_t} \ge \delta B_N\},
$$
and then define recursively, for any $n \in \N$,
\begin{equation*}
%\label{deftdT0}
T'_N(n) = \td{T}'_N(n-1) + T'_N(1) \circ \Theta_{\td{T}'_N(n-1)},
\end{equation*}
\begin{equation*}
%\label{deftdT}
\td{T}'_N(n) = T'_N(n) + \td{e}_N(n) + \msf{T}_N \circ \Theta_{T'_N(n)+\td{e}_N(n)}.
\end{equation*}
We let $x'_N(n)$ be the deep trap visited at time $T'_N(n)$, $\tau'_N(n) = \theta_{x'_N(n)} \ \tau_N$, and $e'_N(n)$ be such that
$$
\int_{T'_N(n)}^{T'_N(n)+\td{e}_N(n)} \1_{X_t = x'_N(n)} \ \d t = e'_N(n) \ G_N(\tau'_N(n)).
$$
As follows from Proposition~\ref{p:stationarity}, the family of triples
$$
((T'_N(n)-\td{T}'_N(n-1),\tau_{N,x'_N(n)}, e'_N(n)))_{n \in \N}
$$ 
is a family of i.i.d.\ triples under $\PPtu$. Moreover, the pair $(T'_N(n)-\td{T}'_N(n-1),\tau_{N,x'_N(n)})$ is independent of $e'_N(n)$ under $\PPtu$. Indeed, conditionally on the history of the walk up to time $T'_N(n)$, the random variable $e'_N(n)$ is simply an exponential random variable of parameter $1$ (the argument is the same as the one in the paragraph below \eqref{defede}). Note that from this argument, we also learn what is the distribution of $e'_N(n)$. 

For definiteness, let us look at the Fourier transform of the random variables under consideration. 
%{\sc Is this really needed? We could argue that the differences between $T'_N(n)-\td{T}'_N(n-1)$ and $T\de_N(n)-{T}\de_N(n-1)$ goes to $0$ in probability (after proper normalization)and that $x'_N(n)$ and $x\de_N(n)$ coincide on a set of measure tending to $1$. We do not really use the Fourier transforms.} 
By the above remarks, for every $\tau_N$, every $k \in \N$ and every $x_n, y_n,z_n \in \R$, we have
\begin{multline}
\label{niceindep}
\EEtu\Ll[ \exp\Ll(\sum_{n = 1}^k i x_n \msft^{-1} (T'_N(n)-\td{T}'_N(n-1)) + i y_n \ B_N^{-1} \ \tau_{N,x'_N(n)} + i z_n \ e'_N(n) \Rr)\Rr] \\
= \EEtu\Ll[ \exp\Ll(\sum_{n = 1}^k i x_n \msft^{-1} (T'_N(n)-\td{T}'_N(n-1)) + i y_n \ B_N^{-1} \ \tau_{N,x'_N(n)} \Rr) \Rr] \ \prod_{n = 1}^{k} \frac{1}{1-iz_n}.
\end{multline}

We now argue that as $N$ tends to infinity, the difference between
\begin{equation}
\label{replace}
\ov{\E}_N\Ll[ \exp\Ll(\sum_{n = 1}^k i x_n \msft^{-1} (T'_N(n)-\td{T}'_N(n-1)) + i y_n \ B_N^{-1} \ \tau_{N,x'_N(n)} + i z_n \ e'_N(n) \Rr)\Rr]
\end{equation}
and 
\begin{equation}
\label{replaced}
\ov{\E}_N\Ll[ \exp\Ll(\sum_{n = 1}^k i x_n \msft^{-1} (T\de_N(n)-{T}\de_N(n-1)) + i y_n \ B_N^{-1} \ \tau_{N,x\de_N(n)} + i z_n \ e\de_N(n) \Rr)\Rr]
\end{equation}
tends to $0$ as $N$ tends to infinity. First, we have already seen several times that the difference between $\td{T}'_N(n-1)$ and $T'_N(n-1)$ is subexponential (to be precise, it is smaller than $N^3$ with probability tending to $1$). So in \eqref{replace}, if we replace $\td{T}'_N(n-1)$ by $T'_N(n-1)$, then the error made tends to $0$ as $N$ tends to infinity. 

The first deep trap discovered is $x\de_N(1)$ and at time $T\de_N(1)$. On the event $\mcl{F}_N(k)$, we know that $x\de_N(1)$ is actually visited during the time interval $[T\de_N(1), T\de_N(1)+N]$, and that no other deep trap is discovered during this time interval, so $x\de_N(1) = x'_N(1)$. On this event, the definitions of $e'_N(1)$ and $e\de_N(1)$ actually coincide. Note also that on this event, we have $T\de_N(1) \le T'_N(1) \le T\de_N(1)+N$. Proceeding as in the proof of part (2) of Proposition~\ref{p:goood}, we can iterate this reasoning to subsequent sites, and obtain that indeed the difference between \eqref{replace} and \eqref{replaced} converges to $0$ as $N$ tends to infinity.

In order to conclude, we want to take the $\E$-expectation in \eqref{niceindep} and identify the right-hand side as the ``correct answer''. Note that the limit of
$$
\ov{\E}_N\Ll[ \exp\Ll(\sum_{n = 1}^k i x_n \msft^{-1} (T\de_N(n)-{T}\de_N(n-1)) + i y_n \ B_N^{-1} \ \tau_{N,x\de_N(n)} \Rr)\Rr]
$$
as $N$ tends to infinity is known and given by Proposition~\ref{p:distribtrap}. By the fact that we can replace $x\de_N(n)$ by $x'_N(n)$ above, we obtain explicitly the limit of the $\E$-expectation of the right-hand side of \eqref{niceindep}, and this is also the limit of
$$
\ov{\E}_N\Ll[ \exp\Ll(\sum_{n = 1}^k i x_n \msft^{-1} (T\de_N(n)-{T}\de_N(n-1)) + i y_n \ B_N^{-1} \ \tau_{N,x\de_N(n)} + i z_n \ e\de_N(n) \Rr)\Rr],
$$
so the proof is complete.
\end{proof}
%
%
%
%
%%%%%%%%%%%%%%%%%%%%%%%%%%%%%%%%%%%%%%%%%%%%%%%%%%%%%%%%%%%%%%
%%%%%%%%%%%%%%%%%%%%%%%%%%%%%%%%%%%%%%%%%%%%%%%%%%%%%%%%%%%%%%
%
%
%
\section{Convergence of the clock process}
\label{s:convclock}
\setcounter{equation}{0}
In this section, we obtain the scaling limit of the clock process.
\begin{thm}
\label{t:convclock} Recall that we assume \eqref{anconstraint} with $\ov{a} < 1/20$, that $(\msft)$ satisfies $\log \msft \sim \ov{c} N$ with $\ov{c} \in (0,\log(2))$, that $\msfd$ is given by Theorem~\ref{t:lln} (and satisfies~\eqref{e:ratios}), that $\phi$, $B_N$ and $\alpha$ are defined in \eqref{defphi}, \eqref{defBn} and \eqref{defalpha} respectively, and that we assume $\alpha < 1$.
Let 
\begin{equation}
\label{defclock}
C(t) = \int_0^t \tau_{N,X_s} \ \d s
\end{equation}
be the \emph{clock process}, and let $C_N$ be its rescaled version, defined by
\begin{equation}
\label{defrescaled-clock}
C_N(t) = N^2 \phi(a_N)^2 B_N^{-1} \ C(t\msft).
\end{equation}
Under $\ov{\P}_N$, the rescaled clock process $C_N$ converges in distribution, for the $M_1$ topology and as $N$ tends to infinity, to the $\alpha$-stable subordinator $\mcl{C}$ whose L\'evy measure is given by
$$
\Gamma(\alpha+1) \ \frac{\alpha \ \d z}{z^{\alpha+1}},
$$
where $\Gamma$ is Euler's Gamma function.
\end{thm}
\begin{rem}
The $M_1$ topology is defined on the space of c\`ad-l\`ag functions, we refer to \cite[(3.3.4)]{whitt} for a precise definition. Here, we understand the convergence of processes defined on $\R_+$ as the convergence of the restriction of the processes on $[0,t]$ for every $t > 0$.
\end{rem}
\begin{rem} 
In Theorem \ref{t:convclock}, as in the next Theorem \ref{t:convage} on the convergence of the age process, we impose two restrictions on the time scale, namely $\ov{c}\in(0,\log(2))$ and $\ov{c}<\beta^2/2$, this last inequality being equivalent to the condition $\alpha<1$. 

We note that we thus cover all possible time scales. Indeed, the convergence of the clock process towards a subordinator can only hold on time scales that are shorter than the equilibrium time. Let us see that this corresponds to the restriction $\ov{c} < \log(2) \wedge \beta^2/2$.

Remember that we are considering a time $\msft$ such that $\log(\msft) \sim \ov{c} N$. This corresponds to looking at the process $Z$ up to times of order $B_N$ (on the exponential scale), with $\log(B_N)\sim \beta\sqrt{2\ov{c}}N$. 

On the one hand, by general facts on reversible Markov chains, we know that the process $Z$ has reached equilibrium for times larger than the inverse spectral gap. As can be seen by adapting the arguments in \cite{fikp}, the inverse spectral gap of the dynamics we are considering here is of order $\exp(\beta\beta_c N)$, where $\beta_c=\sqrt{2\log(2)}$. Imposing that $B_N$ is smaller than $\exp(\beta\beta_c N)$ on the exponential scale leads to the condition $\ov{c}<\log(2)$.  

It turns out that the time needed for the dynamics to reach equilibrium when started with the uniform law (as here) may be much shorter than the inverse spectral gap in the high-temperature regime, and should then rather be estimated using generalized Poincar\'e inequalities. Results in \cite{ma} show that if $\beta>\beta_c$, then $Z$ is indeed close to equilibrium by a time of order $\exp(\beta^2N)$. (The dynamics considered in \cite{ma} is rather the RHT dynamics, but the argument readily generalizes to the dynamics considered here as soon as we assume $a_N=o(\sqrt{N})$, and in particular under our present assumption (\ref{anconstraint}).) Requiring that, on the exponential scale, $B_N$ is smaller than $\exp{\beta^2N}$ is equivalent to assuming that $\ov{c}<\beta^2/2$. 
\end{rem}
Let 
$$
C\de_N(t) = N^2 \phi(a_N)^2 B_N^{-1} \int_0^{t\msft} \tau_{N,X_s} \1_{\tau_{N,X_s} \ge \delta B_N} \ \d s.
$$
$$
\mcl{C}\de_N(t) = N^2 \phi(a_N)^2 B_N^{-1} \sum_{n = 1}^{+\infty} \tau_{N,x\de_N(n)} \ e\de_N(n) \ G_N(\tau\de_N(n)) \ \1_{t \ge \msft^{-1} \ T\de_N(n)}.
$$
\begin{prop}
\label{p:approxclock}
For every $T > 0$, the $M_1$ distance on between the processes $C\de_N$ and $\mcl{C}\de_N$ restricted to $[0,T]$ converges to $0$ in $\ov{\P}_N$-probability.
\end{prop}
\begin{proof}
The proof is similar to that of \cite[Proposition 6.3]{scaling}. Let $\td{B}_N^{-1}$ be such that
\begin{equation}
\label{deftdB}
\td{B}_N^{-1} = N^2 \phi(a_N)^2 B_N^{-1}.
\end{equation}
Let $T > 0$. We begin by showing that for every $n \le k$, the $M_1$ distance between
\begin{equation}
\label{process1}
\td{B}_N^{-1} \int_0^{t\msft \wedge T\de_N(k+1)} \tau_{N,X_s} \1_{X_s = x\de_N(n)} \ \d s
\end{equation}
and
\begin{equation}
\label{process2}
\td{B}_N^{-1}\tau_{N,x\de_N(n)} \ e\de_N(n) \ G_N(\tau\de_N(n)) \ \1_{t \ge \msft^{-1} \ T\de_N(n)},
\end{equation}
as processes defined on $[0,T]$, tends to $0$ in $\ov{\P}_N$-probability. These two processes are increasing. The second process is equal to $0$ until time $\msft^{-1} T\de_N(n)$, when it jumps to the value 
$$
\td{B}_N^{-1} \tau_{N,x\de_N(n)} \ e\de_N(n) \ G_N(\tau\de_N(n)) = \td{B}_N^{-1} \tau_{N,x\de_N(n)} \ \int_{H\de_N(n)}^{H\de_N(n)+\td{e}_N(n)} \1_{X_s = x\de_N(n)} \ \d s
$$
(recall that $H\de_N(n)$ is the first time the walk visits $x\de_N(n)$).
We know from Proposition~\ref{p:goood} that with probability close to $1$, the site $x\de_N(n)$ is not visited by the walk during the time interval $[0,T\de_N(k+1)] \setminus [T\de_N(n), T\de_N(n)+N]$. We know also from Proposition~\ref{p:gotouch} that with probability tending to $1$, $x\de_N(n)$ is indeed visited by the walk during the time interval $[T\de_N(n),T\de_N(n) + N]$. As a consequence, with probability tending to $1$, the process in \eqref{process1} is $0$ until time $\msft^{-1} T\de_N(n)$, and then is constant equal to 
\begin{eqnarray*}
&&\td{B}_N^{-1}\tau_{N,x\de_N(n)} \int_{T\de_N(n)}^{T\de_N(n)+N} \1_{X_s = x\de_N(n)} \ \d s \\
&= &\td{B}_N^{-1}\tau_{N,x\de_N(n)} \int_{H\de_N(n)}^{H\de_N(n)+\td{e}_N(n)} \1_{X_s = x\de_N(n)} \ \d s
\end{eqnarray*}
for all times greater than $\msft^{-1} (T\de_N(n)+N)$. 

In words, we have argued that the two increasing processes in \eqref{process1} and \eqref{process2} are with high probability equal at every time, except possibly during a time interval whose length $\msft^{-1} N$ tends to $0$ as $N$ tends to infinity. Moreover, by Proposition~\ref{p:joint-conv}, their final value is bounded in probability. These observations ensure that we can construct a parametrization of the completed graphs of these processes, as defined in \cite[(3.3.3)]{whitt}, showing that the $M_1$ distance between the two processes tends to $0$, provided we can guarantee that the interval where the two processes possibly differ is far from the endpoints of the interval $[0,T]$. But this is so with high probability, since we know that the limiting distribution of $\msft^{-1} \ T\de_N(n)$ has a density, see Proposition~\ref{p:joint-conv}.

From this observation, we obtain that the $M_1$ distance between 
$$
\sum_{n = 1}^k \td{B}_N^{-1} \int_0^{t\msft \wedge T\de_N(k+1)} \tau_{N,X_s} \1_{X_s = x\de_N(n)} \ \d s
$$
and
\begin{equation}
\label{e:approxclock1}
\td{B}_N^{-1}\sum_{n = 1}^k \tau_{N,x\de_N(n)} \ e\de_N(n) \ G_N(\tau\de_N(n)) \ \1_{t \ge \msft^{-1} \ T\de_N(n)}
\end{equation}
(as processes defined on $[0,T]$) converges to $0$ in $\ov\P_N$-probability. Proposition~\ref{p:goood} ensures that with probability tending to $1$, no other deep trap than $(x\de_N(n))_{n \le k}$ is visited up to time $T\de_N(k+1)$. Hence, we obtain that the $M_1$ distance between
$$
\td{B}_N^{-1} \int_0^{t\msft \wedge T\de_N(k+1)} \tau_{N,X_s} \1_{\tau_{N,X_s} \ge \delta B_N} \ \d s
$$
and \eqref{e:approxclock1} converges to $0$ as well. Now, in view of Proposition~\ref{p:joint-conv}, this is sufficient, since by choosing $k$ very large, we can make the probability that $\msft^{-1} \ T\de_N(k+1) \le T$ as close to $0$ as we wish.
\end{proof}

\begin{prop}
\label{p:convcde}
Under $\ov{\P}_N$, the process $C\de_N$ converges in distribution, for the $M_1$ topology and as $N$ tends to infinity, to the process $\mcl{C}\de$ defined by
$$
\mcl{C}\de(t) = \sum_{n = 1}^{+\infty} \tau\de(n) \ e\de(n) \ \1_{t \ge T\de(n)},
$$
where the random variables 
$$
((T\de(n),\ \tau\de(n),\ e\de(n))_{n \in \N}
$$
are distributed as in Proposition~\ref{p:joint-conv}.
\end{prop}
\begin{proof}
By Proposition~\ref{p:approxclock}, it suffices to study the convergence of $\mcl{C}\de_N$. The proof is then identical to the proof of \cite[Proposition 7.5]{scaling}, so we only briefly mention the argument. We use Skorokhod's representation theorem, so that we can assume that the convergence described in Proposition~\ref{p:joint-conv} is actually an almost sure convergence. With this in hand, the convergence of the (Skorokhod representation of the) process $\mcl{C}\de_N$ to the (Skorokhod representation of the) process $\mcl{C}\de$ for the $M_1$ topology easily follows.
\end{proof}

\begin{proof}[Proof of Theorem~\ref{t:convclock}]
Let $T > 0$. We begin by showing that
\begin{equation}
\label{negligstat0}
\limsup_{N \to +\infty} \E\Ll[ \sup_{[0,T]} \Ll| C\de_N - C_N \Rr| \Rr] = O( \delta^{1-\alpha}) \qquad (\delta \to 0).
\end{equation}
The process $C\de_N-C_N$ is increasing, and thus
\begin{equation}
\label{negligstat}
\E\Ll[ \sup_{[0,T]} \Ll| C\de_N - C_N \Rr| \Rr] = \td{B}_N^{-1} \ov{\E}_N\Ll[\int_0^{T\msft}\tau_{N,X_s} \ \1_{\tau_{N,X_s} \le \delta B_N} \ \d s\Rr],
\end{equation}
where we recall that $\td{B}_N$ was defined in \eqref{deftdB}. Note that the integrand above is simply a function of the environment viewed by the particle, which is stationary under $\ov{\P}_N$. The right-hand side of \eqref{negligstat} is thus equal to
$$
\td{B}_N^{-1} T \msft \ \E\Ll[\tau_{N,0} \ \1_{\tau_{N,0} \ge \delta B_N}\Rr].
$$
Finally, note that by Proposition~\ref{p:ratios}, we have
$$
\td{B}_N^{-1} \msft \sim B_N^{-1}  \msfd.
$$ 
Identity \eqref{negligstat0} then follows from \eqref{e:negligpasprofonds}.

We want to complete the following diagram
\begin{displaymath}
\label{diagram}
\begin{array}{cccc}
C\de_N & \xrightarrow[N \to \infty]{} &  \mcl{C}\de & \\
\downarrow &  & \downarrow & (\delta \to 0) \\
C_N & \xrightarrow[N \to \infty]{} &  \mcl{C}, &
\end{array}
\end{displaymath}
where each arrow represents convergence in distribution for the $M_1$ topology. Since the uniform norm controls the $M_1$ distance, we have proved with \eqref{negligstat0} that the convergence of $C\de_N$ to $C_N$ holds uniformly in $\delta$. Proposition~\ref{p:convcde} ensures that $C\de_N$ converges indeed to $\mcl{C}\de$. In order to deduce the convergence of $C_N$ to $\mcl{C}$ (and thus prove Theorem~\ref{t:convclock}), what remains to be done  (see \cite[Theorem~4.2]{bil}) is thus simply to show that $\mcl{C}\de \xrightarrow[\delta \to 0]{} \mcl{C}$.

Let us thus proceed to prove this convergence. It is clear from its definition that $\mcl{C}\de$ is a subordinator. Let $\psi_\delta(\lambda)$ be the Laplace exponent of $\mcl{C}\de$, defined by
$$
\E[e^{- \lambda \mcl{C}\de(t)}] = e^{-t \psi_\delta(\lambda)}.
$$
Using the definition of $\mcl{C}\de(t)$, we obtain that
$$
\E[e^{- \lambda \mcl{C}\de(t)}] = 1-\frac{t}{\delta^{\alpha}} + \frac{t}{\delta^{\alpha}} \E[e^{-\lambda e\de(1) \tau\de(1)}] + O(t^2),
$$
and thus
$$
\psi_\delta(\lambda) = \frac{1}{\delta^{\alpha}} \ \E\Ll[ 1-e^{-\lambda e\de(1) \tau\de(1)} \Rr].
$$
Given the definition of $(e\de(1),\tau\de(1))$ (see Proposition~\ref{p:joint-conv}), we can compute this last quantity. It is equal to
$$
\frac{1}{\delta^{\alpha}} \int_{\substack{y \ge 0 \\ z \ge \delta}} (1-e^{-\lambda y  z}) e^{-y} \frac{\alpha \delta^\alpha}{z^{\alpha + 1}} \ \d y \ \d z.
$$
The $\delta^{\alpha}$ cancel out. We make the change of variables $u = yz$ to obtain
$$
\int_{u = 0}^{+\infty} (1-e^{-u}) \frac{\alpha}{u^{\alpha + 1}} \underbrace{\int_{y = 0}^{u/\delta} y^\alpha  e^{-y} \ \d y}_{\xrightarrow[\delta \to 0]{} \Gamma(\alpha+1)} \ \d u.
$$
By the monotone convergence theorem, we thus obtain that $\psi_\delta(\lambda)$ converges to
$$
\Gamma(\alpha+1) \int_{u = 0}^{+\infty} (1-e^{-u}) \frac{\alpha}{u^{\alpha + 1}}  \ \d u,
$$
and this is indeed the Laplace exponent of $\mcl{C}$. We have thus shown that for every $t\ge 0$, the Laplace transform of $\mcl{C}\de(t)$ converges to the Laplace transform of $\mcl{C}(t)$. Since both processes are subordinators, this ensures convergence in the sense of finite-dimensional distributions. Finally, for increasing processes, convergence of finite-dimensional distributions imply convergence for the $M_1$ topology (this can be seen directly from the definition, or using the tightness criterion for increasing processes given in \cite[Lemma~8.4]{scaling}; in our case, we could also easily show that the difference between $\mcl{C}$ and $\mcl{C}\de$ converges to $0$ in probability for the supremum norm). This finishes the proof.
\end{proof}

%
%
%
%
%%%%%%%%%%%%%%%%%%%%%%%%%%%%%%%%%%%%%%%%%%%%%%%%%%%%%%%%%%%%%%
%%%%%%%%%%%%%%%%%%%%%%%%%%%%%%%%%%%%%%%%%%%%%%%%%%%%%%%%%%%%%%
%
%
%
\section{Convergence of the age process}
\label{s:convage}
\setcounter{equation}{0}

We now discuss the convergence of the age process. 
In the next theorem as in the next proofs, we use the following definitions. 
Let $f$ be a c\`ad-l\`ag function. Its inverse is defined as $$f^{-1}(t)=\inf\{y\,:\, f(y)>t\}.$$ 
Let $f$ be a continuous non-decreasing function. Then $f'(t)$ denotes the right derivative of $f$ at point $t$ (provided it exists). 
%If $f$ is c\`ad-l\`ag, non-decreasing and not continuous at point $t$ then $f'(t)=f(t)-f(t-)$ denotes the jump of $f$ at point $t$. 

\begin{thm}
\label{t:convage} Recall the definitions and assumptions summarized in Theorem~\ref{t:convclock}.
Let 
\begin{equation}
\label{defage}
A(t)=\tau_{N,Z_t}
\end{equation}
be the \emph{age process}, and let $A_N$ be its rescaled version, defined by
\begin{equation}
\label{defrescaled-age}
A_N(t)=B_N^{-1} A\Ll(t\frac{B_N}{N^2\phi(a_N)^2}\Rr). 
\end{equation}
Under $\ov{\P}_N$, the rescaled age process $A_N$ converges in distribution, for the $L_{1,loc}$ topology and as $N$ tends to infinity. 

The limiting process, say $\mcl{Z}$, may be constructed as follows: let $\ups$ be an $\alpha$-stable subordinator, and let  
$$\mcl{V}_t=\int_0^t T_s \ \d\ups_s,$$ where 
$(T_t)_{t\ge 0}$ is an independent family of independent, mean $1$, exponential random variables. 
Let $\mcl{W}=\mcl{V}^{-1}$ be the inverse of $V$. Finally, define the process 
$$\mcl{Z}_t=\ups_{\mcl{W}_t}-\ups_{\mcl{W}_t-}.$$
\end{thm}

In the sequel, we use the notation $\mu(x)=\ups(x)-\ups(x-)$. 
Note that $\mcl Z$ is a self-similar process of index $1$. One checks that its law does not depend on the choice of the $\alpha$-stable subordinator $\ups$.

The following proof follows a similar strategy as in \cite{fm_aging}. 
Recall that the processes $X$ and $Z$ are connected through 
$$Z_t=X_{C^{-1}(t)},$$ 
where $C$ is the  clock process  
$$C(t)=\int_0^t \tau_{N,X_s}\ \d s.$$ 
Observe that 
$$A(t)=C'\circ C^{-1}(t).$$ 
%where $C'$ is the right derivative of $C$. 
%
As in the proof of Theorem \ref{t:convclock}, we consider the rescaled version of the clock process 
$$C_N(t)=\td{B}_N^{-1} \int_0^{t\msft} \tau_{N,X_s} \ \d s $$ 
(Recall that $\td{B}_N$ was defined in \eqref{deftdB}).
Note that 
$$
C^{-1}(\td{B}_N t) = \msft C_N^{-1}(t),
$$
and thus
$$ A_N(t)= N^{-2}\phi(a_N)^{-2}\msft^{-1} C_N'\circ C_N^{-1}(t)=B_N^{-1}\tau_{N,X(\msft C_N^{-1}(t))}.$$ 
Also, as in the proof of theorem \ref{t:convclock}, we introduce truncated processes : 
$$C\de_N(t)=\td{B}_N^{-1} \int_0^{t\msft} \tau_{N,X_s} \1_{\tau_{N,X_s}\ge \delta B_N}\ \d s $$ 
and 
$$A\de_N(t) = N^{-2}\phi(a_N)^{-2}\msft^{-1} (C\de_N)'\circ (C\de_N)^{-1}(t) ;$$ 
$$
\mcl{C}\de_N(t) = \td{B}_N^{-1} \sum_{n = 1}^{+\infty} \tau_{N,x\de_N(n)} \ e\de_N(n) \ G_N(\tau\de_N(n)) \ \1_{t \ge \msft^{-1} \ T\de_N(n)} 
$$ 
and 
$$\mcl{A}\de_N(t) = B_N^{-1}\tau_{N,x\de_N(n)} \quad \text{for } t \text{ such that } (\mcl{C}\de_N)^{-1}(t)=\msft^{-1}T\de_N(n).$$ 
%where $(\mcl{C}\de_N)^{-1}(t)=\msft^{-1}T\de_N(n)$. 

The proof of the next proposition is similar to the proof of Proposition \ref{p:approxclock}. 
\begin{prop}
\label{p:approxage}
For every $T > 0$, the $M_1$ distance between the processes $A\de_N$ and $\mcl{A}\de_N$ restricted to $[0,T]$ converges to $0$ in $\ov{\P}_N$-probability.
\end{prop}
\begin{proof} 
Let $k\ge1$.  Observe that on the event $\mcl{F}\de_N(k)$, the two processes $A\de_N$ and $\mcl{A}\de_N$ take the same values (namely the sequence $B_N^{-1}\tau_{N,x\de_N(n)}$, for $n\in\{1,\ldots,k\}$) in the same order. 
The holding time at a trap $x\de_N(n)$ is then \\ $C\de_N(\msft^{-1}H\de_N(n+1))-C\de_N(\msft^{-1}H\de_N(n))$ for the process $A\de_N$, and\\  
$\mcl{C}\de_N(\msft^{-1}T\de_N(n+1)-)-\mcl{C}\de_N(T\de_N(n)-)
=\td{B}_N^{-1}  \ e\de_N(n) \ G_N(\tau\de_N(n)) \tau_{N,x\de_N(n)}$  for the process $\mcl{A}\de_N$. 
On $\mcl{F}\de_N(k)$, these two quantities coincide. Therefore the two processes $A\de_N$ and $\mcl{A}\de_N$ coincide on $\mcl{F}\de_N(k)$ until time $C\de_N(\msft^{-1}H\de_N(k+1))$. 

Proposition \ref{p:goood} implies that the $\ov{\P}_N$-probability of $\mcl{F}\de_N(k)$ tends to $1$. Therefore we also have that  the $M_1$ distance between the processes 
$(A\de_N(t))_{t\le C\de_N(\msft^{-1}H\de_N(k+1)) }$ and $(\mcl{A}\de_N(t))_{t\le C\de_N(\msft^{-1}H\de_N(k+1))}$ tends to $0$. 

Finally we argue that, by choosing $k$ large, we can make the probability that $C\de_N(\msft^{-1}H\de_N(k+1))\le T$ as small as we want. 
On the one hand, the process $C\de_N$ has a limit in distribution, $\mcl{C}\de$, which is described in Proposition \ref{p:convcde} and satisfies 
$\lim_{t\to+\infty}\mcl{C}\de(t)=+\infty$. On the other hand, we observe that 
$\msft^{-1}H\de_N(k+1)$ is larger than $\msft^{-1}T\de_N(k+1)$ and that the limit in distribution of $\msft^{-1}T\de_N(k+1)$, which we wrote $T\de (k+1)$, was obtained in Proposition~\ref{p:joint-conv} and satisfies 
$\lim_{k\to+\infty} T\de (k+1)=+\infty$. 
\end{proof}

Next, we introduce truncated versions of the process $\mcl{Z}$: enumerate the set $\{x\geq0:\, \mu(x)>\delta\}$ as 
$$\{x\geq0:\, \mu(x)>\delta\}=\{\upsilon_1<\upsilon_2<\ldots\},$$ 
and let 
$$
\bvd_s=\sum_{i=1,2,\ldots;\upsilon_i\leq s}\mu(\upsilon_i)\,T_{\upsilon_i},\quad s\geq0,
$$ 
where an empty sum equals $0$. We consider the inverse of $\bvd$, say $\bwd$,  and let
$$
\bzd_t=\mu(\bwd_t),\quad t\geq0.
$$ 

\begin{rem}\label{r:bzd} 
The trajectories of the process $\bzd$ are easy to describe: $\bzd$ successively takes the values $\mu(\upsilon_1), \mu(\upsilon_2), \ldots$ with respective holding times $\mu(\upsilon_1)T_{\upsilon_1}$, $\mu(\upsilon_2)T_{\upsilon_2}, \ldots$
\end{rem}

\begin{prop}\label{p:convade} 
Under $\ov{\P}_N$, the process $A\de_N$ converges in distribution, for the $M_1$ topology and as $N$ tends to infinity, to the process $\mcl{Z}\de$. 
\end{prop}

\begin{proof}

By Proposition~\ref{p:approxage}, it suffices to study the convergence of $\mcl{A}\de_N$. We use Skorokhod's representation theorem, so that we can assume that the convergence described in Proposition~\ref{p:joint-conv} is actually an almost sure convergence. With this in hand, the convergence of the (Skorokhod representation of the) process $\mcl{A}\de_N$ to the (Skorokhod representation of the) process $\mcl{Z}\de$ for the $M_1$ topology easily follows: compare the description of the process $\mcl{A}\de_N$ in the proof of Proposition \ref{p:approxage} with the description of $\bzd$ in Remark \ref{r:bzd}. 
\end{proof}

\begin{proof}[Proof of Theorem~\ref{t:convage}]

Recall that the space of trajectories is equipped with the $L_{1,loc}$ topology. 
Note that the $M_1$ is stronger so that Proposition \ref{p:convade} also holds for the $L_{1,loc}$ topology. 

Let $T>0$. Let us first show that, for any $\eta > 0$, 

\begin{equation}
\label{neglig0}
\limsup_{N \to +\infty} \P\Ll[ \int_0^T\vert A\de_N(t)-A_N(t)\vert\, \d t\ge \eta\Rr]  \xrightarrow[\delta \to 0]{} 0 .
\end{equation} 

It will be convenient to consider the following process: 
let ${\bar A}\de_N(t)$ be given by $B_N^{-1}\tau_{N,x}$, where $x$ is the last deep trap visited by the process $X(\msft C_N^{-1}(.))$ before time $t$, 
with the convention ${\bar A}\de_N(t)=0$ if no deep trap has been visited by time $t$. 

On the one hand, we observe that both processes $A\de_N$ and ${\bar A}\de_N$ take the same values in the same order (namely the sequence of values of the successive deep traps visited 
by $X$). 
The holding times differ: for $A\de_N$, it is given by the occupation time of a deep trap; for ${\bar A}\de_N$ it is given by the increment of the clock process $C_N$ between visits of deep traps. 
Thus we see that 
\begin{multline*}  
\int_0^T\vert A\de_N(t)-{\bar A}\de_N(t)\vert\, \d t \\
 \le 2 B_N^{-1} (\sup \tau_{N,x}) \ K\de_N(T)\ \td{B}_N^{-1} \int_0^{\msft C_N^{-1}(T)}\tau_{N,X_s} \ \1_{\tau_{N,X_s} \le \delta B_N} \ \d s, 
\end{multline*} 
where the $\sup$ is computed on the set of deep traps that have been visited by $X$ before time $\msft C_N^{-1}(T)$ and $K\de_N(T)$ is the number of times ${\bar A}\de_N(t)$ changes its value before time $\msft C_N^{-1}(T)$. 

(Let $f$ and $g$ be two functions defined on $[0,T]$ with values in a finite subset of $\R_+$. Assume both $f$ and $g$ are piecewise constant, taking the values $x_1,...,x_K$ in that order and with respective 
sojourn times $w_1,...,w_K$ for $f$ and $\zeta_1,...,\zeta_K$ for $g$. Assume $\zeta_i\ge w_i$ for all $i$. Then $f$ and $g$ coincide except possibly on a set whose Lebesgue measure is smaller than 
$(\zeta_1-w_1)+(\zeta_2+\zeta_1-w_2+w_1)+...\le K\sum_i(\zeta_i-w_i)$. Therefore the $L_1$ distance between $f$ and $g$ is bounded by $2\sup_i\vert x_i\vert K\sum_i(\zeta_i-w_i)$.)

Both $B_N^{-1} \sup \tau_{N,x}$ and $K\de_N(T)$ are of order $1$. 
To justify this last claim, we let $L\de_N(T)$ be the number of deep traps visited until time $\msft C_N^{-1}(T)$. 
On the one hand, we already know that $L\de_N(T)$ is of order $1$. On the other hand 
$K\de_N(T) \1_{L\de_N(T)\le k}\1_{\mcl{F}\de_N(k)}=L\de_N(T) \1_{L\de_N(T)\le k}\1_{\mcl{F}\de_N(k)}$ for all $k$. 
As the probability of $\mcl{F}\de_N(k)$ tends to $1$, we deduce that $K\de_N(T)$ is of order $1$ in probability. 

As we already saw in the proof of \eqref{negligstat0}, the expected value of 
$$
\td{B}_N^{-1} \int_0^{T\msft}\tau_{N,X_s} \ \1_{\tau_{N,X_s} \le \delta B_N} \ \d s
$$ tends to $0$. Thus we conclude that 
\begin{equation}
\label{AdenovAden}
\limsup_{N \to +\infty} \P\Ll[ \int_0^T\vert A\de_N(t)-{\bar A}\de_N(t)\vert\, \d t \ge \eta\Rr] \xrightarrow[\delta \to 0]{} 0 . 
\end{equation}
Let us now consider the difference $A_N-{\bar A}\de_N$: these processes do not necessarily take the same values, but they coincide when the process $X$ is visiting a deep trap, i.e. 
$A_N(t)={\bar A}_N(t)$ whenever $A_N(t)\ge \delta $. Therefore 
$$\int_0^T\vert A_N(t)-{\bar A}\de_N(t)\vert\, \d t  \le 2\delta T.$$ 
Combining this with \eqref{AdenovAden}, we  get (\ref{neglig0}). 

To conclude the proof of Theorem~\ref{t:convage}, it then only remains to show that $\mcl{Z}\de$ converges to $\mcl Z$ as $\delta$ tends to $0$. 
This is elementary%: we refer to \cite{fm_aging}
.
\end{proof}

%
%
%
%
%%%%%%%%%%%%%%%%%%%%%%%%%%%%%%%%%%%%%%%%%%%%%%%%%%%%%%%%%%%%%%
%%%%%%%%%%%%%%%%%%%%%%%%%%%%%%%%%%%%%%%%%%%%%%%%%%%%%%%%%%%%%%
%

%{\sc Do we include results on correlation functions (when considered at random times with a density e.g. exponential times)?} 

\end{document}